\documentclass[a4paper,english]{jnsao}
\usepackage[utf8]{inputenc}
\usepackage[T1]{fontenc}
\usepackage{csquotes}
\usepackage{babel}
\usepackage{enumitem}%
\usepackage{mathtools}
\usepackage[indLines=false,noEnd=false]{algpseudocodex}
\usepackage[nameinlink,capitalise]{cleveref}
\usepackage[textsize=footnotesize,color=DarkOrange!40]{todonotes}
\usepackage{float}
\usepackage{booktabs}
\usepackage{siunitx}
\usepackage[noadjust]{cite}

\newcommand{\term}{\emph}

\newcommand{\field}[1]{\mathbb{#1}}
\newcommand{\N}{\mathbb{N}}

\newcommand{\R}{\field{R}}
\newcommand{\extR}{\overline \R}

\newcommand{\norm}[1]{\|#1\|}

\newcommand{\abs}[1]{|#1|}
\newcommand{\adaptabs}[1]{\left|#1\right|}

\newcommand{\inv}[1]{#1^{-1}}
\newcommand{\grad}{\nabla}
\newcommand{\freevar}{\,\boldsymbol\cdot\,}

\newcommand{\Union}\bigcup
\newcommand{\Isect}\bigcap
\newcommand{\union}\cup
\newcommand{\isect}\cap
\newcommand{\bigunion}\bigcup
\newcommand{\bigisect}\bigcap

\newcommand{\defeq}{:=}

\newcommand{\subdiff}{\partial}

\DeclareRobustCommand{\downto}{{{\mathchoice%
            {\rotatebox[origin=c]{-20}{$\to$}}
            {\rotatebox[origin=c]{-20}{$\to$}}
            {\rotatebox[origin=c]{-20}{\scalebox{0.75}{$\to$}}}
            {\rotatebox[origin=c]{-20}{\scalebox{0.6}{$\to$}}}
}}}
\DeclareRobustCommand{\upto}{{{\mathchoice%
            {\rotatebox[origin=c]{20}{$\to$}}
            {\rotatebox[origin=c]{20}{$\to$}}
            {\rotatebox[origin=c]{20}{\scalebox{0.75}{$\to$}}}
            {\rotatebox[origin=c]{20}{\scalebox{0.6}{$\to$}}}
}}}

\DeclareMathOperator*{\argmin}{arg\,min}

\DeclareMathOperator{\closure}{cl}

\DeclareMathOperator{\sign}{sign}

\DeclareMathOperator{\Dom}{dom}

\DeclareMathOperator{\diag}{diag}
\DeclareMathOperator{\reverse}{rev}

\DeclareMathOperator{\Span}{span}

\def \uminusSym{\setbox0=\hbox{$\cup$}\rlap{\hbox
        to\wd0{\hss\raise0.5ex\hbox{$\scriptscriptstyle{-}$}\hss}}\box0}
    
\newcommand{\iprod}[2]{\langle #1,#2\rangle}

\def\llangle{\langle\kern-3pt\langle}
\def\rrangle{\rangle\kern-3pt\rangle}

\def \weaktostarSym{\setbox0=\hbox{$\rightharpoonup$}\rlap{\hbox
        to\wd0{\hss\raise1ex\hbox{$\scriptscriptstyle{*\,}$}\hss}}\box0}
    \def \weaktostar    {\mathrel{\weaktostarSym}}

\def\linear{\mathbb{L}}

\newcommand{\setto}{\rightrightarrows}

\def\extR{\overline \R}

\def\opt#1{\bar #1}

\def\this#1{#1^k}
\def\nexxt#1{#1^{k+1}}

\def\optu{{\opt{u}}}
\def\optx{{\opt{x}}}

\def\nextu{\nexxt{u}}
\def\nextx{\nexxt{x}}
\def\nextz{\nexxt{z}}

\def\thisu{\this{u}}
\def\thisx{\this{x}}
\def\thisz{\this{z}}
\def\thisy{\this{y}}
\def\thisv{\this{v}}

\def\gap{\mathscr{G}}

\def\thisz{\this{z}}
\def\nextz{\nexxt{z}}

\DeclareMathOperator{\prox}{prox}

\DeclareMathOperator{\supp}{supp}

\def\d{\,d}

\def\dualprod#1#2{\langle #1|#2\rangle}

\newcommand{\Meas}{\mathscr{M}}
\newcommand{\ProbMeas}{\mathscr{P}}

\let\phi=\varphi
\let\epsilon=\varepsilon

\def\Id{\mathop{\mathrm{Id}}}

\def\quasitransportsubdiff{\partial\raise0.25ex\hbox{$\scriptstyle\downarrow$}}

\def\Masses{\Meas(\Omega)}
\def\MassesNonNeg{\Meas_+(\Omega)}
\def\DiscreteMeas{\mathscr{Z}}
\def\DiscreteMasses{\DiscreteMeas(\Omega)}
\def\DiscreteMassesNonNeg{\DiscreteMeas_+(\Omega)}

\def\Wave{\mathscr{D}}

\def\TwoPlansSpace{\Meas(\Omega^2)}
\def\TwoMassesNonNeg{\Meas_+(\Omega^2)}

\def\TwoPlans{\bar\Gamma}

\def\ThreePlansSpace{\Meas(\Omega^3)}

\def\ThreePlansCompatAlg{\Lambda_{01}}
\def\ThreePlansNext{\Lambda^+}
\def\Predual{C_0(\Omega)}
\def\DiffPredual{C_0^1(\Omega)}
\def\PlanPredual{C_0(\Omega^2)}

\def\tritrans{\lambda}

\def\fbmetric#1{V_{#1}}
\def\fbmetricBar#1{V}
\def\fbmetricL#1{V^{0,1}}
\def\fbmetricM#1{\check V}

\def\distsub{\omega}

\def\remainder{r}

\def\UWave{\mathscr{U}}

\def\GQ{J}

\def\extF{\bar F}
\def\extG{\bar G}
\def\extXi{\bar\Xi}

\def\MK{T}
\def\UMK{T}

\def\PPredual{P_{\Meas_*}}
\def\breveq{q}
\def\brevez{z}

\def\curvature{\mathscr{K}}

\def\ellF{\ell_F}
\def\ellCurvature{\ell_{\grad v}}

\def\marginalEnergy#1{\mathscr{E}_{#1}}

\def\constCurvature{C_{\curvature}}

\def\Ctrans{C_r}

\renewrobustcmd{\downto}{{{\mathchoice%
            {\rotatebox[origin=c]{-20}{$\to$}}
            {\rotatebox[origin=c]{-20}{$\to$}}
            {\rotatebox[origin=c]{-20}{\scalebox{0.75}{$\to$}}}
            {\rotatebox[origin=c]{-20}{\scalebox{0.6}{$\to$}}}
}}}

\renewrobustcmd{\upto}{{{\mathchoice%
            {\rotatebox[origin=c]{20}{$\to$}}
            {\rotatebox[origin=c]{20}{$\to$}}
            {\rotatebox[origin=c]{20}{\scalebox{0.75}{$\to$}}}
            {\rotatebox[origin=c]{20}{\scalebox{0.6}{$\to$}}}
}}}

\makeatletter
\def\proofstep#1{%
    \par\medskip\noindent%
    \emph{#1:}~%
    \@ifnextchar\par{\@gobble}{}%
}
\makeatother

\makeatletter
\let\@the@H@page\relax
\makeatother

\usepackage{tikz}
\usetikzlibrary{external}
\tikzset{external/system call={lualatex \tikzexternalcheckshellescape -halt-on-error -interaction=batchmode -jobname "\image" "\texsource"}}
\tikzexternaldisable

\usepackage{tikz,pgfplots,pgfplotstable}
\usetikzlibrary{fpu}
\usepgfplotslibrary{colorbrewer,fillbetween}

\newlength{\figheight}
\newlength{\twofigheight}
\newlength{\dataplotheight}
\newlength{\twodkernelplotheight}
\pgfplotsset{
    compat = 1.18,
    scale only axis = false,
    scale mode = stretch to fill,
    every axis/.append style = {
        trim axis left,
        trim axis right,
    },
    plot coordinates/math parser = false,
    every axis label/.append style = {font = \scriptsize},
    every tick label/.append style = {font = \scriptsize},
    x tick label style = {
        /pgf/number format/fixed,
        /pgf/number format/set thousands separator={\,}
    },
    legend style = {
        inner sep = 0pt,
        outer xsep = 5pt,
        outer ysep = 0pt,
        legend cell align = left,
        align = left,
        draw = none,
        fill = none,
        font = \scriptsize
    },
    legend pos = north east,
    width = \linewidth,
    height = \figheight,
    scaled x ticks = false,
    xminorticks = true,
    unbounded coords = jump,
    axis x line* = bottom,
    axis y line* = left,
    yminorticks = true,
    ylabel style = {
         at = {(yticklabel* cs:0.5)},
         anchor = center,
         yshift = 7ex,
    },
    y tick scale label style = {
         at = {(yticklabel* cs:1.0)},
         anchor = center,
         xshift = -3ex,
    },
    notfirst/.style = {skip coords between index={0}{1}},
    log x ticks with fixed point/.style={
        xticklabel={
            \pgfkeys{/pgf/fpu=true}
            \pgfmathparse{exp(\tick)}%
            \pgfmathprintnumber[fixed relative, precision=3]{\pgfmathresult}
            \pgfkeys{/pgf/fpu=false}
        },
    },
    log y ticks with fixed point/.style={
        yticklabel={
            \pgfkeys{/pgf/fpu=true}
            \pgfmathparse{exp(\tick)}%
            \pgfmathprintnumber[fixed relative, precision=3]{\pgfmathresult}
            \pgfkeys{/pgf/fpu=false}
        },
    },
}

\ExplSyntaxOn
\NewDocumentCommand{\addplotnamedtable}{O{}m}{%
    \use:e { \exp_not:n { \addplot table~[#1] } { \exp_not:c { #2 } } }
}
\NewDocumentCommand{\pgfplotstablereadnamed}{mm}{%
    \use:e { \exp_not:n { \pgfplotstableread{#1} } { \exp_not:c { #2 } } }
}
\ExplSyntaxOff

\def\alglist{{FWf, FB, PDPS, sFB}}

\catcode`\_=12
\def\FWfName{FWf}

\def\FBName{$\mu$FB}
\def\sFBName{sFB}
\def\FISTAName{$\mu$FISTA}
\def\PDPSName{$\mu$PDPS}
\def\fPDPSName{fPDPS}
\def\sPDPSName{sPDPS}
\def\radonFBName{radon$^2$FB}
\def\sFBradonName{radon$^2$sFB}
\def\fPDPSradonName{radon$^2$fPDPS}
\def\sPDPSradonName{radon$^2$sPDPS}
\catcode`\_=8

\newif\ifallalgs
\allalgstrue

\pgfkeys{/matrixdim/.initial = 16}

\pgfplotsset{
    solid mark/.style = {
        every mark/.append style = {solid},
    }
}

\pgfplotstableset{
    create on use/inner_iters_mean/.style={
        create col/expr={\thisrow{inner_iters} / \thisrow{this_iters}},
    }
}

\pgfplotscreateplotcyclelist{reconstructions}{
    {color = Set2-B, line width = 1pt, mark = *, mark size = 2pt}, 
    {color = Set2-F, line width = 1pt, mark = square*},
    {color = Set2-E, line width = 1pt, mark = o, mark size = 2.5pt, densely dotted, solid mark},
    {color = Set2-C, line width = 1pt, mark = triangle, mark size = 2.5pt},
    {color = Set2-D, line width = 1pt, mark = o, mark size = 3.5pt, densely dashdotted, solid mark},
    {color = Set2-H, line width = 1pt, mark = pentagon, mark size = 5pt, densely dashed, solid mark},
    {color = Set2-A, line width = 1pt, mark = square, mark size = 2pt, loosely dashed},
}

\pgfplotsset{
    reconstructionplot/.style = {
        reverse legend,
        cycle list name = reconstructions,
        ycomb,
        mark size = 3pt,
        legend columns = 2,
    },
    performanceplot/.style = {
        reverse legend,
        no markers,
        cycle list shift = 1,
        cycle list name = reconstructions,
        log x ticks with fixed point,
    },
    valueplot/.style = {
        ytick = {1e0, 1e-3, 1e-6, 1e-9, 1e-12, 1e-15, 1e-18, 1e-21},
        legend pos = south west,
    },
}

\newcommand{\findminmax}[5]{
    \def\tempa{#3}
    \def\tempb{true}
    \ifx\tempa\tempb
        \pgfplotstablegetelem{0}{#2}\of{#1}%
        \pgfmathparse{\pgfplotsretval}
        \global\let#5\pgfmathresult
        \pgfplotstablegetelem{0}{#2}\of{#1}%
        \pgfmathparse{\pgfplotsretval}
        \global\let#4\pgfmathresult
    \fi
    \pgfplotstablegetrowsof#1
    \pgfmathparse{\pgfplotsretval-1}
    \foreach \i in {0,...,\pgfmathresult}{
        \pgfplotstablegetelem{\i}{#2}\of{#1}%
        \pgfmathparse{max(#5, \pgfplotsretval)}
        \global\let#5\pgfmathresult
        \pgfplotstablegetelem{\i}{#2}\of{#1}%
        \pgfmathparse{min(#4, \pgfplotsretval)}
        \global\let#4\pgfmathresult
    }
}

\providecommand{\biasrecoalg}{sPDPS}

\def\dataplotONEd#1#2{
    \tikzsetnextfilename{#1_data}
    \def\truevalue{true}
    \def\biasenabled{#2}
    \begin{tikzpicture}
        \pgfplotsset{set layers}
        \begin{axis}[%
            scale only axis,
            width = 0.9\linewidth,
            height = \dataplotheight,
            xmin = 0, xmax = 1,
            ymax = 11, ymin = -2,
            legend pos = north west,
            legend style = {xshift=10ex},
            ylabel = {Spike height},
            xlabel = {Coordinate},
            reconstructionplot,
            ]

            \draw [color = BrBG-E, line width = 0.5pt, dotted] (0,0)--(1,0);

            \addplot table [x = x0, y = weight]{\datapath#1/orig.txt};
            \addlegendentry{Ground truth}

            \expandafter\pgfplotsinvokeforeach\alglist {
                \addplot table [x = x0, y = weight]{\datapath#1/\csname##1Filename\endcsname_reco.txt};
                \addlegendentry{\csname##1Name\endcsname}
            }
        \end{axis}
        \ifx\biasenabled\truevalue
        \pgfplotsset{
            legend style = { xshift =  -2.5ex },
        }
        \fi
        \begin{axis}[%
            scale only axis,
            width = 0.9\linewidth,
            height = \dataplotheight,
            xmin = 0, xmax = 1,
            ymax = 5.5, ymin = -1,
            legend pos = north east,
            axis y line* = right,
            axis x line = none,
            ylabel = {Data magnitude},
            ylabel style = {
                at = {(yticklabel cs:0.5)},
                anchor = center,
            },
            ]

            \ifx\biasenabled\truevalue
                \addplot [color = Set2-D, line width = 1.0pt, dashed, const plot]
                    table [header=false]{\datapath#1/bias.txt};
                \addlegendentry{$\hat z$}
                \addplot [color = Set2-G, line width = 1.0pt, dotted, const plot]
                    table [header=false]{\datapath#1/\csname\biasrecoalg Filename\endcsname _z.txt};
                \addlegendentry{\csname\biasrecoalg Name\endcsname~$z$ reco.}
            \fi

            \addplot [color = BrBG-D, line width = 0.5pt, const plot]
                table [header=false]{\datapath#1/b_hat.txt};
            \addlegendentry{$A\hat\mu$}

            \addplot [color = Set2-A, line width = 0.5pt, const plot]
                table [header=false]{\datapath#1/b_noisy.txt};
            \addlegendentry{$A\hat\mu + \text{noise} = b$}
        \end{axis}
    \end{tikzpicture}%
    \pgfplotsset{set layers=none}%
}

\def\dataplotTWOd#1#2{
    \tikzsetnextfilename{#1_data}
    \pgfplotstableread{\datapath#1/b_hat.txt}{\bdata}
    \pgfplotstableread{\datapath#1/b_noisy.txt}{\bnoisy}
    \pgfplotstablecreatecol[copy column from table = {\bnoisy}{[index]2}]{noisy}{\bdata}
    \pgfplotstableclear{\bnoisy}
    \pgfplotstablecreatecol[expr = {\thisrow{noisy} - \thisrowno{2}}]{noise}{\bdata}
    \findminmax{\bdata}{noisy}{true}{\noisyMin}{\noisyMax}
    \findminmax{\bdata}{noise}{true}{\noiseMinX}{\noiseMaxX}
    \pgfmathsetmacro{\noiseMax}{max(abs(\noiseMinX),abs(\noiseMaxX))}
    \pgfmathsetmacro{\noiseMin}{-\noiseMax}
    \pgfplotsset{set layers}
    \def\truevalue{true}%
    \def\biasenabled{#2}%
    \def\mainplotyshift{0ex}%
    \ifx\biasenabled\truevalue%
        \def\mainplotyshift{10ex}
        \pgfplotstableread{\datapath#1/bias.txt}{\bias}
        \pgfplotstableread{\datapath#1/\csname\biasrecoalg Filename\endcsname _z.txt}{\biasreco}
        \findminmax{\bias}{[index]2}{true}{\biasMin}{\biasMax}
        \findminmax{\biasreco}{[index]2}{false}{\biasMin}{\biasMax}
    \fi
    \begin{tikzpicture}
        \ifx\biasenabled\truevalue{%
            \begin{axis}[%
                scale only axis,
                width = 0.9\linewidth,
                height = \dataplotheight,
                xmin = 0, xmax = 1,
                ymin = 0, ymax = 1,
                zmax = 20, zmin = 0,
                legend pos = south east,
                legend style = {
                    name = biaslegend,
                    yshift = -4ex,
                    xshift = -12ex,
                },
                ylabel shift = -2ex,
                xlabel shift = -2ex,
                xlabel = {$x$},
                ylabel = {$y$},
                axis z line = none,
                view/az = 215,
                view/el = 60,
                point meta = explicit,
                point meta rel = per plot,
                mesh/rows = \pgfkeysvalueof{/matrixdim},
                mesh/cols = \pgfkeysvalueof{/matrixdim},
                mesh/ordering = x varies,
                z buffer = sort,
                legend columns = 2,
                reverse legend = true,
                ]

                \addplot3[
                    matrix plot*, 
                    colormap/YlOrBr,
                    line width = 0.5pt
                ] table [
                    header = false,
                    x index = 0,
                    y index = 1,
                    z expr = {0},
                    meta index = 2,
                ]{\biasreco};
                \addlegendentry{\csname\biasrecoalg Name\endcsname}

                \addplot3 [
                    scatter,
                    only marks,
                    mark = cube*,
                    cube/size z = 1pt,
                    /pgfplots/cube/size x = 4pt,
                    /pgfplots/cube/size y = 4pt,
                    shader = flat,
                    colormap/YlOrBr,
                    line width = 0.5pt,
                    scatter/use mapped color = {
                        draw = mapped color,
                        fill = mapped color!80,
                    },
                    legend image post style = {
                        /pgfplots/cube/size z = 1pt,
                        /pgfplots/cube/size x = 7pt,
                        /pgfplots/cube/size y = 7pt,
                        mark options = {
                            fill = YlOrBr-F!80,
                            draw = YlOrBr-F,
                        }
                    },
                    unbounded coords =discard,
                ]  table [
                    header = false,
                    x index = 0,
                    y index = 1,
                    z expr = {0},
                    meta index = 2,
                ]{\bias};
                \addlegendentry{$\hat z$}
            \end{axis}
            \begin{axis}[
                hide axis,
                scale only axis,
                height = 0pt,
                width = 0pt,
                colormap/YlOrBr,
                colorbar horizontal,
                point meta min = \biasMin,
                point meta max = \biasMax,
                colorbar style = {
                    name = biasbar,
                    anchor = north west,
                    at = (biaslegend.north east),
                    yshift = 0.3ex,
                    axis x line* = top,
                    xtick = {\biasMin,\biasMax},
                    font = {\tiny},
                    width = 0.15\linewidth,
                    height = 0.01\linewidth,
                    axis y line = none,
                    tick style = { major tick length = 3pt, semithick, color = black },
                },
                ]
                \addplot [draw=none] coordinates { (0, 0) (1, 1)};
            \end{axis}%
        }\fi
        \begin{axis}[%
            scale only axis,
            width = 0.9\linewidth,
            height = \dataplotheight,
            xmin = 0, xmax = 1,
            ymin = 0, ymax = 1,
            zmax = 20, zmin = 0,
            legend pos = north east,
            legend style = {
                name = datalegend,
                yshift = 0ex,
                xshift = -12ex,
            },
            axis x line = none,
            axis y line = none,
            axis z line = none,
            view/az = 215,
            view/el = 60,
            point meta = explicit,
            point meta rel = per plot,
            mesh/rows = \pgfkeysvalueof{/matrixdim},
            mesh/cols = \pgfkeysvalueof{/matrixdim},
            mesh/ordering = x varies,
            z buffer = sort,
            yshift = \mainplotyshift
            ]

            \addplot3[
                matrix plot*, 
                colormap/BuGn,
                line width = 0.5pt
            ] table [
                header = false,
                x index = 0,
                y index = 1,
                z expr = {0},
                meta = noisy,
            ]{\bdata};
            \addlegendentry{$A\hat\mu + \text{noise} = b$}

            \addplot3 [
                scatter,
                only marks,
                mark = cube*,
                cube/size z = 1pt,
                shader = flat,
                colormap/BrBG,
                line width = 0.5pt,
                scatter/use mapped color = {
                    draw = mapped color,
                    fill = mapped color!80,
                },
                visualization depends on = {
                    10*sqrt(abs(meta)) \as \noisew
                },
                scatter/@pre marker code/.append style = {
                    /pgfplots/cube/size x = \noisew,
                    /pgfplots/cube/size y = \noisew,
                },
                legend image post style = {
                    /pgfplots/cube/size z = 1pt,
                    /pgfplots/cube/size x = 7pt,
                    /pgfplots/cube/size y = 7pt,
                    mark options = {
                        fill = BrBG-F!80,
                        draw = BrBG-F,
                    }
                },
                x filter/.expression = { abs(\thisrow{noise}) < 0.00000001 ? nan : x },
                unbounded coords =discard,
            ]  table [
                header = false,
                x index = 0,
                y index = 1,
                z expr = {0},
                meta = noise,
            ]{\bdata};
            \addlegendentry{noise}
        \end{axis}
        \ifx\biasenabled\truevalue
            \pgfplotsset{
                xticklabel={\empty},
                yticklabel={\empty},
            }
        \else
            \pgfplotsset{
                xlabel={$x$},
                ylabel={$y$},
            }
        \fi
        \begin{axis}[%
            scale only axis,
            width = 0.9\linewidth,
            height = \dataplotheight,%
            xmin = 0, xmax = 1,
            ymin = 0, ymax = 1,
            zmax = 10, zmin = 0,
            legend pos = north west,
            legend style = { yshift = 1ex },
            zlabel = {Spike height},
            axis z line = left,
            view/az = 215,
            view/el = 60,
            reconstructionplot,
            yshift = \mainplotyshift,
            ]

            \addplot3 table [x = x0, y = x1, z = weight,]{\datapath#1/orig.txt};
            \addlegendentry{Ground truth}

            \expandafter\pgfplotsinvokeforeach\alglist {
                \addplot3 table [x = x0, y = x1, z = weight]{\datapath#1/\csname##1Filename\endcsname_reco.txt};
                \addlegendentry{\csname##1Name\endcsname}
            }
        \end{axis}
        \pgfplotstableclear{\bdata}
        \begin{axis}[
            hide axis,
            scale only axis,
            height = 0pt,
            width = 0pt,
            colormap/BuGn,
            colorbar horizontal,
            point meta min = \noisyMin,
            point meta max = \noisyMax,
            colorbar style = {
                name = noisybar,
                at = (datalegend.north east),
                yshift = 0.3ex,
                anchor = north west,
                axis x line* = top,
                xtick = {\noisyMin,\noisyMax},
                font = {\tiny},
                width = 0.15\linewidth,
                height = 0.01\linewidth,
                axis y line = none,
                tick style = { major tick length = 3pt, semithick, color = black },
            }
            ]
            \addplot [draw=none] coordinates { (0, 0) (1, 1)};
        \end{axis}%
        \begin{axis}[
            hide axis,
            scale only axis,
            height = 0pt,
            width = 0pt,
            colormap/BrBG,
            colorbar horizontal,
            point meta min = \noiseMin,
            point meta max = \noiseMax,
            colorbar style = {
                name = noisebar,
                at = (noisybar.south west), 
                yshift = -2ex,
                anchor = north west,
                xtick = {\noiseMin,\noiseMax},
                font = {\tiny},
                width = 0.15\linewidth,
                height = 0.01\linewidth,
                axis y line = none,
                tick style = { major tick length = 3pt, semithick, color = black },
                anchor = south west,
            }
            ]
            \addplot [draw=none] coordinates { (0, 0) (1, 1)};
        \end{axis}%
    \end{tikzpicture}%
    \pgfplotsset{set layers=none}%
}

\catcode`\_=12

\def\performanceplots#1{%
    \expandafter\pgfplotsinvokeforeach\alglist{
        \pgfplotstablereadnamed{\datapath#1/\csname##1Filename\endcsname_log.txt}{##1Performance}
    }
    \tikzsetnextfilename{#1_performance}
    \begin{tikzpicture}[baseline]
        \begin{axis}[%
            xmode = log,
            ymode = log,
            xmin = 1,
            width = 0.5\linewidth,
            height = \figheight,
            axis x line* = bottom,
            axis y line* = left,
            ylabel = {Relative value error},
            xlabel = {Iteration},
            performanceplot,
            valueplot,
            ]

            \expandafter\pgfplotsinvokeforeach\alglist {
                \addplotnamedtable [x = iter, y = relative_value]{##1Performance};
                \addlegendentry{\csname##1Name\endcsname}
            }
        \end{axis}
    \end{tikzpicture}%
    \quad%
    \tikzsetnextfilename{#1_cputime}%
    \begin{tikzpicture}[baseline]
        \begin{axis}[%
            width = 0.5\linewidth,
            height = \figheight,
            xmode = log,
            ymode = log,
            axis x line* = bottom,
            axis y line* = left,
            ylabel = {Relative value error},
            xlabel = {CPU time (seconds)},
            performanceplot,
            valueplot
            ]

            \expandafter\pgfplotsinvokeforeach\alglist {
                \addplotnamedtable [x = cpu_time, notfirst, y = relative_value]{##1Performance};
                \addlegendentry{\csname##1Name\endcsname}
            }
        \end{axis}
    \end{tikzpicture}%
}

\def\evolutionplot#1{%
    \begin{minipage}[b]{0.5\linewidth}%
    \tikzsetnextfilename{#1_evolution}
    \begin{tikzpicture}%
        \begin{axis}[%
            width = \linewidth,
            height = \twofigheight,
            name = spike count,
            xmode = log,
            xmin = 1,
            ymin = 0,
            ylabel = {Spike count},
            const plot,  
            performanceplot,
            x tick label style = { draw = none },
            ]

            \expandafter\pgfplotsinvokeforeach\alglist {
                \addplotnamedtable [x = iter, y = n_spikes]{##1Performance};
            }
        \end{axis}
    \end{tikzpicture}%
    \\%
    \tikzsetnextfilename{#1_spikes}%
    \begin{tikzpicture}
        \begin{axis}[%
            at = {(spike count.south west)},
            anchor = north west,
            width = \linewidth,
            height = \twofigheight,
            xmode = log,
            xmin = 1,
            ymin = {\ifallalgs 1\else 0\fi},
            ylabel = {Inner iters.},
            xlabel = {Iteration},
            ymode = {\ifallalgs log\else normal\fi},
            log y ticks with fixed point,
            const plot,  
            performanceplot,
            ]

            \expandafter\pgfplotsinvokeforeach\alglist {
                \addplotnamedtable [x = iter, y = inner_iters_mean]{##1Performance};
            }
        \end{axis}
    \end{tikzpicture}%
    \end{minipage}%
}
\def\kernelplotONEd#1{%
    \tikzsetnextfilename{#1_kernel}%
    \tikzpicturedependsonfile{\datapath#1/spread.txt}%
    \tikzpicturedependsonfile{\datapath#1/sensor.txt}%
    \tikzpicturedependsonfile{\datapath#1/base_sensor.txt}%
    \tikzpicturedependsonfile{\datapath#1/kernel.txt}%
    \begin{tikzpicture}
        \begin{axis}[%
            width = 0.5\linewidth,
            height = \figheight,
            xmin = -0.25, xmax = 0.25,
            axis x line* = bottom,
            axis y line* = left,
            ylabel = {Magnitude},
            xlabel = {Coordinate},
            legend pos = north east,
            legend style = { xshift = 2ex, outer xsep = 0pt },
            tick label style = { text width = 3ex, align = right },
            ]

            \addplot [color = Set2-B, line width = 0.5pt, const plot]
                table [header=true]{\datapath#1/sensor.txt};
            \addlegendentry{Sensor $\theta$}

            \addplot [color = Set2-C, line width = 1pt]
                table [header=true]{\datapath#1/base_sensor.txt};
            \addlegendentry{$\psi * \theta \sim A$}

            \addplot [color = Set2-D, line width = 1pt, dashed]
                table [header=true]{\datapath#1/kernel.txt};
            \addlegendentry{Kernel $\rho$ of $\Wave$ }

            \addplot [color = Set2-A, line width = 0.5pt, dashdotted]
                table [header=true]{\datapath#1/spread.txt};
            \addlegendentry{Spread $\psi$}

        \end{axis}
    \end{tikzpicture}%
}

\def\kernelplotTWOd#1{%
    \tikzsetnextfilename{#1_kernel}%
    \tikzpicturedependsonfile{\datapath#1/spread.txt}%
    \tikzpicturedependsonfile{\datapath#1/sensor.txt}%
    \tikzpicturedependsonfile{\datapath#1/base_sensor.txt}%
    \tikzpicturedependsonfile{\datapath#1/kernel.txt}%
    \begin{tikzpicture}
        \begin{axis}[%
            width = 0.5\linewidth,
            height = \twodkernelplotheight,
            xmin = -0.4, xmax = 0.4,
            ymin = -0.4, ymax = 0.4,
            axis z line = left,
            view/az = 215,
            view/el = 60,
            ylabel = {$y$},
            xlabel = {$x$},
            zlabel = {Magnitude},
            mesh/rows = 40,
            mesh/cols = 40,
            mesh/ordering = x varies,
            point meta rel = per plot,
            legend columns = 4,
            legend style = {
                at = {(0.5, 0.97)},
                anchor = north,
                column sep = 1ex,
                font = \tiny,
            },
            tick label style = { text width = 3ex, align = right },
            legend image post style = { scale = 0.7 },
            ]

            \addplot3 [
                surf, shader = flat, colormap/BuGn,
                restrict x to domain = 0:0.4,
                restrict y to domain = -0.4:0,
            ] table [
                x expr = {\thisrow{x} + 0.2},
                y expr = {\thisrow{y} - 0.2},
                z index = 2,
            ]{\datapath#1/spread.txt};
            \addlegendentry{Spread $\psi$}

            \addplot3 [
                surf, shader = flat, colormap/PuBu,
                restrict x to domain = -0.4:0,
                restrict y to domain = -0.4:0,
            ] table [
                x expr = {\thisrow{x} - 0.2},
                y expr = {\thisrow{y} - 0.2},
                z index = 2,
            ]{\datapath#1/base_sensor.txt};
            \addlegendentry{$\psi * \theta \sim A$}

            \addplot3 [
                surf, shader = flat, colormap/YlOrBr,
                restrict x to domain = 0:0.4,
                restrict y to domain = 0:0.4,
            ] table [
                x expr = {\thisrow{x} + 0.2},
                y expr = {\thisrow{y} + 0.2},
                z index = 2,
            ]{\datapath#1/sensor.txt};
            \addlegendentry{Sensor $\theta$}

            \addplot3 [
                surf, shader = flat, colormap/PuRd,
                restrict x to domain = -0.4:0,
                restrict y to domain = 0:0.4,
            ] table [
                x expr = {\thisrow{x} - 0.2},
                y expr = {\thisrow{y} + 0.2},
                z index = 2,
            ]{\datapath#1/kernel.txt};
            \addlegendentry{Kernel $\rho$ of $\Wave$ }

        \end{axis}
    \end{tikzpicture}%
}

\catcode`\_=8

\newif\ifplots
\plotstrue

\def\plotsONEd#1{%
    \ifplots%
    \tikzexternalenable%
    \dataplotONEd{#1}{false}%
    \tikzexternaldisable%
    \\%
    \performanceplots{#1}%
    \\%
    \evolutionplot{#1}%
    \tikzexternalenable%
    \kernelplotONEd{#1}%
    \tikzexternaldisable%
    \fi%
}

\def\plotsONEdTV#1{%
    \ifplots%
    \tikzexternalenable%
    \dataplotONEd{#1}{true}%
    \tikzexternaldisable%
    \\%
    \performanceplots{#1}%
    \\%
    \evolutionplot{#1}%
    \tikzexternalenable%
    \kernelplotONEd{#1}%
    \tikzexternaldisable%
    \fi%
}

\def\plotsTWOd#1{%
    \ifplots
    \tikzexternalenable%
    \dataplotTWOd{#1}{false}%
    \\%
    \tikzexternaldisable
    \performanceplots{#1}%
    \\%
    \evolutionplot{#1}%
    \tikzexternalenable
    \kernelplotTWOd{#1}%
    \tikzexternaldisable%
    \fi
}

\def\plotsTWOdTV#1{%
    \ifplots
    \tikzexternalenable%
    \dataplotTWOd{#1}{true}%
    \\%
    \tikzexternaldisable
    \performanceplots{#1}%
    \\%
    \evolutionplot{#1}%
    \tikzexternalenable
    \kernelplotTWOd{#1}%
    \tikzexternaldisable%
    \fi
}

\usepackage{ifthen}
\tikzifexternalizing{
}{%
    \usepackage{jsonparse}
    \JSONParseSet{ externalize = true, externalize prefix = build/}
}

\colorlet{iitodocolor}{MediumTurquoise!50!white}

\usepackage{mdframed}

\mdfdefinestyle{examplebg}{
    backgroundcolor=black!6,%
    hidealllines=true,%
    innertopmargin=-.5em,%
    innerbottommargin=.5em,%
    innerleftmargin=.5em,%
    innerrightmargin=.5em,%
    footnoteinside=false,
    skipabove=2pt,
    skipbelow=2pt,
}

\surroundwithmdframed[style=examplebg]{example}
\newmdenv[
    backgroundcolor=structure!6,%
    hidealllines=true,%
    innertopmargin=.5em,%
    innerbottommargin=.5em,%
    innerleftmargin=.5em,%
    innerrightmargin=.5em,%
    footnoteinside=false,
    skipabove=2pt,
    skipbelow=2pt,
]{rawalg}

\mdfdefinestyle{warningbg}{
    backgroundcolor=red!7,%
    hidealllines=true,%
    innertopmargin=.3em,%
    innerbottommargin=.7em,%
    innerleftmargin=.7em,%
    innerrightmargin=.7em,%
    skipabove=0pt,
    skipbelow=0pt,
}

\mdfdefinestyle{importantbg}{
    backgroundcolor=green!7,%
    hidealllines=true,%
    innertopmargin=.3em,%
    innerbottommargin=.7em,%
    innerleftmargin=.7em,%
    innerrightmargin=.7em,%
}

\manuscriptcopyright{}
\manuscriptlicense{}%
\manuscripteprinttype{arxiv}
\manuscripteprint{2502.12417}

\floatstyle{ruled}
\floatname{algorithm}{Algorithm}
\newfloat{algorithm}{tbp!}{loa}
\numberwithin{algorithm}{section}
\makeatletter
\AddToHook{env/algorithmic/before}{\def\@currentcounter{ALG@line}}
\crefalias{ALG@line}{line}
\makeatother

\author{
    Tuomo Valkonen\thanks{%
        MODEMAT Research Center in Mathematical Modeling and Optimization, Quito, Ecuador;
        Department of Mathematics and Statistics, University of Helsinki, Finland;
        \emph{and}
        Department of Mathematics, Escuela Politécnica Nacional, Quito, Ecuador.
        \email{tuomo.valkonen@iki.fi}, \orcid{0000-0001-6683-3572}}
    }
\shortauthor{T.~Valkonen}
\title{Point source localisation with unbalanced optimal transport}

\date{2025-02-17 (revised 2026-07-29)}

\acknowledgements{%
    This research has been supported by the Academy of Finland grant 314701, and by the Jane and Aatos Erkko Foundation.
}

\begin{document}

\maketitle

\begin{abstract}
    Replacing the quadratic proximal penalty familiar from Hilbert spaces by an unbalanced optimal transport distance, we develop forward-backward type optimisation methods in spaces of Radon measures.
    We avoid the actual computation of the optimal transport distances through the use of transport three-plans and the rough concept of transport subdifferentials.
    The resulting algorithm has a step similar to the sliding heuristics previously introduced for conditional gradient methods, however, now non-heuristically derived from the geometry of the space.
    We demonstrate the improved numerical performance of the approach.
\end{abstract}

\section{Introduction}

We continue the quest started in \cite{tuomov-pointsource} to understand the challenges Banach space geometries pose to the realisation of forward-backward type optimisation methods and the proximal step. We do this by focussing on the point source localisation problem \cite{candes2014towards,lindberg2012mathematical}
\begin{equation}
    \label{eq:intro:problem}
    \min_{\mu \in \Masses}~ \frac{1}{2}\norm{A\mu-b}_Y^2 + \alpha \norm{\mu}_\Meas + \delta_{\ge 0}(\mu),
\end{equation}
where $A \in \linear(\Masses; Y)$ is a bounded, linear forward-operator from the space of Radon measures $\Masses$ on $\Omega \subset \R^n$ to a Hilbert space $Y$ of measurements $b$. The parameter $\alpha > 0$ controls the sparsity of the solution via the Radon norm regularisation term, while the final positivity constraint slightly simplifies the technical details by concentrating on sources only, avoiding sinks.

Our hunch when starting the work on \cite{tuomov-pointsource} was that the weak-$*$ topology would be the natural topology for working with measures and, therefore, any forward-backward method should try to replace the Hilbert-space quadratic penalisation in the forward-backward method
\begin{equation}
    \label{eq:intro:hilbert-alg}
    \nextx \defeq \argmin_{x}~ F(\thisx) + \iprod{\grad F(\thisx)}{x-\thisx} + G(x) + \frac{1}{2\tau}\norm{x-\thisx}^2
\end{equation}
for the minimisation of $F+G$, by a metrisation of weak-$*$ convergence.
Wasserstein distances provide such metrisations \cite{santambrogio2015optimal}, however, we were in the work leading up to \cite{tuomov-pointsource} unable to yet make them work.
Instead, we introduced “particle-to-wave” operators $\Wave \in \linear(\Masses; \Predual)$ from Radon measures to the predual space of continuous functions that vanish at infinity, and defined \mbox{(semi-)}norms and \mbox{(semi-)}inner products with the help of these operators. That way, we could make optimisation in arbitrary normed spaces appear like optimisation in Hilbert spaces. Practically, we took $\Wave$ as a convolution operator, similarly to the construction of Maximum Mean Discrepancy (MMD) norms for optimal transport \cite{sejourne2023unbalanced}.

Now, in this work we finally introduce a forward-backward method based on (conventional, non-MMD) unbalanced optimal transport of measures \cite{schmitzer2019framework,schmitzer2019unbalanced,kondratyev2016new,chizat2018unbalanced,chizat2016interpolating,liero2017optimal,benamou2003numerical,sejourne2023unbalanced}.
Practically, such an optimal transport based algorithm implements the sliding heuristics introduced in \cite{denoyelle2019sliding} for conditional gradient methods \cite{brediespikkarainen2013inverse,denoyelle2019sliding,duval2017sparse,walter2019linear,blank2017extension,bredies2021linear} in a less heuristic way, tied to the geometry of the space.
Our specific flavour of unbalanced optimal transport bears a resemblance to the flat metric of geometric measure theory, however, is based on squared terms for practical reasons.
Our primary variant, which we introduce in \cref{sec:transport-new}, still includes $\Wave$ as the marginal cost between the transport and the source and traget measures.
We will, however, also consider the Radon norm $\norm{\freevar}_\Meas$ as a marginal cost, with weaker convergence guarantees.
Indeed, we will see in \cref{sec:numerical} that numerically the $\Wave$-norm is more effective.
However, it is much easier to prove that the differential of the data term $F$ is Lipschitz with respect to the Radon norm. This is the term $\frac{1}{2}\norm{A\mu-b}^2$ in \eqref{eq:intro:problem}.

\emph{We never actually compute the optimal transport plans $\gamma \in \TwoPlansSpace$ that realise the optimal transport distances.}
Instead, the forward-backward algorithm that we develop in \cref{sec:fb,sec:sub}, is based on the “weaving” of suboptimal \term{three-plans} $\tritrans \in \ThreePlansSpace$ between the current iterate $\this\mu$, the next iterate $\nexxt\mu$, and a comparison point $\opt\mu$, typically a solution to \eqref{eq:intro:problem}.
We also employ \emph{transport subdifferential} ideas adapted from \cite{ambrosio2008gradientflows} and updated to the unbalanced setting.
We briefly review these ideas and then develop relevant subdifferential-type and transport smoothness estimates in \cref{sec:subdiff}.
Before illustrating numerical performance in \cref{sec:numerical}, we also present an extension of our method to product spaces in \cref{sec:extensions}. This allows treating problems with auxiliary variables as well as dual variables in a primal-dual splitting method.
For the weaker subdifferential convergence results of \cref{sec:fb}, we do not assume the convexity of the data term $F$. For the $O(1/N)$ function value convergence results of \cref{sec:sub} we make that assumption.
The algorithm studied in \cref{sec:fb,sec:sub} depends on bounding a remainder term.
There are many ways to do this. We treat some options in \cref{sec:controls}.

Besides the aforementioned conditional gradient methods and \cite{tuomov-pointsource}, other algorithms for \eqref{eq:intro:problem} include \cite{casas2012approximation,casas2013parabolic,flinth2020linear}.
In \cite{chizat2021sparse,chizat2018global} Wasserstein gradient flows are considered for probability measures modelling a discrete set of possible source locations.
A gradient flow of probability measures based on a balanced transport distance very similar to our unbalanced transport distance, has been recently introduced in \cite{gladin2024interactionforcetransportgradientflows}.
Although their theoretical time-discretised numerical method is a cyclic proximal point method, their practically employed numerical method, based on a fixed number of interacting particles, bears parallels to ours, but lacks its finer details, and is presented without convergence proof.
Optimal transport regularised dynamic point source location problems---not algorithms based on optimal transport---are considered in \cite{bredies2020generalized}.
As optimal transport plans can be related to geodesics with respect to a Wasserstein distance, our overall approach can also be related to optimisation methods on manifolds \cite{bergmann2019fenchel,udriste1994convex} and general spaces \cite{luke2021convergence}.
Our work could also be combined with \cite{tuomov2024tracking} for single-loop PDE-constrained and bilevel optimisation with measures.

\subsection*{Notation and elementary results}

We denote the extended reals by $\extR \defeq [-\infty,\infty]$.
We use the notation $\abs{\freevar}_p$ (and $\abs{\freevar}$ for $p=2$) to distinguish $p$-norms in $\R^n$ from norms on other, usually infinite-dimensional spaces.

The space of finite Radon measures on a locally compact Borel measurable set $\Omega \subset \R^n$ we denote by $\Masses$, while $\MassesNonNeg$ stands for the non-negative Radon measures.
We write $\norm{\freevar}_{\Masses}$ or, for short, $\norm{\freevar}_\Meas$ for the Radon norm.
The subset $\ProbMeas(\Omega) \subset \Masses$ consists of probability measures, and the subspace $\DiscreteMasses \subset \Masses$ of \term{discrete measures} $\mu=\sum_{k=1}^n \alpha_k \delta_{x_k}$ for any $n \in \N$, where the weights $\alpha_k \in \R$, and locations $x_k \in \Omega$.
Here, $\delta_x$ is the Dirac measure with mass one at the point $x \in \Omega$.

For a $\mu, \nu \in \Masses$, if $\mu$ is absolutely continuous with respect to $\nu$, denoted $\mu \ll \nu$, we write $\defeq d\mu/d\nu$ for the Radon--Nikodym derivative of $\mu$ respect to $\nu$.
In particular, $\sign \mu \defeq d\mu/d\abs{\mu}$.
We write $\mu^\pm \in \MassesNonNeg$ for the positive and negative parts of $\mu=\mu^+-\mu^-$.
The support is defined as $\supp\mu \defeq \supp\mu^+ \union \supp \mu^-$, where, for $\nu \in \MassesNonNeg$, $\supp \nu \defeq \closure \{ A \text{ measureable} \mid \nu(A) > 0\}$.
If $\mu=\sum_{i=1}^n \beta_i \delta_{x_i}$, we have $\supp\mu=\{x_i \mid \beta_i \ne 0,\, i=1,\ldots,n\}$.

For a $\mu$-integrable functions $\rho$, we write $[\rho * \mu](x) \defeq \int \rho(x-y) d\mu(y)$ for the convolution.

With $F:X \to R$  Fréchet differentiable on a normed space $X$, we write $F'(x) \in X^*$ for the Fréchet derivative at $x \in X$.
Here $X^*$ is the dual space to $X$.
We call $F$ \term{pre-differentiable} if $F'(x) \in X_*$ for $X_*$ a designated \term{predual space} of $X$, satisfying $X=(X_*)^*$.
We have $(X_*)^{**}=X^*$, so $X_*$ canonically injects into $X^*$.
For a convex function $F: X \to \extR$, we write $\subdiff F: X \setto X_*$ for the (set-valued) pre-subdifferential map, defined as $\subdiff F(x) = \{ x_*  \in X_* \mid F(\tilde x) - F(x) \ge \dualprod{x_*}{\tilde x - x} \text{ for all } \tilde x \in X\}$.
We write $\delta_C: X \to \extR$ for the $\{0,\infty\}$-valued indicator function of a set $C \subset X$.

We write $\iprod{x}{x'}$ for the inner product between two elements $x$ and $x'$ of a Hilbert space $X$, and $\dualprod{x^*}{x} \defeq \dualprod{x^*}{x}_{X^*,X} \defeq x^*(x)$ for the duality pairing between elements of a normed space $X$ and its dual or predual $X^*$.
The space $\linear(X; Y)$ stands for bounded linear operators between two vector spaces $X$ and $Y$.
The identity operator is $\Id \in \linear(X; X)$.
With $Y$ a Hilbert space for simplicity, we call $A \in \linear(X; Y)$ \term{pre-adjointable} if there exists a \term{pre-adjoint} $A_*: \linear(X_*; Y)$ whose mixed Hilbert–Banach adjoint satisfies $(A_*)^* = A$.
In other words $\dualprod{x}{A_*z} = \iprod{Ax}{z}$ for all $z \in Y$ and $x \in X$.

A predual of $\Masses$ is the space $C_c(\Omega)$ of continuous functions with compact support.
We write $C_0(\Omega)$ for continuous functions on $\Omega$ that vanish at infinity, which is also a predual space of $\Masses$ \cite[Theorem 1.200]{fonseca2007mmc}.
Moreovr, the space
\[
    \DiffPredual \defeq \{ f \in C_0(\Omega) \mid f \text{ is Fréchet differentiable with } f' \in C_0(\Omega; \R^n)\}
\]
is a dense subspace of $C_0(\Omega)$, as seen by taking convolutions with a differentiable bump function.

For a map $f: X \to Y$ and measure $\mu$ on $X$, we define the \term{push-forward} measure $f_\# \mu$ by
\[
    [f_\# \mu](A) \defeq \mu(\{x \in X \mid f(x) \in A\})
    \quad (A \subset Y \text{ measurable} ).
\]
We will typically take $f$ as a projection $\pi^{i_1,\ldots,i_k}(x_1,\ldots,x_n) \defeq (x_{i_1}, \ldots, x_{i_k})$, the diagonal constructor $\diag(x) \defeq (x,x)$, or the reversal $\reverse(x, y) = (y, x)$. Then, for $\tritrans \in \Meas(\Omega^n)$,
\[
    \pi_\#^{i_1,\ldots,i_k}\tritrans(A)=\tritrans(\{x \in \Omega^n \mid (x_{i_1}, \ldots, x_{i_k}) \in A\})
    \quad (A \subset \Omega^k \text{ measurable}).
\]
whereas, for $\mu \in \Masses$,
\[
    \diag_\#\mu(A) = \mu(\{x \in \Omega \mid (x, x) \in A\})
    \quad (A \subset \Omega^2 \text{ measurable}).
\]

Finally, for any differentiable function $w: \Omega \to \extR$ on a set $\Omega \subset X$, we define the Bregman divergence generated by $w$ as
\[
    B_w(x, y) \defeq w(y) - w(x) + \dualprod{w'(x)}{y-x}.
\]
If $w$ is convex and non-differentiable, we indicate the choice of  $x^* \in \subdiff w(x)$, explicitly in $B_w^v(x, y) \defeq w(y) - w(x) + \dualprod{x^*}{y-x}$. In this case, $B_v^w \ge 0$.

\section{An approach to unbalanced transport}
\label{sec:transport-new}

We now introduce our approach to unbalanced optimal transport.
We start by recalling basic definitions of balanced optimal transport in \cref{sec:transport-new:basic}.
We then define our variant of unbalanced optimal transport in \cref{sec:transport-new:definitions}, and study metrisation of weak-$*$ convergence in \cref{sec:transport-new:weakstar}.

\subsection{Balanced optimal transport theory}
\label{sec:transport-new:basic}

Balanced optimal transport theory treats the cost of transporting the total mass of one probability measure to another \cite{santambrogio2015optimal}.
Let $c: \Omega^2 \to [0, \infty)$ be some symmetric \term{transport cost} that models the cost of transporting a unit mass between two points of $\Omega \subset \R^n$.
For example, $c=c_p$ for some $p \ge 1$ and
\[
    c_p(x, y) \defeq \frac{1}{p}\abs{x-y}_p^p.
\]
Given $\gamma \in \TwoMassesNonNeg$, we then define the Monge–Kantorovich \term{transport cost} between $\mu, \nu \in \MassesNonNeg$ as
\[
	\MK_c(\mu, \nu) \defeq \inf_{\gamma \in \Gamma(\mu, \nu)} \int_{\Omega^2} c(x, y) \d\gamma(x, y)
\]
over the set of \term{transport plans}
$
    \Gamma(\mu, \nu) \defeq \{ \gamma \in \TwoMassesNonNeg \mid \pi^0_\#\gamma=\mu,\, \pi^1_\#\gamma=\nu\}.
$
These plans indicate the direct paths from $x$ to $y$, along which mass is transported, and the corresponding amount. They are restricted to transport all mass of $\mu$ exactly to $\nu$. Indeed, $\Gamma(\mu, \nu)=\emptyset$ if $\norm{\mu}_\Meas \ne \norm{\nu}_\Meas$.

The $p$-Wasserstein distance is now defined through
\[
    W_p(\mu, \nu) \defeq \MK_{c_p}(\mu,\nu)^{1/p}.
\]
Restricted to probability measures, this distance metricises the topology of weak-$*$ convergence.
For more details, we refer to \cite{santambrogio2015optimal}.
If we want to use $W_p$ to measure the distance between two iterates $\this\mu$ and $\nexxt\mu$ in an optimisation algorithm, such as replacing the quadratic penalty in \eqref{eq:intro:hilbert-alg}, we are limited to working with probability measures.
Therefore, we need a theory of unbalanced optimal transport that does not impose this restriction of equal masses. Such theories have been introduced in, e.g., \cite{schmitzer2019framework,schmitzer2019unbalanced,kondratyev2016new,chizat2018unbalanced,chizat2016interpolating,liero2017optimal,benamou2003numerical,sejourne2023unbalanced}.

\subsection{Basic definitions}
\label{sec:transport-new:definitions}

Let $c$ be as above, and let $E: \Masses^2 \to [0, \infty)$ be some (possibly nonsymmetric) but convex and lower semicontinuous \term{marginal cost}.
For $\mu_0,\mu_1 \in \Masses$, also let $\TwoPlans(\mu_0,\mu_1) \subset \TwoPlansSpace$ be a closed and convex set, such that $0 \in \TwoPlans(\mu_0,\mu_1)$.
We call it the set of admissible (unbalanced) transport plans between these two measures.
We then define the \term{unbalanced transport cost}\footnote{We do not call this a “distance”, since we will concentrate on $c(x,y)=c_2(x,y)=\frac{1}{2}\abs{x-y}^2$ without taking the corresponding square root of the integral.}
between $\mu_0$ and $\mu_1$ as
\begin{gather*}
    \UMK_{c,E,\TwoPlans}(\mu_0, \mu_1) \defeq \inf_{\gamma \in \TwoPlans(\mu_0,\mu_1)} V_{c,E}(\mu_0,\mu_1; \gamma),
\shortintertext{where}
    V_{c,E}(\mu_0,\mu_1; \gamma)
    \defeq
    \int_{\Omega^2} c(x, y) \d \abs{\gamma}(x, y)
    +
    E(\mu_0 - \pi_\#^0 \gamma, \mu_1 - \pi_\#^1 \gamma).
\end{gather*}

\paragraph{The marginal cost}

We consider the squared Radon norm
\begin{equation}
    \label{eq:unbalanced:emeas}
    E_\Meas(\nu_0, \nu_1) \defeq \frac{1}{2}\norm{\nu_1-\nu_0}_\Meas^2
\end{equation}
as well as Bregman divergences
\begin{equation}
    \label{eq:unbalanced:ep}
    E_J(\nu_0, \nu_1) = B_J^\omega(\nu_0, \nu_1)
    \defeq
    \begin{cases}
        J(\nu_1)-J(\nu_0) - \dualprod{\omega}{\nu_1-\nu_0}, & \omega \in \subdiff J(\nu_0), \\
        \infty, & \omega \not\in \subdiff J(\nu_0),
    \end{cases}
\end{equation}
for some convex, proper, and weakly-$*$ lower semicontinuous $J: \Meas(\Omega) \to \extR$.
If $J$ is differentiable, we simply write $B_J(\nu_0, \nu_1) \defeq B_J^{J'(\nu_0)}(\nu_0, \nu_1)$.

In particular, following \cite{tuomov-pointsource}, let $\Wave \in \linear(\Masses; \Predual)$ be self-adjoint and positive semi-definite, i.e.,
\[
    \dualprod{\Wave x}{y}_{\Predual,\Masses} = \dualprod{x}{\Wave y}_{\Predual,\Masses}
    \quad\text{and}\quad
    \dualprod{\Wave x}{x}_{\Predual,\Masses} \ge 0
    \quad\text{for all}\quad
    x, y \in \Masses.
\]
Then the semi-inner product and semi-norm
\[
    \iprod{x}{x}_{\Wave} \defeq \dualprod{\Wave x}{x}_{\Predual,\Masses}
    \quad\text{and}\quad
    \norm{x}_{\Wave} \defeq \sqrt{\iprod{x}{x}_{\Wave}}
\]
are well-defined \cite{tuomov-pointsource}.
We may, for example, take $\Wave\mu = \rho * \mu$ for a symmetric and self-adjoint convolution kernel $\rho$.
For $J_\Wave \defeq \frac{1}{2}\norm{\freevar}_\Wave^2$, we then have
\[
    E_\Wave(\nu_0, \nu_1) \defeq B_{J_\Wave}(\nu_0, \nu_1) = \frac{1}{2}\norm{\nu_1-\nu_0}_\Wave^2.
\]

For both $\star=\Meas,\Wave$, we have
\begin{equation}
    \label{eq:unbalanced:v-star}
    V_{c,E_\star} (\mu_0, \mu_1; \gamma)
    =
    \int_{\Omega^2} c(x, y) \d \abs{\gamma}(x, y)
    +
    \frac{1}{2}\norm{\mu_1 - \mu_0 - (\pi_\#^1-\pi_\#^0)\gamma}_\star^2.
\end{equation}
These are invariant with respect to diagonal additions ($\diag_\#\mu$ for a $\mu \in \Masses$) to the transport $\gamma$.

\begin{remark}[Flat metrics]
    The marginal term in \eqref{eq:unbalanced:v-star} makes $\UMK_{c,E,\TwoPlans}$ reminiscent of the flat metric in the theory of currents \cite{federer1969gmt,morgan2000gmt}.
    However, it is squared in our approach.
\end{remark}

\paragraph{The set of admissible plans}

For $c(x, y)=\norm{x-y}$, we can take $\TwoPlans(\mu_0,\mu_1) = \TwoPlansSpace$.
However, if, for example, $c(x,y)=c_2(x,y)=\norm{x-y}_2^2/2$, as we will be using from \cref{sec:fb} onwards, $\TwoPlans(\mu_0,\mu_1)$ needs to impose further restrictions for $\UMK_{c,E,\TwoPlans}(\mu_0, \mu_1)$ to be meaningful:

\begin{example}In $\Omega=\R$, let $\mu_0=\delta_0$ and $\mu_1=\delta_1$.
    Take $\gamma^n=\sum_{k=0}^{n-1} \delta_{(k/n, (k+1)/n)} \in \Meas(\R^2)$.
    Then $(\pi_\#^1-\pi_\#^0)\gamma^n = \delta_1-\delta_0$.
    Thus, for both $\star=\Meas,\Wave$, we have $E_\star(\mu_0-\pi_\#^0\gamma^n, \mu_1-\pi_\#^1\gamma^n)=0$.
    On the other hand
    \[
        \int_{\Omega^2} \abs{x-y} \d\gamma^n(x, y) = 1,
        \quad\text{but}\quad
        \int_{\Omega^2} \abs{x-y}^2 \d\gamma^n(x, y) = \frac{1}{n}.
    \]
    Therefore, $\UMK_{c_1,E,\TwoPlans}(\delta_0, \delta_1)=1$, but $\UMK_{c_2,E,\TwoPlans}(\delta_0, \delta_1)=0$.
\end{example}

To avoid this issue, we can take, for example
\[
    \TwoPlans(\mu_0, \mu_1) \defeq \{
        \gamma \in \Meas(\Omega^2)
        \mid
        \pi_\#^0 \gamma \ll \abs{\mu_0},\, \pi_\#^1 \gamma \ll \abs{\mu_1}
        \text{ and }
        \norm{\gamma} \le \norm{\mu_0} + \norm{\mu_1}
    \}.
\]
The norm bound is not required for the previous example, however, when it is included, the next lemma establishes the existence of a minimising transport plan.
Such existence is, however, not critical to us, as we will mainly be working with $V_{c,E}$ and explicit, possibly suboptimal, transport plans.

\begin{lemma}\label{lemma:unbalanced:existence}
    Let $\mu_0, \mu_1 \in \Masses$.
    Suppose $0 \le  c \in \PlanPredual$; that $E: \TwoPlansSpace \to [0, \infty)$ is convex and lower semicontinuous; and that $0 \in \TwoPlans(\mu_0, \mu_1) \subset \TwoPlansSpace$ is convex, closed, and bounded.
    Then there exists $\gamma \in \TwoPlans(\mu_0,\mu_1)$ such that $\UMK_{c,E,\TwoPlans}(\mu_0, \mu_1)=V_{c,E}(\mu_0, \mu_1;\gamma)$.
\end{lemma}

\begin{proof}Let $F(\gamma) \defeq V_{c,E}(\mu_0,\mu_1; \freevar) + \delta_{\TwoPlans(\mu_0, \mu_1)}$.
    By assumption, $0 \in \Dom F$, so $F$ is proper.
    By the boundedness of $\TwoPlans(\mu_0, \mu_1)$, it is also coercive.
    It is also lower semicontinuous by the closedness of $\TwoPlans(\mu_0, \mu_1)$, the assumed lower semicontinuity of $E$, and the linearity of the transport term.
    Now the convexity of $F$ and the direct method of the calculus of variations establishes the existence of a minimiser of $F$; see, e.g., \cite[Lemma 1.20, Theorem 2.1, and Remark 2.2]{clasonvalkonen2020nonsmooth}.
\end{proof}

\subsection{\texorpdfstring{Weak-$*$ convergence}{Weak-star convergence}}
\label{sec:transport-new:weakstar}

We next characterise weak-$*$ convergence via the convergence of the unbalanced transport costs $\UMK_{c,E,\TwoPlans}$. Our first lemma shows that weak-$*$ convergence implies the convergence of the unbalanced transport costs when the marginal cost has this property.

\begin{lemma}\label{lemma:unbalanced:weak-to-w}
    Suppose $0 \le c$, $0 \le E$, and $0 \in \TwoPlans(\nu,\mu)$ for all $\nu,\mu \in \Masses$.
    If $\this\mu-\this\nu \weaktostar 0$ implies $E(\this\nu, \this\mu) \to 0$, then $\this\mu-\this\nu \weaktostar 0$ implies $\UMK_{c,E,\TwoPlans}(\this\nu, \this\mu) \to 0$.
\end{lemma}

\begin{proof}Let $\{(\this\mu,\this\nu)\}_{k \in \N}$ be such that  $\this\mu-\this\nu \weaktostar 0$.
    Then $E(\this\nu,\this\mu) \to 0$.
    Since
    $
        0 \le \UMK_{c,E,\TwoPlans}(\this\nu, \this\mu) \le V_{c,E}(\this\nu, \this\mu; 0) = E(\this\nu, \this\mu),
    $
    it follows, as claimed, that $\UMK_{c,E,\TwoPlans}(\this\nu, \this\mu) \to 0$.
\end{proof}

The opposite implication requires stronger additional assumptions.

\begin{lemma}\label{lemma:unbalanced:w-to-weak}
    Let $\Omega \subset \R^n$ be a bounded set.
    Suppose $\epsilon \norm{x-y}^p \le c(x, y)$ for some $\epsilon>0$ and $p>1$ for all $x, y \in \Omega$; $0 \le E$; and that $E(\this\nu,\this\mu) \to 0$ (with $\{(\this\nu, \this\mu)\}_{k \in \N}$ bounded) implies $\this\mu-\this\nu \weaktostar 0$.
    Then $\UMK_{c,E,\TwoPlans}(\this\nu, \this\mu) \to 0$ (resp.~with $\{(\this\nu, \this\mu)\}_{k \in \N}$ bounded) implies $\this\mu - \this\nu \weaktostar 0$.
\end{lemma}

\begin{proof}Let $\{(\this\mu,\this\nu)\}_{k \in \N}$ be bounded with $\UMK_{c,E,\TwoPlans}(\this\nu, \this\mu) \to 0$.
    Then for some $\this\gamma \in \TwoPlans(\this\mu,\this\nu)$, we have
    \begin{equation}
        \label{eq:unbalanced:w-to-weakstar-e:1}
        E(\this\nu-\pi_\#^0\this\gamma, \this\mu-\pi_\#^1\this\gamma) \to 0
        \quad\text{and}\quad
        \int_{\Omega^2} c d\abs{\this\gamma} \to 0.
    \end{equation}
    Let $\phi \in \Predual$. Since $\Omega$ is bounded, $\phi$ can be extended to a continuous function on the closure of $\Omega$.
    Because the latter is a compact set, $\phi$ is Lipschitz on it with some factor $L_\phi$.
    By Hölder's inequality,
    \[
        \begin{split}
        \abs{\dualprod{(\pi_\#^1-\pi_\#^0)\this\gamma}{\phi}}
        &
        =
        \adaptabs{\int_{\Omega^2} \phi(y)-\phi(x) \d\this\gamma(x,y)}
        \le
        L_\phi \int_{\Omega^2} \norm{x-y} d\abs{\this\gamma}(x, y)
        \\
        &
        \le
        L_\phi
        \norm{\this\gamma}_{\TwoPlans}^{1-1/p}
        \left(\int_{\Omega^2} \norm{x-y}^p d\abs{\this\gamma}(x, y)\right)^{1/p}
        \\
        &
        \le
        L_\phi \epsilon^{-1/p}
        \norm{\this\gamma}_{\TwoPlans}^{1-1/p}
        \left(\int_{\Omega^2} c(x, y) d\abs{\this\gamma}(x, y)\right)^{1/p}.
        \end{split}
    \]
    By the second part of \eqref{eq:unbalanced:w-to-weakstar-e:1}, it now follows that $\dualprod{(\pi_\#^1-\pi_\#^0)\this\gamma}{\phi} \to 0$.
    We, thus,  obtain the weak-$*$ convergence of $(\pi_\#^1-\pi_\#^0)\this\gamma$ to zero.
    In particular, $\{(\pi_\#^1-\pi_\#^0)\this\gamma\}_{k \in \N}$ must be bounded.
    Our assumptions and the first part of \eqref{eq:unbalanced:w-to-weakstar-e:1} then imply
    $
        (\this\mu-\this\nu)-(\pi_\#^1-\pi_\#^0)\this\gamma \weaktostar 0.
    $
    Since $(\pi_\#^1-\pi_\#^0)\this\gamma \weaktostar 0$, necessarily $\this\mu-\this\nu \weaktostar 0$.
\end{proof}

We now concentrate on the special case $c=c_2$ for our marginal energies of primary interest.

\begin{corollary}
    $\UMK_{c_2,E_\Meas,\TwoPlans}(\this\nu, \this\mu) \to 0$ with $\{(\this\nu, \this\mu)\}_{k \in \N}$ bounded implies $\this\mu - \this\nu \weaktostar 0$.
\end{corollary}

\begin{proof}The conditions of \cref{lemma:unbalanced:w-to-weak} hold with $p=2$ and $\epsilon=1/2$.
\end{proof}

\begin{corollary}
    \label{cor:unbalanced:weak-wave}
    Suppose $0 \in \TwoPlans(\nu,\mu)$ for all $\nu,\mu \in \Masses$.
    Let $0 \not\equiv \rho \in C_0(\R^n) \isect L^2(\R^n)$ be symmetric and positive semi-definite and presentable as the autoconvolution $\rho=\rho^{1/2} * \rho^{1/2}$ for some $\rho^{1/2} \in L^2(\R^n) \isect C_0(\R^n)$.
    On a bounded domain $\Omega \subset \R^n$, let $\Wave \in \linear(\Masses; \Predual)$ be defined by $\Wave \mu = \rho * \mu$ for $\mu \in \Meas(\Omega)$.
    Then
    \begin{enumerate}[label=(\roman*)]
        \item\label{item:wave:constr:weak-to-wave-new}
        If  $\rho^{1/2} \in C_c(\R^n)$, then $\this\mu \weaktostar \mu$ weakly-$*$ in $\Meas(\Omega)$ implies $\UMK_{c_2, E_\Wave,\TwoPlans}(\this\mu, \mu) \to 0$.
        \item\label{item:wave:constr:wave-to-weak-new}
        If $\rho^{1/2}$ is locally Lipschitz, and
        $\Span \{ x \mapsto \rho^{1/2}(x - y) \mid y \in \Omega \}$ is dense in $\Predual$, then $\UMK_{c_2, E_\Wave,\TwoPlans}(\this\mu, \mu) \to 0$ with $\{\this\mu\}_{k \in \N} \subset \Meas(\Omega)$ bounded implies $\this\mu \weaktostar \mu$ weakly-$*$ in $\Meas(\Omega)$.
    \end{enumerate}
\end{corollary}

\begin{proof}Let $E(\mu, \nu) \defeq \frac{1}{2}\norm{\nu-\mu}_\Wave$.
    Then \cite[Theorem 2.4]{tuomov-pointsource} verifies the conditions of \cref{lemma:unbalanced:weak-to-w,lemma:unbalanced:w-to-weak} regarding $E$.
    The rest follow from those lemmas.
\end{proof}

\section{Estimates and expansions}
\label{sec:subdiff}

We now study subdifferential-type lower estimates that incorporate unbalanced transport, as well as transport-relative smoothness of functions.
We first review existing “transport subdifferentials” in \cref{sec:subdiff:balanced}, then develop relevant estimates for the unbalanced transport costs $V_{c,E}$ in \cref{sec:subdiff:cost}.
Based on this, in \cref{sec:subdiff:subdiff} we suggest a definition of an unbalanced transport subdifferential.
In the final \cref{sec:subdiff:smoothness}, we discuss relevant concepts of smoothness, that we will be using.

\subsection{Balanced transport subdifferentials}
\label{sec:subdiff:balanced}

In \cite{ambrosio2008gradientflows}, a theory is presented for (extended) Fréchet subdifferentials of functions of probability measures with respect to three-transport plans.
Specifically, with $\mu_1, \mu_2 \in \ProbMeas(\Omega)$, denote the set of optimal transports from $\mu_1$ to $\mu_2$ by
\[
    \Gamma_o(\mu_1, \mu_2) \defeq \left\{
    \gamma \in \Gamma(\mu_1, \mu_2)
        \,\middle|\,
        \MK_{c_2}(\mu_1, \mu_2) = \int c_2(x, y) \d\gamma(x,y)
    \right\}.
\]
Then \cite[Definition 10.3.1]{ambrosio2008gradientflows} defines the (extended) Fréchet subdifferential of $F: \ProbMeas(\Omega) \to (-\infty, \infty]$ at $\mu_1$ to be the set $\subdiff F(\mu_1)$ of transport plans $\gamma \in \TwoMassesNonNeg$ satisfying $\pi_\#^0\gamma=\mu_1$ and\footnote{%
    For simplicity we consider here only the exponent $p=2$, and a bounded set $\Omega$.
    The definition of $\Gamma_o$ on \cite[page 14]{ambrosio2008gradientflows} has a typing mistake.
    The element of $\ProbMeas(\Omega^3)$ should be $\mathbf{\mu}$ instead of $\gamma$.
    For consistency, our indexing also differs from \cite{ambrosio2008gradientflows}.
}
\begin{equation}
    \label{eq:subdiff:frechet-transport}
    F(\mu_2)-F(\mu_1) \ge \inf \left\{
        \int_{\Omega^3} \iprod{x}{z-y} d\tritrans(x,y,z)
        + o(W_2(\mu_1, \mu_2))
        \,\middle|\,
        \begin{array}{l}
            \tritrans \in \ProbMeas(\Omega^3),\,
            \pi_\#^{1,0}\tritrans=\gamma,\\
            \pi_\#^{1,2} \tritrans \in \Gamma_o(\mu_1, \mu_2)
        \end{array}
    \right\}.
\end{equation}
According to  \cite[Theorem 10.3.6]{ambrosio2008gradientflows}, for geodesically convex $F$ the $o$-term can be omitted.

On the other hand, with $\mu_0,\mu_1,\mu_2 \in \ProbMeas(\Omega)$, let $\tritrans \in \ProbMeas(\Omega^3)$ be such that $\pi_\#^{0,2}\tritrans \in \Gamma_0(\mu_0, \mu_2)$ and  $\pi_\#^{0,1}\tritrans \in \Gamma_0(\mu_0, \mu_1)$.
Then, using Pythagoras' identity, we expand
\[
	\begin{split}
		W_2^2(\mu_0, \mu_2) - W_2^2(\mu_0, \mu_1)
		&
        =
        \frac{1}{2}\int_{\Omega^2} \abs{z-x}^2 d\pi_\#^{0,2}\tritrans(x, z)
        - \frac{1}{2}\int_{\Omega^2} \abs{y-x}^2 d\pi_\#^{0,1}\tritrans(x, y)
        \\
        &
		=
		\frac{1}{2}\int_{\Omega^3} \abs{z-x}^2 - \abs{y-x}^2 d\tritrans(x, y, z)
        \\
		&
		=
		\int_{\Omega^3} \iprod{y-x}{z-y} + \frac{1}{2}\abs{z-y}^2 d\tritrans(x, y, z).
    \end{split}
\]
Hence,
\[
    W_2^2(\mu_0, \mu_2)
    -
    W_2^2(\mu_0, \mu_1)
    \ge
    \int_{\Omega^3} \iprod{y-x}{z-y} d\tritrans(x, y, z)
    +W_2^2(\mu_1, \mu_2).
\]
In analogy with standard convex subdifferentials of $c_2$, this suggests a different definition of a “transport subdifferential”: $g(y, x)=y-x$ should be a transport subdifferential of $W_2^2(\nu, \freevar)$.
We will base our approach on this latter idea, extending it to the setting of unbalanced transport.
However, to avoid introducing high computational costs, we will not work with the optimal squared distances $W_2^2$, and more generally, $\UMK_{c,E,\TwoPlans}$, but with $V_{c,E}$.

\subsection{Transport costs}
\label{sec:subdiff:cost}

With an eye on replacing the Pythagoras' or three-point identity satisfied by Hilbert space norms, we now derive three-point identities and inequalities for unbalanced optimal transport distances.
For our first result, we assume to be given an energy $E: \Masses^2 \to \R$ that satisfies the bound
\begin{equation}
    \label{eq:unbalanced:energy}
    E(\mu, \opt\mu) \ge \dualprod{\omega}{\opt\mu-\nu} + E(\mu, \nu) + E(\nu, \opt\mu)
    \quad (\omega \in \subdiff E(\mu, \freevar)(\nu))
\end{equation}
and is convex in the second parameter.
This is the case for $E=B_J$ a Bregman divergence with a (for simplicity) smooth generator $J$; see \cref{lemma:bregman-three-point}.
In particular, $E_\Wave$ satisfies \eqref{eq:unbalanced:energy}, while $E_\Meas$ does not.

We wish to apply $V_{c_2,E}$ of \cref{sec:transport-new:definitions} to the marginal projections $\pi_\#^{i,j}\tritrans$ of the three-transport $\tritrans$.
However, for the argument of the next theorem to work, we need to swap the total variation measure $\abs{\pi_\#^{i,j}\tritrans}$ of the projection for the projection $\pi_\#^{i,j}\abs{\tritrans}$ of the total variation measure.
Towards this end, we introduce for $(i, j) \in \{(0,1),(0,2),(1,2)\}$ and all $\mu, \nu \in \Masses$ and $\tritrans \in \ThreePlansSpace$, the distances
\[
    \begin{split}
    \bar V_{c_2, E}^{i,j}(\mu, \nu; \tritrans)
    &
    \defeq
    V_{c_2, E}(\mu, \nu; \pi_\#^{i,j}\tritrans)
    + \int_{\Omega^2} c_2 d(\pi_\#^{i,j}\abs{\tritrans} - \abs{\pi_\#^{i,j}\tritrans}).
    \\
    &
    =
    \int_{\Omega^2} c_2(x, y) \d \pi_\#^{i,j}\abs{\tritrans}(x, y)
    +
    E(\mu - \pi_\#^i \tritrans, \nu - \pi_\#^j \tritrans).
    \end{split}
\]
If $\tritrans$ is a positive measure, $\bar V_{c_2, E}^{i,j}(\mu, \nu; \tritrans)=V_{c_2, E}(\mu, \nu; \pi_\#^{i,j}\tritrans)$.
For $\star=\Meas,\Wave$ (see \eqref{eq:unbalanced:v-star}), we can expand
\[
    \bar V_{c_2, E_\star}^{i,j}(\mu, \nu; \tritrans)
    =
    \int_{\Omega^3} \frac{1}{2}\abs{x-y}^2 \d \abs{\tritrans}(x, y, z)
    +
    \frac{1}{2}\norm{\nu-\mu - (\pi_\#^j-\pi_\#^i) \tritrans}_\star^2.
\]
In \cref{sec:sub}, we will apply the next theorem with $\mu_0=\this\mu$ and $\mu_1=\nexxt\mu$ two consecutive iterates of our proposed algorithm, and $\mu_2$ a comparison measure, e.g., an optimal solution.

\begin{theorem}\label{thm:unbalanced:twocost-three-point}
    Let $E$ be as above and $J: \Meas(\Omega) \to \extR$ be convex, proper, weak-$*$ lower semicontinuous.
    Pick $\mu_0,\mu_1, \mu_2 \in \Masses$ as well as $\tritrans \in \ThreePlansSpace$.
    Then for any $\omega \in \subdiff E(\mu_0 - \pi_\#^0\tritrans_{}, \freevar)(\mu_1 - \pi_\#^1\tritrans_{})$, we have
    \begin{equation}
        \label{eq:unbalanced:twocost-three-point:0}
        \begin{split}
        \bar V_{c_2, E}^{0,2}(\mu_0, \mu_2; \tritrans) - \bar V_{c_2, E}^{0,1}(\mu_0, \mu_1; \tritrans)
        &
        \ge
        \int_{\Omega^3} \iprod{y-x}{z-y} \d\abs{\tritrans}(x, y, z)
        \\
        \MoveEqLeft[-1]
        +\dualprod{\omega}{\mu_2- \mu_1-(\pi_\#^2-\pi_\#^1)\tritrans}
        + \bar V_{c_2, E}^{1,2}(\mu_1, \mu_2; \tritrans).
        \end{split}
    \end{equation}
\end{theorem}

\begin{proof}We expand
    \begin{gather}
        \label{eq:unbalanced:twocost-three-point:split}
        \bar V_{c_2, E}^{0,2}(\mu_0, \mu_2; \tritrans)
        - \bar V_{c_2, E}^{0,1}(\mu_0, \mu_1; \tritrans)
        =
        D + I
    \shortintertext{for}
        \label{eq:unbalanced:twocost-three-point:I}
        \begin{split}
        I
        &
        \defeq
        \frac{1}{2}\int_{\Omega^2} \abs{z-x}^2 \d\pi_\#^{0,2}\abs{\tritrans}(x, z)
        - \frac{1}{2}\int_{\Omega^2} \abs{y-x}^2 \d\pi_\#^{0,1}\abs{\tritrans}(x, y)
        \\
        &
        =
        \frac{1}{2}\int_{\Omega^3} [\abs{z-x}^2 - \abs{y-x}^2] \d\abs{\tritrans}(x, y, z)
        \\
        &
        =
        \int_{\Omega^3} \iprod{y-x}{z-y} + \frac{1}{2}\abs{z-y}^2 \d\abs{\tritrans}(x, y, z)
        \\
        &
        =
        \int_{\Omega^3} \iprod{y-x}{z-y} \d\abs{\tritrans}(x, y, z)
        +  \frac{1}{2} \int_{\Omega^2} \abs{z-y}^2 \d\pi_\#^{1,2}\abs{\tritrans}(y,z)
        \end{split}
    \shortintertext{and}
        \label{eq:unbalanced:twocost-three-point:D}
        \begin{split}
            D
            &
            \defeq
            E(\mu_0 - \pi_\#^0\tritrans, \mu_2 - \pi_\#^2\tritrans)
            - E(\mu_0 - \pi_\#^0\tritrans, \mu_1 - \pi_\#^1\tritrans)
            \\
            &
            \ge
            E(\mu_1 - \pi_\#^1\tritrans, \mu_2 - \pi_\#^2\tritrans)
            + \dualprod{\omega}{(\mu_2 - \pi_\#^2\tritrans)-(\mu_1 - \pi_\#^1\tritrans)}.
        \end{split}
    \end{gather}
    Combined, \cref{eq:unbalanced:twocost-three-point:split,eq:unbalanced:twocost-three-point:I,eq:unbalanced:twocost-three-point:D} yield the claim.
\end{proof}

\begin{remark}
    In the previous theorem, we can replace $c_2$ by any Bregman divergence $B_j$ for a convex and smooth $j: \Omega \to \R$, and $\iprod{y-x}{z-y}$ by $\iprod{D_2 B_j(x, y)}{z-y}$.
\end{remark}

The following variant only considers two-transport, and omits the comparison measure $\mu_2$.
We then do not require \eqref{eq:unbalanced:energy}, so the result applies to $E_\Meas$ of \eqref{eq:unbalanced:emeas}. We will use the result in \cref{sec:fb} to obtain a weaker convergence result than the one that we obtain in \cref{sec:sub} by assuming \cref{eq:unbalanced:energy}.

\begin{theorem}\label{thm:unbalanced:twocost-three-point:reverse}
    Let $E: \Masses^2 \to \R$ be convex in the second parameter and satisfy $E(\nu,\nu)=0$ for all $\nu$.
    Then for any $\mu_0,\mu_1 \in \Masses$; $\gamma \in \TwoPlansSpace$; and $\omega \in \subdiff E(\pi_\#^0\gamma - \mu_0, \freevar)(\pi_\#^1\gamma - \mu_1)$, we have
    \[
        \begin{split}
        0
        &
        \ge
        \int_{\Omega^{2}} -\abs{y-x}^2\d\abs{\gamma}(x,y)
        -\dualprod{\omega}{\mu_{1}- \mu_{0}-(\pi_\#^1-\pi_\#^0)\gamma}
        + V_{2c_2, E}(\mu_0, \mu_1; \gamma)
        \\
        \MoveEqLeft[-1]
        + E(\mu_0 - \pi_\#^0\gamma, \freevar + \mu_0 - \pi_\#^0\gamma)^*(\omega).
        \end{split}
    \]
\end{theorem}

\begin{proof}By the Fenchel–Young equality,
    \[
        E(\mu_0 - \pi_\#^0\gamma, \mu_1 - \pi_\#^1\gamma)
        +
        E(\mu_0 - \pi_\#^0\gamma, \freevar)^*(\omega)
        =
        \dualprod{\omega}{\mu_1 - \pi_\#^1\gamma}.
    \]
    By the properties of Fenchel conjugates (e.g., \cite[Lemma 5.7\,(ii)]{clasonvalkonen2020nonsmooth}), this rearranges as
    \[
        0
        =
        E(\mu_0 - \pi_\#^0\gamma, \mu_1 - \pi_\#^1\gamma)
        +
        E(\mu_0 - \pi_\#^0\gamma, \freevar + \mu_0 - \pi_\#^0\gamma)^*(\omega)
        +\dualprod{\omega}{\mu_0-\mu_1-(\pi_\#^0-\pi_\#^1)\gamma}.
    \]
    This and the construction of $V_{2c_2, E}$ give the claim.
\end{proof}

\begin{remark}
    If $E(\mu,\nu)=J(\mu-\nu)$, we have $E(\mu_0 - \pi_\#^0\gamma, \freevar + \mu_0 - \pi_\#^0\gamma)^*(\omega)=J^*(-\omega)$.
\end{remark}

\subsection{Transport subdifferentials}
\label{sec:subdiff:subdiff}

\begin{theorem}\label{thm:unbalanced:subdiff:general}
    Let $G: \Masses \to \extR$ be convex, proper, and lower semicontinuous.
    Then, for any $\mu_2,\mu_1 \in \Masses$ as well as $\tritrans \in \ThreePlansSpace$, for all differentiable $w \in \subdiff G(\mu_1)$, we have
    \begin{equation}
        \label{eq:unbalanced:subdiff:general}
        G(\mu_2) - G(\mu_1)
        \ge
        \int_{\Omega^3} \iprod{\grad w(y)}{z-y} \d\tritrans(x, y, z) + \dualprod{w}{\mu_2-\mu_1 - (\pi_\#^2-\pi_\#^1)\tritrans}
        + r(w,\tritrans),
    \end{equation}
    where
    $
        r(w,\tritrans) \defeq  \int_{\Omega^3} B_w(y, z) \d\tritrans(x, y, z).
    $
\end{theorem}

\begin{proof}By the definition of the convex subdifferential, we have
    \[
        \begin{split}
        G(\mu_2) - G(\mu_1)
        &
        \ge
        \dualprod{w}{\mu_2-\mu_1}
        =
        \int_{\Omega^3} w(z)-w(y) \d\tritrans(x, y, z) + \dualprod{w}{\mu_2-\mu_1 - (\pi_\#^2-\pi_\#^1)\tritrans}
        \\
        &
        =
        \int_{\Omega^3} \iprod{\grad w(y)}{z-y} \d\tritrans(x, y, z) + \dualprod{w}{\mu_2-\mu_1 - (\pi_\#^2-\pi_\#^1)\tritrans}
        + r(w,\tritrans).
        \qedhere
        \end{split}
    \]
\end{proof}

\begin{remark}[Intepretation: transport subdifferentials]
    Taking for simplicity $\tritrans \ge 0$, we can interpret \eqref{eq:unbalanced:twocost-three-point:0} as $(g, \omega)$ for $g(y, x) \defeq y-x$ being a ($V_{c_2,E}$-strong) \term{unbalanced transport subdifferential} of $V_{c_2,E}(\mu_0, \freevar; \freevar)$.
    Likewise, if $w$ were convex and $\tritrans \ge 0$, so that $r(w,\tritrans) \ge 0$, we could interpret $(g, w)$ for $g(x, y) \defeq \grad w(x)$ as an unbalanced transport subdifferential of $G$.
    In both cases, $g$ is the transport component of the subdifferential, and $w$ or $\omega$ is the marginal component, only evaluated against $\mu_2- \mu_1-(\pi_\#^2-\pi_\#^1)\tritrans$.

    In general, we cannot expect $w$ to be convex.
    Then, to define a weaker form of an unbalanced transport subdifferential, similarly to \eqref{eq:subdiff:frechet-transport}, it would be possible to assume $r(w,\tritrans)$ to be small by considering proximal subdifferentials or restricting the set of three-plans $\tritrans$ such that $(\pi_\#^2-\pi_\#^1)\tritrans \approx 0$.
    The algorithm that we present in \cref{sec:fb} will require such restrictions, however, we refrain at this stage from proposing an explicit definition of a transport subdifferential.
\end{remark}

\subsection{Concepts of smoothness}
\label{sec:subdiff:smoothness}

As we intend to derive forward-backward type methods for \eqref{eq:intro:problem}, we will need to make more precise the various flavours of smoothness required from $F$.
For our main concept smoothness, we say that a convex and pre-differentiable $F: \Masses \to \R$ is $(L,\ell)$-\term{smooth} with respect to $E$ and $c$, if for all $\mu,\nu \in \Masses$ and $\gamma \in \TwoPlansSpace$, we have
\[
    B_F(\mu + (\pi_\#^1-\pi_\#^0)\gamma, \nu)
    \le
    V_{\ell c, L E}(\mu, \nu; \gamma).
\]

\begin{example}\label{ex:estimates:smoothness}
    Let $F(\mu)=\frac{1}{2}\norm{A\mu-b}^2$ for some $A \in \linear(\Masses; Y)$ and a Hilbert space $Y$. Then
    \[
        B_F(\mu + (\pi_\#^1-\pi_\#^0)\gamma, \nu)
        =
        \frac{1}{2}\norm{A(\nu-\mu-(\pi_\#^1-\pi_\#^0)\gamma)}^2.
    \]
    On the other hand, for $E(\mu, \nu)=\frac{1}{2}\norm{\nu-\mu}_\Wave^2$ and a $\Wave \in \linear(\Masses; \Predual)$, we have
    \[
        V_{\ell c_2, L E}(\mu, \nu; \gamma)
        =
        \ell \int_{\Omega^2} c_2(x, y) \d\abs{\gamma}(x, y)
        +
        \frac{L}{2}\norm{\nu-\mu-(\pi_\#^1-\pi_\#^0)\gamma}_\Wave^2
    \]
    Therefore, $F$ is  $(L,\ell)$-smooth with respect to $E$ and $c_2$ if
    $
        A_*A \le L \Wave
        \quad\text{and}\quad
        \ell \ge 0.
    $
\end{example}

We also need $F$ to satisfy for any (bounded) $\mu \in \Masses$ and $\gamma \in \TwoPlansSpace$, the \emph{lower curvature bound}
\begin{equation}
    \label{eq:unbalanced:lower-curvature}
    \ellF \abs{\gamma}(c_2)
    \ge
    - B_F(\mu+(\pi_\#^1-\pi_\#^0)\gamma,\mu)
\end{equation}
The next lemma provides an example of the satisfaction of this property when $\mu$ is “observed” by $m$ linear “sensors” $a_i$, and the quality of the observation is measured by the fidelity functions $\phi_i$.
With $Y=\R^m$, this includes $F$ of \cref{ex:estimates:smoothness}, which satisfies \eqref{eq:unbalanced:lower-curvature} also through convexity.

\begin{lemma}\label{lemma:unbalanced:example-lower-curvature}
    The lower curvature bound \eqref{eq:unbalanced:lower-curvature} holds in the following cases:
    \begin{enumerate}[label=(\alph*)]
        \item $F$ is convex, in which case $\ellF=0$.
        \item Provided $\norm{\gamma} \le m_\gamma$, also $F(\mu)=\sum_{i=1}^m \phi_i(a_i(\mu))$, where $a_i \in \Predual$ is $L_{a_i}$-Lipschitz, and $\phi_i: \R \to \R$ has $L_{\phi_i'}$-Lipschitz derivative, ($i=1,\ldots,m$).
        In this case $\ellF = m_\gamma\sum_{i=1}^m L_{\phi_i'} L_{a_i}^2$.

        If the $L_{a_i}=L$ and $L_{\phi_i*}=L'$ are equal, we can take $\ellF = 2m_\gamma N_\psi L' L^2$,  where the maximum number of overlapping supports
        $
            N_\psi \defeq \max\{ \#P \mid y \in \Omega,\, P \subset \{1,\ldots,m\},\, a_k(y) \ne 0 \text{ for all } k \in P\}.
        $
    \end{enumerate}
\end{lemma}

\begin{proof}Let $\mu \in \Masses$, $\gamma \in \TwoPlansSpace$.
    Abbreviate $\Delta \defeq (\pi_\#^1-\pi_\#^0)\gamma$.
    If $F$ is convex, $B_F(\mu+\Delta,\mu) \ge 0$ by the definitions of the Bregman divergence and the convex subdifferential.

    Otherwise, by rearrangements, the Lipschitz assumptions, and Jensen's inequality, we deduce
    \begin{equation}
        \label{eq:unbalanced:tlip}
        \begin{split}
        \dualprod{F'(\mu + (\pi_\#^1-\pi_\#^0)\gamma) &- F'(\mu)}{(\pi_\#^1-\pi_\#^0)\gamma}
        =
        \sum_{i=1}^m [\phi_i'(a_i(\mu+\Delta))-\phi_i'(a_i(\mu))]a_i(\Delta)
        \\
        &
        \le
        \sum_{i=1}^m L_{\phi_i'} \abs{a_i(\Delta)}^2
        =
        \sum_{i=1}^m L_{\phi_i'} \adaptabs{\int_{\Omega^2}[a_i(y)-a_i(x)]\d\gamma(x,y)}^2
        \\
        &
        \le
        \sum_{i=1}^m L_{\phi_i'} \norm{\gamma}_\Meas \int_{\Omega^2}\abs{a_i(y)-a_i(x)}^2\d\abs{\gamma}(x,y)
        \\
        &
        \le
        \sum_{i=1}^m L_{\phi_i'} L_{a_i}^2 \norm{\gamma}_\Meas \int_{\Omega^2}\abs{y-x}^2\d\abs{\gamma}(x,y)
        =
        \Theta\norm{\gamma}_\Meas \abs{\gamma}(c_2),
        \end{split}
    \end{equation}
    where $\Theta=2\sum_{i=1}^m L_{\phi_i'} L_{a_i}^2$.
    Observing that at most $2N_\psi$ components of the sum inside the first integral above are nonzero for each $(x, y)$, we can refine this to $\Theta = 4N_\psi L' L^2$.

    Applying \eqref{eq:unbalanced:tlip} to $\mu+\Delta$ in place of $\mu$, and $t\gamma$ in place of $\gamma$, we then estimate
    \[
        \begin{split}
        -B_F(\mu+\Delta,\mu)
        &
        =
        F(\mu+\Delta) -  F(\mu) + \dualprod{F'(\mu+\Delta)}{-\Delta}
        =
        \int_0^1 \dualprod{F'(\mu+\Delta -t\Delta)-F'(\mu+\Delta)}{\Delta} \d t
        \\
        &
        \le
        \int_0^1 t \Theta \abs{\gamma}(c_2)\norm{\gamma}_\Meas \d t
        \le
        \frac{1}{2}\Theta m_\gamma \abs{\gamma}(c_2).
        \end{split}
    \]
    Inserting the expression for $\Theta$, we obtain the claim.
\end{proof}

Finally, it will be useful to have estimates of Lipschitz factors of the gradients of the values of $F'$.

\begin{lemma}\label{lemma:unbalanced:lip-diff-value-sensorA}
    With $F$ as in  \cref{lemma:unbalanced:example-lower-curvature}, assume each $a_i$ is Lipschitz-differentiable with the factors $L_{\grad a_i}$, and each $\phi_i'$ is bounded by some constant $M_i$.
    Let $v=F'(\mu)$ for some $\mu \in \Masses$.
    Then $\grad v$ is $\ell_{\grad v}$-Lipschitz with $\ell_{\grad v}=\sum_{i=1}^m L_{\grad a_i} M_i$.
    If the $L_{\grad a_i}=L$ and $M_i=M$ are independent of the index $i$, this can be improved to $\ell_{\grad v}=N_\psi L M$.
\end{lemma}
\begin{proof}We have $v = \sum_{i=1}^m \phi_i'(a_i(\mu))a_i$.
    Thus, $\grad v(x) = \sum_{i=1}^m \phi_i'(a_i(\mu))\grad a_i$, from where the claims follow.
\end{proof}

\section{Sliding forward-backward splitting}
\label{sec:fb}

With a particular interest in the instance \eqref{eq:intro:problem}, we now develop a forward-backward splitting method based on unbalanced optimal transport for the general problem
\begin{equation}
    \label{eq:fb:problem}
    \min_{\mu \in \Masses} F(\mu) + G(\mu),
\end{equation}
where $F$ satisfies several smoothness properties, while $G$ is a general convex, and possibly nonsmooth function.
In this section we do not assume the convexity of $F$.
We require that the differentials $F'(\mu) \in \DiffPredual$.
The assumption on $F'(\mu)$ could similarly be relaxed, and the differentiability could be relaxed to a bounded linear approximation property (encoded by bounds on integrals of Bregman divergences $B_{F'(\mu)}$, etc.).

In the next \cref{sec:fb:preparation}, we introduce the general form of our algorithm.
Then in \cref{sec:fb:assumptions}, we state and provide examples of the satisfaction of the basic subdifferential convergence theory of \cref{sec:fb:sub-convergence}.
We treat function value convergence in \cref{sec:sub}, and some fine details of the algorithms in \cref{sec:controls}.

\subsection{The algorithm}
\label{sec:fb:preparation}

On each step of our algorithm, we will, roughly speaking, take a forward step with respect to $F$, and a proximal step with respect to $G$, in the $V_{c_2,E}$ “metric”.
The step is determined by the marginal and transport components of the transport subdifferentials of \cref{sec:subdiff:subdiff}, although we do not make the connection explicit here.
Moreover, similarly to how a standard forward step evaluates the differential at the previous iterate, we also evaluate differentials at intermediate iterates.

\begin{rawalg}
    Let an energy $E: \Masses \times \Masses \to \R$, step length parameters $\theta, \tau > 0$, and tolerances $\{\nexxt\epsilon\}_{k \in \N} > 0$ be given.
    On each step $k \in \N$, we seek
    \begin{subequations}
    \label{eq:fb:oc0}
    \begin{equation}
        \label{eq:fb:oc0:vars}
        \left\{\begin{array}{ll}
        \text{a measure}& \nexxt\mu \in \Masses
        \quad\text{and}
        \\
        \text{a transport plan}&\nexxt\gamma \in \TwoPlansSpace
        \quad
        \text{(unbalanced from $\this\mu$ to $\nexxt\mu$),\quad and}
        \\
        \text{subderivatives}&
        \nexxt w \in \subdiff G(\nexxt\mu)
        \quad\text{and}
        \\
        &
        \nexxt{\distsub} \in \subdiff E(\this\mu-\pi_\#^0\nexxt\gamma, \freevar)(\nexxt\mu-\pi_\#^1\nexxt\gamma)
        \end{array}\right.
    \end{equation}
    that satisfy for the \emph{transported iterate}
    \begin{gather}
        \nonumber
        \this{\breve\mu} \defeq \this\mu + (\pi_\#^1-\pi_\#^0)\nexxt\gamma
        \\
    \shortintertext{as well as the derivatives}
        \nonumber
        \thisv \defeq F'(\this\mu)
        \quad\text{and}\quad
        \this{\breve v} \defeq F'(\this{\breve\mu}),
        \\
    \shortintertext{\term{the transport step}}
        \label{eq:fb:oc0:transport}
        y = x - \theta \tau \sign \nexxt\gamma(x, y)\grad \thisv(x)
        \quad\text{for}\quad \nexxt\gamma \text{-a.e. } (x, y),
        \\
    \shortintertext{\term{the marginal (insertion/weight-optimisation) step}}
        \label{eq:fb:oc0:marginal}
        -\nexxt\epsilon
        \le
        \nexxt{\tilde\epsilon} \defeq \tau[\this{\breve v} + \nexxt w] + \nexxt{\distsub}
        \le
        \nexxt\epsilon,
    \end{gather}%
    \end{subequations}%
    and, for some factors $\check C, \ellCurvature \ge 0$, the \term{remainder condition}
    \begin{gather}
        \label{eq:fb:remainder}
        \begin{split}
        \check C \nexxt\epsilon
        \ge
        \nexxt{\check\remainder}
        &
        \defeq
        - \dualprod{\nexxt{\tilde\epsilon}}{\this\mu- \nexxt\mu}
        - \tau\ellCurvature\abs{\nexxt\gamma}(c_2)
        \\
        \MoveEqLeft[-1]
        + \int_{\Omega^3}
            \tau \iprod{\grad \this v(x)}{x-y}
            +\nexxt\omega(x)-\nexxt\omega(y)
        \d\nexxt\gamma(x,y).
        \end{split}
    \end{gather}
\end{rawalg}

\begin{remark}[No transport]
   The condition \cref{eq:fb:oc0:transport} holds for $\nexxt\gamma=0$.
   In this case, in \eqref{eq:fb:remainder}, also $\nexxt{\check\remainder} = - \dualprod{\nexxt{\tilde\epsilon}}{\this\mu - \nexxt\mu}$ can be controlled by solving the marginal step to sufficiently high quality.
\end{remark}

\begin{remark}[Explanation of the conditions]
    \label{rem:fb:oc-explanation}
    Let us unpack \eqref{eq:fb:oc0} and \eqref{eq:fb:remainder}:
    \begin{enumerate}\item Part \eqref{eq:fb:oc0:vars} establishes the spaces of variables, and the symbols for select subderivatives of the functions $G$ and $E$.
        \item Part \eqref{eq:fb:oc0:transport} is a \term{transport step}.
        It initialises $\nexxt\gamma=\sum_i \beta_i \delta_{(x_i, y_i)}$ for some $\beta_i \in \R$, enforcing that each source point $x_i$ and target point $y_i$ realise the spatial gradient descent (or ascent) update
        \[
            y_i = x_i - \sign\beta_i\theta\tau \grad \thisv(x_i).
        \]
        To weights $\beta_i$ is mainly constrained by \cref{eq:fb:remainder}.
        In practise, we start with $\gamma$ such that $\pi_\#^0\nexxt\gamma=\sum_{i=1} \beta_i \delta_{x_i}=\this\mu$, solve \eqref{eq:fb:oc0:marginal}, and reduce the weights $\beta_i$ and repeat the process, if needed.

        \item Part \eqref{eq:fb:oc0:marginal} is an inexact forward-backward \term{marginal step} with respect to the energy $E$ to form $\nexxt\mu$ starting from $\this{\breve\mu}$. Indeed, the conditions expand as  $\norm{\nexxt{\tilde\epsilon}}_\infty \le \nexxt\epsilon$ with
        \[
            \nexxt{\tilde\epsilon} \in \tau[F'(\this{\breve\mu}) + \subdiff G(\nexxt\mu)] + \subdiff E(\this\mu-\pi_\#^0\nexxt\gamma, \freevar)(\nexxt\mu-\pi_\#^1\nexxt\gamma)
        \]
        In practise, given $\this{\breve \mu}=\sum_{i=1}^n \alpha_i \delta_{y_i}$, we form $\nexxt\mu$ through the addition (“insertion”) of new spikes $\sum_{i=n+1}^m \alpha_i \delta_x$, and the subsequent optimisation of the weights $\alpha_i$ through a finite-dimensional variant of \eqref{eq:fb:oc0:marginal}.
        This procedure is similar to \cite{tuomov-pointsource} and conditional gradient methods. The exact insertion mechanism depends on the energy $E$, and on the function $G$, as we elaborate below.

        \item Part \eqref{eq:fb:remainder} is a \emph{remainder control} condition. The first term of $\nexxt{\check\remainder}$ further controls the inexactness of the marginal step.
        The remaining terms are related to the curvature of the problem along $\nexxt\gamma$.
        They can be given various alternative forms through \eqref{eq:fb:oc0:marginal}, some more suggestive of the term “curvature” than the form given here.
        We, however, find, this form the easiest to control, as we will discuss in \cref{sec:controls}.
        The factor $\ellCurvature$ will also appear in our step length conditions.
        Although its notation and use suggest that it would be a uniform Lipschitz factor of $\grad\this v$ of $k \in \N$, our theory does not enforce this: any deviation is compensated for by \eqref{eq:fb:remainder}.
    \end{enumerate}
\end{remark}

\begin{algorithm}[t!]
    \caption{Rough algorithm for the point source localisation problem}
    \label{alg:fb:rough}
    \begin{algorithmic}[1]
        \Require Fréchet-differentiable $F: \Masses \to \R$,  $G=\alpha \norm{\freevar}_\Meas + \delta_{\ge 0}$ for some $\alpha > 0$, and either $E=E_\Meas$ or $E=E_\Wave$ for a self-adjoint and positive semi-definite $\Wave \in \linear(\Masses; \DiffPredual)$.
        (Adaptation to $G=\alpha\norm{\freevar}_\Meas$ only depends on replacing \cref{alg:fb:insert-and-remove,alg:fb:insert-and-remove-radon}; see \cite{tuomov-pointsource}.)
        Tolerance multipliers $\check C>0$ and $\kappa \in (0, 1)$, and base tolerances $\{\nexxt\epsilon\}_{k \in \N}$.
        Step length parameters $\tau,\theta>0$.
        Maximum number $N_\gamma \in \N$ of transport attempts.
        \Ensure
        \cref{ass:fb:all,ass:fb:subdiff} (subdifferential convergence, \cref{thm:fb:subdiff-conv}); or
        \\
        \cref{ass:fb:all,ass:sub:energy,ass:sub:rebalancing} (function value convergence, \cref{thm:sub:convergence:function-ergodic} or \cref{cor:sub:bounded}).
        \State Choose an initial $\mu^0 \in \DiscreteMasses$.
        \For{$k \in \N$}
            \State
                Set $\this v \defeq F'(\this\mu)$.\label{step:fb:rough:v}
            \State
                Decompose $\sum_{i=1}^n \alpha_i\delta_{x_i} \defeq \this\mu$.
            \State
                Form\label{step:fb:rough:gamma-init}
                $
                    \nexxt\gamma=\sum_{i=1}^n \beta_i\delta_{(x_i, y_i)}
                $
                with
                $
                    \beta_i \defeq \alpha_i,
                $
                and
                $
                    y_i \defeq x_i - s_i\theta\tau \grad \thisv(x_i).
                $
                \Comment{Ensures \eqref{eq:fb:oc0:transport}}
            \Repeat\label{step:fb:rough:transport-adaptation}
                \State
                    Form
                    $
                        \this{\breve\mu} \defeq \this\mu + (\pi_\#^1-\pi_\#^0)\nexxt\gamma
                        =\sum_{i=1}^n\left( \beta_i \delta_{y_i} + (\beta_i-\alpha_i)\delta_{x_i}\right).
                    $
                \State
                    Set $\this{\breve v} \defeq F'(\this{\breve\mu})$.\label{step:fb:rough:vbreve}
                \If{$E=E_\Wave$}
                    \State
                    $
                        \nexxt\mu \defeq
                        \textproc{insert\_and\_adjust}(\this{\breve\mu}, \this{\breve v}, \alpha, \tau, \kappa \nexxt{\epsilon}, \Wave)
                    $
                    \Comment{\cref{alg:fb:insert-and-remove} for \eqref{eq:fb:oc0:marginal}}
                \ElsIf{$E=E_\Meas$}
                    \State
                    $
                        \nexxt\mu \defeq
                        \textproc{insert\_and\_adjust\_radon}(\this{\breve\mu}, \this{\breve v}, \alpha, \tau, \kappa \nexxt{\epsilon})
                    $
                    \Comment{\cref{alg:fb:insert-and-remove-radon} for \eqref{eq:fb:oc0:marginal}}
                \EndIf
                \If{\eqref{eq:fb:remainder} (or \eqref{eq:sub:remainder}, or \eqref{eq:sub:remainder-tight}, depending on mode of convergence sought) does not hold}
                    \State Reduce the weights $\beta_i$ in $\nexxt\gamma$.\label{step:fb:rough:reduce}
                    \Comment{\cref{sec:controls}}
                \EndIf
                \If{some $\beta_i$ was changed above, and this loop has been executed $N_\gamma$ times}
                    \State Set $\beta_i=0$ for all $i \in \{1,\ldots,n\}$.
                    \Comment{Failsafe reduction to non-transporting step}
                \EndIf
            \Until{this iteration of the loop did not change any of the weights $\beta_i$.}
            \State
                Prune spikes with zero weights from the representation of $\nexxt\mu$.
                \label{step:fb:rough:prune}
            \State If desired, perform spike merging and other heuristics on $\nexxt\mu$.
            \Comment{See \cite{tuomov-pointsource}.}
        \EndFor
    \end{algorithmic}
\end{algorithm}

\begin{algorithm}[t]
    \caption{Point insertion and weight adjustment for $\Wave$-marginal term \cite{tuomov-pointsource}}
    \label{alg:fb:insert-and-remove}
    \begin{algorithmic}[1]
        \Require $\breve\mu \in \DiscreteMasses$, $\breve v \in \Predual$, $\alpha, \tau,\epsilon > 0$ on a domain $\Omega \subset \R^n$. A self-adjoint and positive semi-definite $\Wave \in \linear(\Masses; \DiffPredual)$.
        \Function{insert\_and\_adjust}{$\breve\mu$, $\breve v$, $\alpha$, $\tau$, $\epsilon$,  $\Wave$}
            \State Decompose $\sum_{x \in S} \alpha_x \delta_x \defeq \breve\mu$.
            \State Initialise $\mu \defeq \breve\mu$.
            \Repeat
                \State Form $\vec\eta \defeq ([\tau \breve v - \Wave\breve\mu](x))_{x \in S} \in \R^{\#S}$ and $D \defeq ([\Wave\delta_y](x))_{x, y \in S} \in \R^{\#S \times \#S}$.
                \State%
                Find $\vec\beta=(\beta_x)_{x \in S} \in \R^{\#S}$ solving $\min f$
                to the accuracy
                $
                    \inf_{g \in \subdiff f(\vec\beta)} \norm{g}_\infty \le \epsilon/(1 + \norm{\vec\beta}_1)
                $
                for
                \abovedisplayskip=5pt%
                \belowdisplayskip=-5pt%
                \[
                    f(\vec\beta) \defeq
                    \frac{1}{2}\iprod{\vec \beta}{D \vec \beta}
                    + \iprod{\vec \eta}{\vec \beta} + \tau\alpha \norm{\vec \beta}_1 + \delta_{\ge 0}(\vec\beta).
                \]
                \label{step:fb:insert-and-remove:findim}
                \State Let $\mu \defeq \sum_{x \in S} \beta_x \delta_x$
                \State Find $\opt x$ (approximately) minimising $\tau\breve v + \Wave(\mu-\breve\mu)$.
                \Comment{For example, branch-and-bound.}
                \State Let $S \defeq S \union \{\opt x\}$
                \Comment{$\optx$ will only be inserted into $\mu$ if the next bounds check fails.}
            \Until $\tau\breve v(\opt x ) + \theta + [\Wave(\mu-\breve\mu)](\opt x) \ge -\epsilon$
            \State \Return $\mu$
        \EndFunction
    \end{algorithmic}
\end{algorithm}

\begin{algorithm}[t]
    \caption{Point insertion and weight adjustment for Radon marginal term}
    \label{alg:fb:insert-and-remove-radon}
    \begin{algorithmic}[1]
        \Require $\breve\mu \in \DiscreteMasses$, $\breve v \in \Predual$, $\tau,\alpha,\epsilon > 0$ on a domain $\Omega \subset \R^n$.
        \Function{insert\_and\_adjust\_radon}{$\breve\mu$, $\breve v$, $\alpha$, $\tau$, $\epsilon$}
            \State Find $\opt x$ (approximately) minimising $\breve v$.
                \Comment{For example, branch-and-bound.}
            \State Decompose $\sum_{x \in S} \theta_x \delta_x \defeq \breve\mu + 0\delta_{\opt x}$.
            \State Form $\vec\eta \defeq (\tau \breve v(x))_{x \in S} \in \R^{\#S}$.
            \State Find $\vec\beta=(\beta_x)_{x \in S} \in \R^{\#S}$ solving $\min f$
            to the accuracy\label{step:fb:insert-and-remove-radon:sub}
            $
                \inf_{g \in \subdiff f(\vec\beta)} \norm{g}_\infty \le \epsilon/(1 + \norm{\vec \beta}_1)
            $
            for
            \abovedisplayskip=5pt
            \belowdisplayskip=5pt
            \[
                f(\vec\beta) \defeq
                \frac{1}{2}\norm{\vec \beta - \vec\theta}_1^2
                + \iprod{\vec \eta}{\vec \beta} + \tau\alpha \norm{\vec \beta}_1 + \delta_{\ge 0}(\vec\beta).
            \]
            \State \Return $\mu \defeq \sum_{x \in S} \beta_x \delta_x$
        \EndFunction
    \end{algorithmic}
\end{algorithm}

We present these observations as \cref{alg:fb:rough}. We restrict attention to  $G=\alpha \norm{\freevar}_\Meas + \delta_{\ge 0}$, although the extension to  $G=\alpha \norm{\freevar}_\Meas$ merely requires altering the subroutine for the marginal step \eqref{eq:fb:oc0:marginal}.
By construction, the algorithm ensures \cref{eq:fb:oc0,eq:fb:remainder} on each step $k \in \N$.
Moreover, by the failsafe mechanism in the algorithm, the loop on \cref{step:fb:rough:transport-adaptation} terminates in at most $N_\gamma$ steps, irrespective of how $\nexxt\gamma$ is reduced to satisfy \eqref{eq:fb:remainder}.
\Cref{alg:fb:rough} remains rough in the sense that it does not indicate the exact way to reduce the weights $\beta_i$ on \cref{step:fb:rough:reduce} to satisfy \eqref{eq:fb:remainder}.
A simple approach is to set all weights to zero if the initial attempt fails.
We treat another possibility in \cref{sec:controls}, after proving convergence.

When $E=E_\Wave$ for a self-adjoint and positive semi-definite $\Wave \in \linear(\Masses; \DiffPredual)$, we use \cref{alg:fb:insert-and-remove} from \cite{tuomov-pointsource} for \eqref{eq:fb:oc0:marginal}.
That method inserts, possibly multiple, new spikes $\beta_{\bar x}\delta_{\bar x}$ into $\this{\breve\mu}$ to form $\nexxt\mu$, and, as a finite-dimensional convex subproblem, optimises the weights of the spikes, until the condition is satisfied.
The algorithm, including modifications to satisfy \eqref{eq:fb:oc0:marginal} for $G=\alpha\norm{\freevar}_\Meas$ without the positivity constraint, are further discussed in \cite{tuomov-pointsource}.

When $E=E_\Meas$, the subderivative $\nexxt\omega$ has enough degrees of freedom to satisfy \eqref{eq:fb:oc0:marginal} with (at most) a single insertion, as the next lemma shows.

\begin{lemma}Let  $G=\alpha \norm{\freevar}_\Meas + \delta_{\ge 0}$ and $E=E_\Meas$.
    Then, \cref{alg:fb:insert-and-remove-radon} solves \eqref{eq:fb:oc0:marginal}, provided $\this{\breve\mu} \in \DiscreteMassesNonNeg $ (as ensured by \cref{alg:fb:rough}).
\end{lemma}

\begin{proof}[Sketch of proof]
    We consider exact minimisers $\bar x$ in \cref{alg:fb:insert-and-remove-radon}; inexact minimisers can be treated by more careful analysis of the tolerances.
    We have
    \begin{equation}
        \label{eq:fb:omega-radon}
        \begin{split}
        \nexxt\omega
        &
        \in \subdiff E(\this\mu-\pi_\#^0\nexxt\gamma, \freevar)(\nexxt\mu-\pi_\#^1\nexxt\gamma)
        \\
        &
        =
        \{
            \norm{\nexxt\mu-\this{\breve\mu}}_\Meas \phi
            \mid
            \phi \in \DiffPredual,\,
            \norm{\phi}_\infty \le 1,\,
            \dualprod{\nexxt\mu-\this{\breve\mu}}{\phi}=\norm{\nexxt\mu-\this{\breve\mu}}_\Meas
        \}.
        \end{split}
    \end{equation}
    Write $n \defeq \norm{\nexxt\mu-\this{\breve\mu}}_\Meas$, and, for brevity $\epsilon \defeq \nexxt\epsilon$.
    By assumption, $\this{\breve\mu} \ge 0$. Moreover, $\nexxt w \in \subdiff G(\nexxt\mu) \ne \emptyset$ exists if and only if $\nexxt\mu \ge 0$.
    Therefore, writing $A^c \defeq \Omega \setminus A$ for the $\Omega$-complement of a set $A$, four cases can occur when we expand \eqref{eq:fb:oc0:marginal}:
    \begin{enumerate}[label=(\alph*),nosep]
        \item\label{item:sub:radon-alg-just:insertion}
        Insertion or weight increase: $x \in \supp\nexxt \mu \isect \supp(\nexxt\mu-\this{\breve\mu})^+$. Then $-\epsilon \le \tau(\this{\breve v}(x) + \alpha) + n \le \epsilon$.

        \item\label{item:sub:radon-alg-just:decrease}
        Weight decrease: $x \in \supp(\nexxt\mu-\this{\breve\mu})^-$:
        Then $0 \in \tau\this{\breve v}(x) + (-\infty, \tau\alpha] - n + [-\epsilon,\epsilon]$, i.e., $n \le \tau(\this{\breve v}(x) + \alpha) + \epsilon$.

        \item\label{item:sub:radon-alg-just:nochange-nomass}
        No change, no mass: $x \in \supp(\nexxt\mu-\this{\breve\mu})^c \isect (\supp\nexxt \mu)^c$:
        Then $0 \in \this{\breve v}(x) + (-\infty, \alpha] + [-n, n] + [-\epsilon,\epsilon]$, i.e., $\tau(\alpha + \this{\breve v}(x)) \ge -n-\epsilon$.

        \item\label{item:sub:radon-alg-just:nochange-mass}
        No change, mass: $x \in \supp(\nexxt\mu-\this{\breve\mu})^c \isect \supp\nexxt \mu$:
        Then $0 \in \this{\breve v}(x) + \alpha + [-n, n] + [-\epsilon,\epsilon]$, i.e., $-n-\epsilon \le \tau(\alpha + \this{\breve v}(x)) \le n + \epsilon$.
    \end{enumerate}
    Thus, inserting a new spike at a minimiser $\opt x$ of $\this{\breve v}$, as done by \cref{alg:fb:insert-and-remove-radon}, we must by \cref{item:sub:radon-alg-just:insertion} have $n \in -\tau(\alpha + \this{\breve v}(\opt x)) + [-\epsilon,\epsilon]$.
    Then, by \cref{item:sub:radon-alg-just:nochange-nomass}, \eqref{eq:fb:oc0:marginal} holds at any point outside $\supp\nexxt \mu$, while either \cref{item:sub:radon-alg-just:decrease} or \cref{item:sub:radon-alg-just:nochange-mass} can be satisfied in $\supp\nexxt\mu \setminus\{\opt x\}$ by controlling the weights.
    This is handled by the finite-dimensional weight optimisation step of \cref{alg:fb:insert-and-remove-radon}, which solves \eqref{eq:fb:oc0:marginal} on $S=\supp\nexxt\mu$, once the latter is determined.
\end{proof}

\begin{remark}[Admissible transport plans]
    \label{rem:fb:admissible}
    \Cref{alg:fb:rough} ensures $\pi_\#^0\nexxt\gamma \ll \abs{\this\mu}$, more precisely $0 \le d\pi_\#^0\nexxt\gamma/d\this\mu \le 1$.
    If we further enforce (around \cref{step:fb:rough:reduce}) $\beta_i=0$ if $x_i \not \in \supp\nexxt\mu$, then also $\pi_\#^1\nexxt\gamma \ll \abs{\nexxt\mu}$.
    Therefore, $\nexxt\gamma \in \TwoPlans(\this\mu, \nexxt\mu) \subsetneq \TwoPlansSpace$ for the admissible set of transport plans (see \cref{sec:transport-new:definitions})
    \[
        \TwoPlans(\mu, \nu) \defeq \{
            \gamma \in \TwoPlansSpace
            \mid
            \pi_\#^0\gamma \ll \mu,\,
            \pi_\#^1\gamma \ll \nu,\,
            \norm{\nexxt\gamma} \le \norm{\this\mu}
        \}
    \]
    However, as we do not work with the marginalised transport costs $\UMK_{c_2,E,\TwoPlans}$, it is also not necessary to work with admissible transport plans.
\end{remark}

\subsection{Assumptions}
\label{sec:fb:assumptions}

In the next subsection we prove a very weak but easily obtained form of convergence for \cref{alg:fb:rough}, in implicit form \cref{eq:fb:oc0,eq:fb:remainder}.
The following collects our core assumptions, shared with \cref{sec:sub}:

\begin{assumption}\label{ass:fb:all}
    We have:
    \begin{enumerate}[label=(\roman*)]
        \item\label{item:fb:all:e}
        \textbf{Convex energy:}
        $E: \Masses \times \Masses \to [0, \infty]$ is convex in the second parameter and $E(\nu,\nu)=0$ for all $\nu \in \Masses$.
        \item\label{item:fb:all:g}
        \textbf{Convex regulariser:}
        $G: \Masses \to \extR$ is convex, proper, and lower semicontinuous.
        \item\label{item:fb:all:f}
        \textbf{Smooth data term:}
        $F: \Masses \to \R$ is Fréchet differentiable with $F'(\mu) \in \DiffPredual$ for all $\mu$, and $(L,\ell)$-smooth with respect to $B_J$ and $c_2$, i.e., for all $\mu,\nu \in \Masses$ and $\gamma \in \TwoPlansSpace$ we have
        \[
            B_F(\mu + (\pi_\#^1-\pi_\#^0)\gamma, \nu)
            \le
            V_{\ell c_2, L E}(\mu, \nu; \gamma).
        \]

        \item\label{item:fb:all:tolerances}
        \textbf{Vanishing tolerances:}
        The tolerances $\{\nexxt\epsilon\}_{k \in \N} \subset [0, \infty)$ satisfy
        $
            \lim_{N \to \infty} \frac{1}{N}\sum_{k=0}^{N-1} \nexxt\epsilon = 0.
        $
    \end{enumerate}
\end{assumption}

An advantage of $E_\Meas$ over $E_\Wave$ is that for the former, \cref{item:fb:all:f} holds if $F'$ is $L$-Lipschitz with respect to the Radon norm. The $\Wave$-marginal energy depends on a more difficult condition:

\begin{example}\label{ex:fb:assumptions-quadratic}
    Let $F(\mu)=\frac{1}{2}\norm{A\mu-b}_Y^2$ with $A \in \linear(\Masses; Y)$ for a Hilbert space $Y$, and $E=E_\Wave$.
    Then the $(L, \ell)$-smoothness of \cref{ass:fb:all}\,\cref{item:fb:all:f} is treated in \cref{ex:estimates:smoothness} and holds when $A_*A \le L \Wave$ and $\ell \ge 0$.
\end{example}

In this section, we require the following additions to \cref{ass:fb:all}.

\begin{assumption}\label{ass:fb:subdiff}
    The following hold:
    \begin{enumerate}[label=(\roman*)]
        \item
        \label{item:fb:subdiff:curvature}
        \textbf{Curvature lower bound:}
        For some $\ellF \ge 0$, for all $k \in \N$, we have
        $
        \ellF \abs{\nexxt\gamma}(c_2)
        \ge
        - B_F(\this{\breve \mu},\this\mu)
        $
        (e.g., $F$ is convex).

        \item\label{item:fb:subdiff:tau}
        \textbf{Step length bounds:}
        The step length parameters $\tau,\theta>0$ and the factors $L,\ell,\ellF,\ellCurvature$ satisfy $\tau L < 1$ and $\theta \tau[\ell+\ellCurvature+\ellF] < 2$.

        \item\label{item:fb:subdiff:e-f}
        \textbf{Marginal continuity:}
        $F'$ is continuous with respect to $E$ in the sense that,
        for any sequence $\{(\this\mu, \nexxt\gamma)\}_{k \in \N} \subset \Masses \times \TwoPlansSpace$,
        the convergence $E(\this\mu - \pi_\#^0\nexxt\gamma, \nexxt\mu - \pi_\#^1\nexxt\gamma) \to 0$
        implies $F'(\this\mu + (\pi_\#^1-\pi_\#^0)\nexxt\gamma) - F'(\nexxt\mu) \to 0$.

        \item\label{item:fb:subdiff:e-fenchel}
        \textbf{Inverse continuity of conjugate energy:}
        For any sequence $\{\this\omega\}_{k \in \N} \subset \Predual$, the convergence $\marginalEnergy{k}^*(\this\omega) \to 0$ implies $\norm{\this\omega}_{\Predual} \to 0$.

        \item\label{item:fb:subdiff:f}
        \textbf{Objective lower bound:}
        $\inf [F+G] > -\infty$.
    \end{enumerate}
\end{assumption}

Besides convex $F$, examples of the curvature lower bound are provided in \cref{lemma:unbalanced:example-lower-curvature}.

\begin{example}For $E(\nu,\mu)=\frac{1}{2}\norm{\freevar}_\Wave^2$, where $\Wave \in \linear(\Masses; \Predual)$ is self-adjoint and positive semi-definite, we have $\marginalEnergy{k}(\mu)=\frac{1}{2}\norm{\mu}_\Wave^2$.
    If $F(\mu)=\frac{1}{2}\norm{A\mu-b}_Y^2$ in a Hilbert space $Y$, then \cref{ass:fb:subdiff}\,\cref{item:fb:subdiff:e-f} holds under the assumption $A_*A \le L \Wave$ already seen in \cref{ex:fb:assumptions-quadratic}.
    Let then $c \defeq \norm{\Wave}_{\linear(\Masses; \Predual)}$.
    By the properties of conjugates (e.g., \cite[Lemmas 5.4 and 5.7]{clasonvalkonen2020nonsmooth}), we have
    \[
        \frac{2}{c}\norm{\omega}_{\Predual}^2
        =
        \frac{c}{2}\norm{\omega/c}_{\Predual}^2
        =
        \left(\frac{c}{2}\norm{\freevar}_{\Masses}^2\right)^*(\omega)
        \le
        \left(\frac{1}{2}\norm{\freevar}_{\Wave}^2\right)^*(\omega)
        = \marginalEnergy{k}^*(\omega).
    \]
    Thus, also \cref{ass:fb:subdiff}\,\cref{item:fb:subdiff:e-fenchel} holds.
\end{example}

\begin{example}For $E(\nu,\mu)=\frac{1}{2}\norm{\mu-\nu}_\Meas^2$, we have $\marginalEnergy{k}(\mu)=\frac{1}{2}\norm{\mu}_\Meas^2$ and $\marginalEnergy{k}^*(\omega)=\frac{1}{2}\norm{\omega}_{\Predual}^2$; see \cite[Lemma 5.4]{clasonvalkonen2020nonsmooth}.
    \Cref{ass:fb:subdiff}\,\cref{item:fb:subdiff:e-f,item:fb:subdiff:e-fenchel} clearly hold if $F'$ is, e.g., Lipschitz in the Radon norm.
\end{example}

Recalling the definition of $V$ in \cref{sec:transport-new:definitions}, we will measure the distance between iterates with
\begin{equation}
    \label{eq:fb:metric-x-y}
    \fbmetricM{k}(\mu, \nu; \gamma)
    \defeq
    V_{(2\inv\theta -\tau[\ell+\ellCurvature+\ellF])c_2, (1-\tau L)E}(\mu, \nu; \gamma).
\end{equation}
To measure the convergence of subdifferentials, we use the Fenchel pre-conjugate $\marginalEnergy{k}^*: \Predual \to \extR$ of
\begin{equation}
    \label{eq:fb:marginalenergy}
    \marginalEnergy{k}(\mu) \defeq E(\this\mu - \pi_\#^0\nexxt\gamma, \mu + \this\mu - \pi_\#^0\nexxt\gamma),
\end{equation}
already introduced in \cref{thm:unbalanced:twocost-three-point:reverse}.
The following bound is immediate:

\begin{lemma}\label{lemma:fb:v01-lower-bound}
    Suppose \cref{ass:fb:all}\,\cref{item:fb:all:e} and \ref{ass:fb:subdiff}\,\cref{item:fb:subdiff:tau,item:fb:subdiff:curvature} hold.
    Then $\marginalEnergy{k}^* \ge 0$ and
    \[
        \fbmetricM{k}(\this\mu, \nexxt\mu; \nexxt\gamma) \ge
        \epsilon V_{c_2,E}(\this\mu, \nexxt\mu; \nexxt\gamma)
        \ge \epsilon \MK_{c_2,E}(\this\mu, \nexxt\mu)
        \ge 0
    \]
    for $\epsilon \defeq \min\{1 - \tau L, 2\inv\theta - \tau[\ell + \ellF+\ellCurvature]\}>0$.
\end{lemma}

\begin{proof}As $\marginalEnergy{k}(0)=0$ by \cref{ass:fb:all}\,\cref{item:fb:all:e}, we have, from the definition of the Fenchel conjugate, that $\marginalEnergy{k}^* \ge 0$.
    The bounds on $\fbmetricM{k}$ follow from the step length conditions.
\end{proof}

\subsection{Monotonicity and subdifferential convergence}
\label{sec:fb:sub-convergence}

We can now prove a transport-adapted version of a standard quasi-monotonicity result.

\begin{lemma}[Value quasi-monotonicity]
    \label{lemma:fb:monotonicity}
    Suppose \cref{ass:fb:all}\,\cref{item:fb:all:e,item:fb:all:g,item:fb:all:f} and \ref{ass:fb:subdiff}\,\cref{item:fb:subdiff:curvature} hold, and that $k \in \N$ and $(\this\mu, \nexxt\gamma) \in \Masses \times \TwoPlansSpace$ are given.
    If \cref{eq:fb:oc0} holds, then
    \begin{equation}
        \label{eq:fb:monotonicity}
        \tau [F+G](\nexxt\mu)
        + \marginalEnergy{k}^*(\nexxt\omega)
        + \fbmetricM{k}(\this\mu, \nexxt\mu; \nexxt\gamma)
        \le
        \tau [F+G](\this\mu)
        + \nexxt{\check\remainder}.
    \end{equation}
\end{lemma}

\begin{proof}
    \begin{equation}
        \label{eq:fb:rearrangements0}
        \begin{split}
        \int_{\Omega^2}&\frac{1}{\theta} \iprod{y-x}{x-y} \d\abs{\nexxt\gamma}(x, y)
        =
        \int_{\Omega^2}
            - \tau \iprod{\grad \this v(x)}{x-y}
        \d\nexxt\gamma(x,y)
        \\
        &
        =
        \dualprod{\nexxt\omega}{(\pi_\#^0-\pi_\#^1)\nexxt\gamma}
        - \dualprod{\nexxt{\tilde\epsilon}}{\this\mu-\nexxt\mu}
        - \tau\ellCurvature\abs{\nexxt\gamma}(c_2)
        - \nexxt{\check\remainder}
        \\
        &
        =
        \dualprod{\nexxt\omega}{\nexxt\mu-\this\mu-(\pi_\#^1-\pi_\#^0)\nexxt\gamma}
        - \tau\dualprod{\this{\breve v}+\nexxt{w}}{\this\mu - \nexxt\mu}
        - \tau\ellCurvature\abs{\nexxt\gamma}(c_2)
        - \nexxt{\check\remainder}.
        \end{split}
    \end{equation}
    Combining \cref{thm:unbalanced:twocost-three-point:reverse,eq:fb:rearrangements0}, we then deduce
    \begin{equation}
        \label{eq:fb:monotonicity:0}
        \begin{split}
        0
        &
        \ge
        V_{2\inv\theta c_2, E}(\this\mu, \nexxt\mu; \nexxt\gamma)
        + \marginalEnergy{k}^*(\nexxt\omega)
        \\
        \MoveEqLeft[-1]
        -\frac{1}{\theta}\int_{\Omega^{2}} |x-y|^2\d\abs{\nexxt\gamma}(x,y)
        -\dualprod{\nexxt\omega}{\nexxt\mu - \this\mu-(\pi_\#^1-\pi_\#^0)\nexxt\gamma}
        \\
        &
        =
        V_{[2\inv\theta - \tau\ellCurvature]c_2, E}(\this\mu, \nexxt\mu; \nexxt\gamma)
        + \marginalEnergy{k}^*(\nexxt\omega)
        - \tau\dualprod{\this{\breve v} + \nexxt w}{\this\mu - \nexxt\mu}
        - \nexxt{\check\remainder}.
        \end{split}
    \end{equation}

    Using the definition of the Bregman divergence $B_F$ of $F$, and the convexity/subdifferentiability of $G$ (\cref{ass:fb:all}\,\cref{item:fb:all:g}), we get
    \[
        [F+G](\this\mu) - [F+G](\nexxt\mu)
        \ge
        \dualprod{\this{\breve v} + \nexxt w}{\this\mu-\nexxt\mu}
        + B_F(\this{\breve \mu},\this\mu)
        - B_F(\this{\breve \mu},\nexxt\mu).
    \]
    By \cref{ass:fb:all}\,\cref{item:fb:all:f} and \cref{ass:fb:subdiff}\,\cref{item:fb:subdiff:curvature}, we also have
    \begin{gather}
        \nonumber
        B_F(\this{\breve \mu},\nexxt\mu)
        \le
        V_{\ell c_2, LE}(\this\mu,\nexxt\mu; \nexxt\gamma)
        \quad\text{and}\quad
        B_F(\this{\breve \mu},\this\mu)
        \ge -\ellF\abs{\nexxt\gamma}(c_2),
    \shortintertext{so we get}
        \label{eq:fb:monotonicity:fg-est}
        \begin{split}
        [F+G](\this\mu) - [F+G](\nexxt\mu)
        &
        \ge
        \dualprod{\this{\breve v} + \nexxt w}{\this\mu-\nexxt\mu}
        - V_{[\ell +\ellF]c_2, LE}(\this\mu,\nexxt\mu; \nexxt\gamma).
        \end{split}
    \end{gather}
    We finish by adding this inequality multiplied by $\tau$ to \eqref{eq:fb:monotonicity:0}, and using the definition of $\fbmetricM{k}$ from \eqref{eq:fb:metric-x-y}.
\end{proof}

\begin{theorem}[Subdifferential convergence]
    \label{thm:fb:subdiff-conv}
    Suppose \cref{ass:fb:all,ass:fb:subdiff} hold, and that $\{(\this\mu, \nexxt\gamma)\}_{k \in \N}$ are generated through the satisfaction of \cref{eq:fb:oc0,eq:fb:remainder}.
    Then $\inf_{w \in \subdiff G(\nexxt\mu)} \norm{F'(\nexxt\mu)+w}_{\Predual} \to 0$.
\end{theorem}

\begin{proof}Let $N \in \N$. We apply \cref{lemma:fb:monotonicity} for all $k=0,\ldots,N-1$.
    Summing \eqref{eq:fb:monotonicity} over $k=0,\ldots,N-1$ and using $\nexxt{\check\remainder} \le \check C \nexxt\epsilon$ from \eqref{eq:fb:remainder} , we obtain
    \[
        \tau [F+G](\mu^N)
        + \sum_{k=0}^{N-1}\left(
            \marginalEnergy{k}^*(\nexxt\omega)
            + \fbmetricM{k}(\this\mu, \nexxt\mu; \nexxt\gamma)
        \right)
        \le
        \tau [F+G](\mu^0)
        + \check C \sum_{k=0}^{N-1} \nexxt\epsilon.
    \]
    By \cref{ass:fb:subdiff}\,\cref{item:fb:subdiff:f}, $[F+G](\mu^N) \ge \inf[F+G] > -\infty$.
    Minding \cref{ass:fb:all}\,\cref{item:fb:all:tolerances}, it follows for some constant $C$, independent of $N$, that
    \[
        \sum_{k=0}^{N-1}\left(
            \marginalEnergy{k}^*(\nexxt\omega)
            + \fbmetricM{k}(\this\mu, \nexxt\mu; \nexxt\gamma)
        \right)
        \le C.
    \]
    By \cref{lemma:fb:v01-lower-bound}, we have $\fbmetricM{k}(\this\mu, \nexxt\mu; \nexxt\gamma)\ge 0$ and $\marginalEnergy{k}^*(\nexxt\omega) \ge 0$.
    Thus, letting $N \to \infty$ above, we see that both converge to zero.
    By \cref{lemma:fb:v01-lower-bound}, this implies $V_{c_2, E}(\this\mu, \nexxt\mu; \nexxt\gamma) \to 0$, hence $E(\this\mu-\pi_\#^0\nexxt\gamma, \nexxt\mu-\pi_\#^1\nexxt\gamma) \to 0$.
    \Cref{ass:fb:subdiff}\,\cref{item:fb:subdiff:e-f} then implies $F'(\nexxt\mu)-F'(\this{\breve\mu}) \to 0$.
    By \eqref{eq:fb:oc0:marginal} we now have $\nexxt\omega + \tau[F'(\nexxt\mu) + \nexxt w] = \tau[F'(\nexxt\mu)-F'(\this{\breve\mu})] + \nexxt{\tilde\epsilon} \to 0$.
    As \cref{ass:fb:subdiff}\,\cref{item:fb:subdiff:e-fenchel} establishes $\nexxt\omega \to 0$, it follows that $F'(\nexxt\mu) + \nexxt w \to 0$. This implies the claim.
\end{proof}

The following corollary can be useful for verifying various assumptions in an inductive manner.

\begin{corollary}
    \label{cor:fb:bounded}
    Suppose \cref{ass:fb:all,ass:fb:subdiff} hold.
    Pick $N \in \N$, and for initial $\mu^0 \in \Masses$, generate $\{(\nexxt\mu, \nexxt\gamma)\}_{k=0}^{N-1}$ through the satisfaction of \cref{eq:fb:oc0,eq:fb:remainder}.
    If $\inf F \ge i_F > -\infty$ and $G \ge \phi(\norm{\mu})$ for a coercive $\phi: [0, \infty) \to [0, \infty)$, then $\sup_{k=0,\ldots,N-1} \norm{\nexxt\mu} \le m_\mu$ for a constant $m_\mu$ independent of $N$.
\end{corollary}

\begin{proof}The proof of \cref{thm:fb:subdiff-conv} establishes $\tau [F+G](\mu^N) \le \tau [F+G](\mu^0)  + \check C \sum_{k=0}^{N-1} \nexxt\epsilon$.
    The claim now follows after using \cref{ass:fb:all}\,\cref{item:fb:all:tolerances} and the bounds on $F$ and $G$.
\end{proof}

\begin{remark}[Line search and spike-specific step lengths]
    With an appropriate redefinition of $\fbmetricM{k}$, the proofs of \cref{lemma:fb:monotonicity,thm:fb:subdiff-conv,cor:fb:bounded} go through if $\theta=\theta_{x,k}$ depends on both the source point $x$ and the iteration $k$.
    Thus, it will be possible to spike-specific transport line search, in particular, to satisfy the remainder condition \eqref{eq:fb:remainder}.
    This will no longer be the case in \cref{sec:sub}.
\end{remark}

\section{Sublinear convergence of function values}
\label{sec:sub}

We continue from \cref{sec:fb} to show $O(1/N)$ convergence rates under additional technical requirements.
In particular, this section only applies to energies $E$ that satisfy the three-point inequality \cref{eq:unbalanced:energy}, hence the conclusion of  \cref{thm:unbalanced:twocost-three-point}.
Thus, this section does not apply to the Radon-norm energy $E_\Meas$.
Under suitable second-order growth assumptions, the work here could be extended to linear convergence.

We will work with three-plans $\nexxt\tritrans \in \ThreePlansSpace$ to transport between the triples $(\this\mu, \nexxt\mu, \opt\mu)$ for a reference point $\opt\mu$, usually a minimiser of $F+G$.
The idea is that the marginal $\pi_\#^0\nexxt\tritrans$ models the mass transported from $\this\mu$,
the marginal $\pi_\#^1\nexxt\tritrans$ models the mass transported to $\nexxt\mu$, and the marginal $\pi_\#^2\nexxt\tritrans$ models the connection to a reference point $\opt\mu$, typically a solution of the problem \eqref{eq:fb:problem}; see \cref{fig:weaving}.
The marginal $\pi_\#^{0,1}\nexxt\tritrans$ will include $\nexxt\gamma$, but also will contain additional terms to allow “weaving” $\this\tritrans$ into $\nexxt\tritrans$ in a way that works with the convergence proofs. This is the topic of \cref{sec:sub:weaving}, before the convergence proofs of \cref{sec:sub:convergence}.

\begin{figure}[t]
    \centering
    \includegraphics{illustr/weaving.pdf}
    \caption{%
        Illustration of the “weaving” of the three three-plan $\this\tritrans$ into $\nexxt\tritrans$.
        The green areas indicate the iterates $\this\mu$, the orange area the (unknown) solution $\hat\mu$, and the arrows the three-plans.
        The purple areas serve to illustrate the triangle of arrows formed by each $\this\tritrans$.
        The marginals $\this\gamma=\pi_\#^{0,1}\this\tritrans$.
    }
    \label{fig:weaving}
    \tikzexternaldisable
\end{figure}

\subsection{Distances, remainders, and weaving of three-plans}
\label{sec:sub:weaving}

We denote the set of three-plans compatible with a transport plan $\gamma \in \TwoPlansSpace$ by
\begin{equation}
    \label{eq:fb:ThreePlansCompatAlg}
    \ThreePlansCompatAlg(\gamma) \defeq \left\{
        \tritrans_0 + p_\#\tilde\gamma \in \ThreePlansSpace
        \middle|
        \begin{array}{l}
        \pi_\#^{0,1}\tritrans_0 = \gamma;\
        \tilde\gamma \in \TwoPlansSpace;\
        p(x,z)=(x,z,z)
        \text{ and }
        \\
        \sign\pi_\#^{0,1}\tritrans_0(x, y) =  \sign\tritrans_0(x,y,z)
        \text{ for } \abs{\tritrans_0}\text{-a.e. } (x,y,z)
        \end{array}
    \right\}.
\end{equation}
The restriction $\nexxt\tritrans \in \ThreePlansCompatAlg(\nexxt\gamma)$ forces the marginal $\pi_\#^{0,1}\nexxt\tritrans$ to equal $\nexxt\gamma$ plus a “return” of mass to $\opt\mu$.
It also ensures that $\sign\nexxt\gamma$ matches $\sign\nexxt\tritrans$.
The return-of-mass is required to satisfy the transition constraint  $\nexxt\tritrans \in \ThreePlansNext(\tritrans^k)$ for
\begin{equation}
    \label{eq:fb:ThreePlansNext}
    \ThreePlansNext(\tritrans) \defeq
    \left\{ \tilde\tritrans \in \ThreePlansSpace \,\middle|\,
        \begin{array}{l}
            \pi_\#^{0,2} \tilde\tritrans = \pi_\#^{1,2}\tritrans + \diag_\#\nu \text{ for some } \nu \in \Masses
            \\
            \pi_\#^{0,2} \abs{\tilde\tritrans} \le \pi_\#^{1,2}\abs{\tritrans} + \diag_\#\bar\nu \text{ for some } \bar\nu \in \Masses
        \end{array}
    \right\}.
\end{equation}
This enforces the agreement of the marginal $\pi_\#^{0,2}\nexxt\tritrans$ with the marginal $\pi_\#^{1,2}\this\tritrans$, up to a diagonal term, irrelevant in the definition of $V_{c_2,E}$.
Both marginals model transport from $\this\mu$ to $\opt\mu$.

The next lemma shows that these restrictions can be satisfied.

\begin{lemma}\label{lemma:fb:triplan-existence}
    There exists a $\nexxt\tritrans \in \ThreePlansCompatAlg(\nexxt\gamma) \isect \ThreePlansNext(\this\tritrans)$.
    Moreover, $\inf_{\mu \in \Masses} \norm{\pi_\#^{1,2}\nexxt\tritrans-\diag\mu} \le \norm{\nexxt\gamma}$.
\end{lemma}

\begin{proof}We decompose $\nexxt\gamma = \nexxt\gamma_+ - \nexxt\gamma_-$ and $\this\tritrans=\this\tritrans_+ - \this\tritrans_-$, where the component measures are non-negative.
    We will construct non-negative measures $\nexxt\tritrans_\pm \in \ThreePlansCompatAlg(\nexxt\gamma_\pm) \isect \ThreePlansNext(\this\tritrans_\pm)$, and then set $\nexxt\tritrans= \nexxt\tritrans_+ - \nexxt\tritrans_-$.
    Because the supports of $\nexxt\gamma_\pm$ are disjoint, as are those of $\this\tritrans_\pm$, it follows that the supports of $\nexxt\tritrans_\pm$ are disjoint.
    It follows that also $\nexxt\tritrans \in \ThreePlansCompatAlg(\nexxt\gamma) \isect \ThreePlansNext(\this\tritrans)$.
    For the rest of the proof, we may therefore assume, w.log, that $\nexxt\gamma \ge 0$ and $\this\tritrans \ge 0$.

    Abbreviate $\bar\gamma \defeq \pi_\#^{1,2}\this\tritrans$.
    By the Lebesgue decomposition and Radon–Nikodym theorems, we can write $\pi_\#^0\nexxt\gamma = f\pi_\#^0\bar\gamma + \nu_s$, where $\nu_s \perp \pi_\#^0\bar\gamma$ and $f \ge 0$ is Borel measurable.
    Let $\tilde\gamma \defeq \min\{1, \bar f\}\bar\gamma + \diag_\#\nu$ for $\bar f(x,y) \defeq f(x)$ and $\nu \defeq \pi_\#^0\nexxt\gamma - \min\{1,f\}\pi_\#^0\bar\gamma$.
    Then $\pi_\#^0\tilde\gamma=\pi_\#^0\nexxt\gamma$, so by \cite[Lemma 5.3.2]{ambrosio2008gradientflows}, there exists $0 \le \tritrans_0 \in \ThreePlansSpace$ with $\pi_\#^{0,2}\tritrans_0=\tilde\gamma$ and $\pi_\#^{0,1}\tritrans_0=\nexxt\gamma$.
    Let then $\tritrans \defeq \tritrans_0 + p_\#\max\{0, 1-\bar f\}\bar\gamma$ for $p(x,z)=(x,z,z)$.
    This readily gives the claimed norm bound.
    Moreover, by positivity, this construction satisfies $\tritrans \in \ThreePlansCompatAlg(\nexxt\gamma)$.
    Finally, we have
    $
        \pi_\#^{0,2}\tritrans=\bar\gamma + \diag_\#\nu,
    $
    which, by non-negativity, satisfies the definition of $\ThreePlansNext(\this\tritrans)$ in \eqref{eq:fb:ThreePlansNext}.
\end{proof}

\begin{remark}[Admissible transport plans]
    Define the admissible set of transport plans (see \cref{sec:transport-new:basic})
    \[
        \TwoPlans^{k+1}(\mu,\nu) \defeq \{
            \gamma \in \TwoPlansSpace
            \mid
            \pi_\#^0 \gamma \ll \abs{\mu} + \abs{\opt\mu},\,
            \pi_\#^1 \gamma \ll \abs{\nu} + \abs{\opt\mu},\,
            \norm{\gamma} \le \norm{\nexxt\gamma} + \norm{\pi_\#^{1,2}\this\tritrans}
        \}.
    \]
    If, as in \cref{rem:fb:admissible}, we enforce $\pi_\#^0\nexxt\gamma \ll \abs{\nexxt\mu}$, it is possible to see from the proof of \cref{lemma:fb:triplan-existence} that on iteration $k$, we have $\pi_\#^{0,1}\nexxt\tritrans \in \TwoPlans^{k+1}(\this\mu,\nexxt\mu)$, $\pi_\#^{0,2}\nexxt\tritrans \in \TwoPlans^{k+1}(\this\mu,\opt\mu)$, and $\pi_\#^{1,2}\nexxt\tritrans \in \TwoPlans^{k+1}(\nexxt\mu,\opt\mu)$ provided that the previous iteration already satisfies $\pi_\#^{1,2}\this\tritrans \in \TwoPlans^k(\this\mu,\opt\mu)$.
\end{remark}

The following will be useful to transform integrals with respect to $\tritrans$:

\begin{lemma}\label{lemma:sub:integrals}
    Let $\tritrans \in \ThreePlansCompatAlg(\gamma)$ and $f: \Omega^3 \to \R$ be $\abs{\tritrans}$-integrable with $f(x, y, d)=0$ when $d=0$. Then
    \[
        \int_{\Omega^3} f(x,y,z-y) \d\tritrans(x,y,z) = \int_{\Omega^3} f(x,y,z-y) \d\tritrans_0(x,y,z),
    \]
    and likewise for the total variation measures.
    Moreover,
    \[
        \int_{\Omega^3} f(x,y,z-y) \gamma(x,y) \d\abs{\tritrans}(x,y,z) = \int_{\Omega^3} f(x,y,z-y) \d\tritrans(x,y,z).
    \]
    If $f(x, y, d)=g(x, y)$, then, moreover,
    \[
        \int_{\Omega^3} g(x,y) \d\abs{\tritrans_0}(x,y,z) = \int_{\Omega^2} g(x,y) \d\abs{\gamma}(x,y).
    \]
\end{lemma}

\begin{proof}Let $p$ and $\tilde\gamma$ be as in \eqref{eq:fb:ThreePlansCompatAlg}.
    The first claim holds because $y=z$, hence $f(x,y,z-y)=0$, for $p_\#\tilde\gamma$-a.e. $(x,y,z)$.
    For the second claim, we observe using the first claim twice, and \eqref{eq:fb:ThreePlansCompatAlg}. that
    \[
        \begin{split}
        \int_{\Omega^3} f(x, y, z-y) \sign\gamma(x,y)\d\abs{\tritrans}(x,y,z)
        &
        =
        \int_{\Omega^3} f(x, y, z-y) \sign\pi_\#^{0,1}\tritrans_0(x,y)\d\abs{\tritrans_0}(x,y,z)
        \!\!\!\!
        \\
        \MoveEqLeft[6]
        =
        \int_{\Omega^3} f(x, y, z-y) \d\tritrans_0(x,y,z)
        =
        \int_{\Omega^3} f(x, y, z-y) \d\tritrans(x,y,z).
        \end{split}
    \]
    For the final claim, we use
    \begin{multline*}
        \int_{\Omega^2} g(x,y) \d\abs{\tritrans_0}(x,y,z)
        =\int_{\Omega^2} g(x,y) \sign\tritrans_0(x,y,z)\d\tritrans_0(x,y,z)
        \\
        =\int_{\Omega^2} g(x,y) \sign\pi_\#^{0,1}\tritrans_0(x,z)\d\pi_\#^{0,1}\tritrans_0(x,y)
        =\int_{\Omega^2} g(x,y) \d\abs{\pi_\#^{0,1}\tritrans_0}(x,y)
        =\int_{\Omega^2} g(x,y) \d\abs{\nexxt\gamma}(x,y).
        \qedhere
    \end{multline*}
\end{proof}

We need to modify the remainder condition \eqref{eq:fb:remainder} on $\nexxt{\check\remainder}$.

\begin{rawalg}
    Using the variables defined in \cref{eq:fb:oc0}, we introduce the generalised remainder
    \begin{equation}
        \label{eq:fb:r-def}
        \begin{split}
        \nexxt{\remainder}(\mu, \tritrans)
        &
        \defeq
        - \dualprod{\nexxt{\tilde\epsilon}}{\mu- \nexxt\mu}
        - \tau\ellCurvature\pi_\#^{0,1}\abs{\tritrans}(c_2)
        \\
        \MoveEqLeft[-1]
        +\int_{\Omega^3}\tau\iprod{\grad \this v(x)}{z-y}
        +\nexxt\omega(z)-\nexxt\omega(y)
        \d\tritrans(x, y, z).
        \end{split}
    \end{equation}
    For some $C_{\opt\mu} > 0$ and a comparison point $\opt\mu \in \Masses$ of interest (typically a minimiser), for a $\nexxt\tritrans \in \ThreePlansCompatAlg(\nexxt\gamma) \isect \ThreePlansNext(\this\tritrans)$, we then consider the \emph{remainder conditions}
    \begin{align}
        \label{eq:sub:remainder}
        &
        C_{\opt\mu} \nexxt\epsilon \ge \nexxt{\remainder}(\opt\mu, \nexxt\tritrans)
        &&
        \text{to replace \eqref{eq:fb:remainder} for ergodic convergence,}
        \\
        \label{eq:sub:remainder-tight}
        &
        C_{\opt\mu} \nexxt\epsilon \ge
        (k+1)\check\remainder^{k+1} + \nexxt{\remainder}(\opt\mu, \nexxt\tritrans)
        &&
        \text{to replace \eqref{eq:fb:remainder} for non-erg.~convergence}.
    \end{align}
\end{rawalg}

\begin{remark}
    \label{rem:fb:r-check-eq}
    Let
    $
        \nexxt\tritrans_{\reverse} \defeq q_\# \nexxt\gamma
        \in \ThreePlansCompatAlg(\nexxt\gamma)
    $
    for
    $
        q(x,y) \defeq (x,y,x).
    $
    Then $\nexxt{\check\remainder} = \nexxt\remainder(\this\mu, \nexxt\tritrans_{\reverse})$.
\end{remark}

On the other hand, we will not require the curvature lower bound \cref{ass:fb:subdiff}\,\cref{item:fb:subdiff:curvature} and the corresponding factor $\ell_F$, because we will for simplicity assume $F$ convex. This is contained in the following assumption that we require besides \cref{ass:fb:all}:

\begin{assumption}\label{ass:sub:energy}
    We have:
    \begin{enumerate}[label=(\roman*)]
        \item
        \label{item:sub:energy:f}
        \textbf{Convex data term:}
        $F$ is convex.

        \item
        \textbf{Step length bounds:}
        \label{item:sub:energy:tau}
        The step lengths parameters $\tau,\theta>0$ and the factors $L,\ell$, and $\ellCurvature$ satisfy $\tau L \le 1$ and $\theta \tau [\ell+\ellCurvature] \le 1$.

        \item
        \textbf{Sub-Pythagorean energy:}
        \label{item:sub:energy:e}
        $E$ satisfies \eqref{eq:unbalanced:energy}.
    \end{enumerate}
\end{assumption}
The final condition rules out the Radon norm marginal term $E_\Meas$.

Recalling the definition of $\bar V_{\inv\theta c_2, E}^{i,j}$ from \cref{thm:unbalanced:twocost-three-point}, we abbreviate the distances between $\mu^{k+i}$ and $\opt\mu$ for $i=0,1$ with
\begin{equation}
    \label{eq:sub:metric02-12}
    \fbmetricBar{k}^{i,2}(\mu, \nu; \tritrans)
    \defeq
    \bar V_{\inv\theta c_2, E}^{i,2}(\mu, \nu; \tritrans)
    =
    \inv\theta \pi_\#^{i,2}\abs{\tritrans}(c_2)
    + E(\mu - \pi_\#^i \tritrans, \nu - \pi_\#^2 \tritrans)
\end{equation}
For the distance between $\this\mu$ and $\nexxt\mu$, we
also introduce
\begin{equation}
    \label{eq:sub:metric01}
    \begin{split}
    \fbmetricL{k}(\mu, \nu; \tritrans)
    &
    \defeq
    \bar V_{\inv\theta c_2, E}^{0,1}(\mu, \nu; \tritrans)
    - \tau
            V_{\ell c_2, L E}(\mu, \nu; \pi_\#^{0,1}\tritrans)
    \\
    &
    =
    (\inv\theta -\tau[\ell+\ellCurvature])\pi_\#^{0,1}\abs{\tritrans}(c_2)
    + (1-\tau L) E(\mu - \pi_\#^0 \tritrans, \nu - \pi_\#^1 \tritrans).
    \end{split}
\end{equation}

The following lower bounds are immediate:

\begin{lemma}\label{lemma:sub:v-lower-bound}
    Suppose \cref{ass:fb:all}\,\cref{item:fb:all:e} and \ref{ass:sub:energy}\,\cref{item:sub:energy:tau} hold.
    Then, for any $k \in \N$ and $\nu \in \Masses$, we have
    \begin{enumerate}[label=(\roman*)]
        \item\label{item:sub:v-lower-bound:opt}
        $\fbmetricBar{k}^{0,2}(\this\mu, \nu; \nexxt\tritrans) \ge 0$.
        \item\label{item:sub:v-lower-bound:iter}
        $\fbmetricL{k}(\this\mu, \nexxt\mu; \nexxt\tritrans) \ge 0$.
        \item\label{item:sub:v-lower-bound:m}
        $\fbmetricM{k}(\this\mu, \nexxt\mu; \nexxt\gamma) \ge 0$.
    \end{enumerate}
\end{lemma}

The weaving compatibility condition  $\nexxt\tritrans \in \ThreePlansNext(\this\tritrans)$ allows us to rebalance the three-plan before commencing with the next step, as follows:

\begin{assumption}[Rebalancing]
    \label{ass:sub:rebalancing}
    For all $k \in \N$, for given $\nexxt\mu \in \Masses$, $ \nexxt\gamma \in \TwoPlansSpace$, and $\nexxt\tritrans \in \ThreePlansCompatAlg(\nexxt\gamma)$,
    the choice of
    $\gamma^{k+2} \in \TwoPlansSpace$ and $\tritrans^{k+2} \in \ThreePlansCompatAlg(\gamma^{k+2}) \isect \ThreePlansNext(\nexxt\tritrans)$
    is such that
    \begin{equation}
        \label{eq:fb:rebalancing}
        \fbmetricBar{k}^{1,2}(\nexxt\mu, \opt\mu; \nexxt\tritrans)
        \ge
        \fbmetricBar{k+1}^{0,2}(\nexxt\mu, \opt\mu; \tritrans^{k+2})
        \quad\text{for all}\quad
        \opt\mu \in \Masses.
    \end{equation}
\end{assumption}

The next lemma proves \cref{ass:sub:rebalancing} for our typical choices of $E$.

\begin{lemma}\label{lemma:fb:rebalancing}
    Let $E=E_\Wave$ or $E=E_\Wave$ for a self-adjoint $\Wave \in \linear(\Masses; \Predual)$.
    Then \cref{ass:sub:rebalancing} holds for any choice of $\gamma^{k+2} \in \TwoPlansSpace$ and $
        \tritrans^{k+2} \in \ThreePlansCompatAlg(\gamma^{k+2})
        \isect \ThreePlansNext(\nexxt\tritrans).
    $
\end{lemma}

\begin{proof}Write $\star=\Wave$ or $\star=\Meas$.
    Inserting \eqref{eq:sub:metric02-12} into \eqref{eq:fb:rebalancing} and expanding, we need to show that
    \begin{equation}
        \label{eq:fb:rebalancing:0}
        \begin{split}
        0
        &
        \le
        \int_{\Omega^2} \frac{1}{\theta} c_2(x, z) \d(\pi_\#^{1,2}\abs{\nexxt\tritrans}-\pi_\#^{0,2}\abs{\tritrans^{k+2}})(x, z)
        \\
        \MoveEqLeft[-1]
        +
        \frac{1}{2}\norm{\opt\mu-\nexxt\mu-(\pi_\#^2-\pi_\#^1)\nexxt\tritrans}_\star^2
        -
        \frac{1}{2}\norm{\opt\mu-\nexxt\mu-(\pi_\#^2-\pi_\#^0)\tritrans^{k+2}}_\star^2.
        \end{split}
    \end{equation}
    By the definition of $\ThreePlansNext(\nexxt\tritrans) \ni \tritrans^{k+2}$ in \eqref{eq:fb:ThreePlansNext}, we have $\pi_\#^{0,2}\tritrans^{k+2}=\pi_\#^{1,2}\nexxt\tritrans+\diag\nu$ and $\pi_\#^{0,2}\abs{\tritrans^{k+2}} \le \pi_\#^{1,2}\abs{\nexxt\tritrans}+\diag\bar\nu$ for some $\nu,\bar\nu \in \Masses$, hence also $(\pi_\#^2-\pi_\#^0)\tritrans^{k+2}=(\pi_\#^2-\pi_\#^1)\nexxt\tritrans$.
    Since the diagonal measures do not contribute to any term of \eqref{eq:fb:rebalancing:0}, this shows that the first term is non-negative, and the rest zero.
\end{proof}

\subsection{Convergence of function values}
\label{sec:sub:convergence}

We proceed with a transport-aware version of a standard descent estimate.

\begin{lemma}[Descent estimate to a reference point]
    \label{lemma:fb:fejer}
    Suppose \cref{ass:fb:all}\,\cref{item:fb:all:e,item:fb:all:g,item:fb:all:f} and \ref{ass:sub:energy}\,\ref{item:sub:energy:e} hold, and that $k \in \N$ and $(\this\mu, \nexxt\gamma) \in \Masses \times \TwoPlansSpace$ are given.
    If \cref{eq:fb:oc0} holds, then for any $\mu \in \Masses$ and $\tritrans \in \ThreePlansCompatAlg(\nexxt\gamma)$,
    \begin{equation}
        \label{eq:fb:descent:alt}
        \tau [F+G](\mu) - \tau [F+G](\nexxt\mu)
        \ge
        \fbmetricBar{k}^{1,2}(\nexxt\mu, \mu; \tritrans)
        - \fbmetricBar{k}^{0,2}(\this\mu, \mu; \tritrans)
        + \fbmetricL{k}(\this\mu, \nexxt\mu; \tritrans)
        - \nexxt{\remainder}(\mu, \tritrans).
    \end{equation}
\end{lemma}

\begin{proof}Using \cref{lemma:sub:integrals}, \cref{eq:fb:oc0:transport,eq:fb:oc0:marginal,eq:fb:r-def}, similarly to \eqref{eq:fb:rearrangements0}, we get
    \begin{equation}
        \label{eq:fb:rearrangements}
        \begin{split}
        \int_{\Omega^3}&\frac{1}{\theta} \iprod{y-x}{z-y} \d\abs{\tritrans}(x, y, z)
        =
        - \int_{\Omega^3}\iprod{\grad \this v(x)}{z-y} \nexxt\gamma(x, y) \d\abs{\tritrans}(x, y, z)
        \\
        &
        =
        -\tau \int_{\Omega^3} \iprod{\grad \this v(x)}{z-y} \d\tritrans(x, y, z)
        \\
        &
        =
        \dualprod{\nexxt\omega}{(\pi_\#^2-\pi_\#^1)\tritrans}
        - \dualprod{\nexxt{\tilde\epsilon}}{\mu- \nexxt\mu}
        - \tau\ellCurvature\pi_\#^{0,1}\abs{\tritrans}(c_2)
        - \nexxt\remainder(\mu, \tritrans)
        \\
        &
        =
        - \dualprod{\nexxt\omega}{\mu- \nexxt\mu-(\pi_\#^2-\pi_\#^1)\tritrans}
        - \tau\dualprod{\this{\breve v}+\nexxt{w}}{\mu - \nexxt\mu}
        - \tau\ellCurvature\pi_\#^{0,1}\abs{\tritrans}(c_2)
        - \nexxt\remainder(\mu, \tritrans).
        \end{split}
    \end{equation}
    \Cref{thm:unbalanced:twocost-three-point} establishes
    \begin{multline*}
        \bar V_{\inv\theta c_2, E}^{0,2}(\this\mu, \mu; \tritrans)
        - \bar V^{0,1}_{\inv\theta c_2, E}(\this\mu, \nexxt\mu; \tritrans)
        \\
        \ge
        \frac{1}{\theta}\int_{\Omega^3} \iprod{y-x}{z-y} \d\abs{\tritrans}(x, y, z)
        +\dualprod{\nexxt\omega}{\mu- \nexxt\mu-(\pi_\#^2-\pi_\#^1)\tritrans}
        + \bar V_{\inv\theta c_2, E}^{1,2}(\nexxt\mu, \mu; \tritrans).
    \end{multline*}
    Combining this with \cref{eq:fb:rearrangements}, we get
    \begin{equation}
        \label{eq:fb:descent:1:alt}
        \begin{split}
        0
        &
        \ge
        \bar V_{\inv\theta c_2, E}^{1,2}(\nexxt\mu, \mu; \tritrans)
        - \bar V_{\inv\theta c_2, E}^{0,2}(\this\mu, \mu; \tritrans)
        +\bar V^{0,1}_{[\inv\theta - \tau\ellCurvature]c_2, E}(\this\mu, \nexxt\mu; \tritrans)
        \\
        \MoveEqLeft[-1]
        - \tau\dualprod{\this{\breve v}+\nexxt{w}}{\mu - \nexxt\mu}
        - \nexxt\remainder(\mu, \tritrans).
        \end{split}
    \end{equation}
    Using the definitions \cref{eq:sub:metric01,eq:sub:metric02-12}, we obtain
    \begin{equation}
        \label{eq:fb:descent:2:alt}
        \begin{split}
        0
        &
        \ge
        \fbmetricBar{k}^{1,2}(\nexxt\mu, \mu; \tritrans)
        - \fbmetricBar{k}^{0,2}(\this\mu, \mu; \tritrans)
        + \fbmetricBar{k}^{0,1}(\this\mu, \nexxt\mu; \tritrans)
        - \nexxt\remainder(\mu, \tritrans).
        \\
        \MoveEqLeft[-1]
        - \tau\dualprod{\this{\breve v}+\nexxt{w}}{\mu - \nexxt\mu}
        + V_{\ell c_2, LE}(\this\mu,\nexxt\mu; \nexxt\gamma)
        \end{split}
    \end{equation}
    Analogously to \eqref{eq:fb:monotonicity:fg-est}, we prove
    \begin{equation}
        \label{eq:fb:descent:fg-est}
        [F+G](\mu) - [F+G](\nexxt\mu)
        \ge
        \dualprod{\this{\breve v} + \nexxt w}{\mu-\nexxt\mu}
        + B_F(\this{\breve \mu},\mu)
        - V_{\ell c_2, LE}(\this\mu,\nexxt\mu; \nexxt\gamma).
    \end{equation}
    Since  $B_F(\this{\breve \mu},\mu) \ge 0$ by the convexity\footnote{%
        This is the only place where we require the convexity.
        With some effort, the assumption could be made local; compare \cite{tuomov2024tracking}.
    } of $F$, the claim now follows by combining \cref{eq:fb:descent:fg-est} multiplied by $\tau$ with \cref{eq:fb:descent:2:alt}.
\end{proof}

We are now ready to state our main convergence theorems \cref{alg:fb:rough}, based on the satisfaction of \cref{eq:fb:oc0} and the remainder condition \eqref{eq:sub:remainder} or \eqref{eq:sub:remainder-tight}.

\begin{theorem}[Ergodic function value convergence]
    \label{thm:sub:convergence:function-ergodic}
    Suppose \cref{ass:fb:all,ass:sub:energy,ass:sub:rebalancing} hold, and let $\opt\mu \in \Masses$.
    For an initial $\mu^0 \in \Masses$, generate $\{(\nexxt\mu, \nexxt\gamma)\}_{k \in \N}$ through the satisfaction of \cref{eq:fb:oc0,eq:sub:remainder}.
    Pick $N \in \N$ and $\tritrans^1 \in \ThreePlansCompatAlg(\gamma^1)$.
    Then for $\tilde\mu^N \defeq \frac{1}{N}\sum_{k=0}^{N-1}\mu^{k+1}$, we have
    \begin{equation}
        \label{eq:fb:convergence:function-ergodic}
        [F+G](\tilde\mu^N)
        - [F+G](\opt\mu)
        \le
        \frac{1}{\tau N} \fbmetricBar{0}^{0,2}(\mu^0, \opt\mu;  \tritrans^1)
        + \frac{\constCurvature}{\tau N}
        + \frac{C_{\opt\mu}}{\tau N} \sum_{k=0}^{N-1} \nexxt\epsilon
    \end{equation}
    In particular, $[F+G](\tilde\mu^N) \to \inf [F+G]$ at the rate $O(1/N)$.
\end{theorem}

\begin{proof}We have assumed to be given some $\tritrans^1 \in \ThreePlansCompatAlg(\gamma^1)$; for example, $\tritrans^1=\tritrans^1_{\reverse}$, as constructed in \cref{lemma:fb:fejer}.
    For all $k \ge 1$, we let $\nexxt\tritrans \in \ThreePlansCompatAlg(\nexxt\gamma) \isect \ThreePlansNext(\this\tritrans)$ be determined by \eqref{eq:sub:remainder} and the rebalancing \cref{ass:sub:rebalancing}.

    Pick any $k \in \{0,\ldots,N-1\}$.
    We apply \cref{lemma:fb:fejer} with $\mu=\opt\mu$ and $\tritrans = \nexxt\tritrans$ to obtain \cref{eq:fb:descent:alt}.
    Using \cref{ass:sub:rebalancing} and \cref{lemma:sub:v-lower-bound}\,\cref{item:sub:v-lower-bound:iter} there, we obtain
    \begin{equation}
        \label{eq:fb:convergence:onestep}
        \tau[F+G](\opt\mu) - \tau[F+G](\nexxt\mu)
        \ge
        \fbmetricBar{k+1}^{0,2}(\nexxt\mu, \opt\mu; \tritrans^{k+2})
        - \fbmetricBar{k}^{0,2}(\this\mu, \opt\mu; \nexxt\tritrans)
        - \nexxt{\remainder}(\opt\mu, \nexxt\tritrans).
    \end{equation}
    Using \eqref{eq:sub:remainder} and summing over $k=0,\ldots,N-1$ in \eqref{eq:fb:convergence:onestep}, we thus obtain
    \[
        \fbmetricBar{N}^{0,2}(\mu^N, \opt\mu; \tritrans^{N+1})
        + \sum_{k=0}^{N-1} \tau [F+G](\nexxt\mu)
        \le
        \fbmetricBar{0}^{0,2}(\mu^0, \opt\mu; \tritrans^1)
        + \tau N [F+G](\opt\mu)
        + \sum_{k=0}^{N-1} C_{\opt\mu} \nexxt\epsilon.
    \]
    An application of \cref{lemma:sub:v-lower-bound}\,\cref{item:sub:v-lower-bound:opt} and Jensen's inequality then establishes \eqref{eq:fb:convergence:function-ergodic}. \Cref{ass:fb:all}\,\cref{item:fb:all:tolerances} now establishes the convergence rate claim.
\end{proof}

\begin{corollary}[Function value convergence]
    \label{cor:sub:convergence:function}
    Suppose \cref{ass:fb:all,ass:sub:energy,ass:sub:rebalancing} hold, and let $\opt\mu \in \Masses$.
    For an initial $\mu^0 \in \Masses$, generate $\{(\nexxt\mu, \nexxt\gamma)\}_{k \in \N}$ through the satisfaction of \cref{eq:fb:oc0,eq:sub:remainder-tight}.
    Then for any $N \in \N$ and $\tritrans^1 \in \ThreePlansCompatAlg(\gamma^1)$, we have
    \[
        [F+G](\mu^N)
        - [F+G](\opt\mu)
        \le
        \frac{1}{N\tau} \fbmetricBar{0}^{0,2}(\mu^0, \opt\mu; \tritrans^1)
        + \frac{\constCurvature}{\tau N}
        + \frac{C_{\opt\mu}}{\tau N}\sum_{k=0}^{N-1}
            \nexxt\epsilon.
    \]
    In particular, $[F+G](\mu^N) \to \inf [F+G]$ at the rate $O(1/N)$.
\end{corollary}

\begin{proof}We follow the proof of \cref{thm:sub:convergence:function-ergodic} until \eqref{eq:fb:convergence:onestep}.
    Note that \cref{ass:sub:energy}\,\cref{item:sub:energy:f,item:sub:energy:tau} ensure \cref{ass:fb:subdiff}\,\cref{item:fb:subdiff:curvature,item:fb:subdiff:tau} because $\ellF=0$.
    By \cref{lemma:fb:monotonicity,lemma:fb:v01-lower-bound}, we then have
    \[
        \tau [F+G](\nexxt\mu) \le \tau [F+G](\this\mu) + \nexxt{\check\remainder}.
    \]
    We may repeatedly apply this inequality in \eqref{eq:fb:convergence:onestep} and sum over $k=0,\ldots,N-1$ to obtain
    \begin{gather}
        \label{eq:fb:convergence:function:0}
        \fbmetricBar{N}^{0,2}(\mu^N, \opt\mu; \tritrans^{N+1})
        + b_N
        + N\tau [F+G](\mu^N)
        \le
        \fbmetricBar{0}^{0,2}(\mu^0, \opt\mu; \tritrans^1)
        + N\tau [F+G](\opt\mu)
        + a_N
    \shortintertext{for}
        \nonumber
        a_N
        \defeq
        \sum_{k=0}^{N-1} \left(
            \sum_{j=k+1}^{N-1} \check\remainder^{j+1}
            + \nexxt{\remainder}(\opt\mu, \nexxt\tritrans)
        \right)
        =
        \sum_{k=0}^{N-1} \left(
            (k+1)\check\remainder^{k+1}
            + \nexxt{\remainder}(\opt\mu, \nexxt\tritrans)
        \right)
    \shortintertext{and}
        \nonumber
        b_N
        \defeq
        \sum_{k=0}^{N-1}%
            \sum_{j=k+1}^{N-1} \fbmetricM{j}(\mu^j, \mu^{j+1}; \gamma^{j+1}).
    \end{gather}
    By \cref{lemma:fb:v01-lower-bound} and \ref{lemma:sub:v-lower-bound}\,\cref{item:sub:v-lower-bound:m}, we have $b_N \ge 0$.
    Dividing \eqref{eq:fb:convergence:function:0} by $N\tau$ and using \cref{eq:sub:remainder-tight,lemma:sub:v-lower-bound}\,\cref{item:sub:v-lower-bound:opt}, the claim thus follows.
\end{proof}

The following corollary will be useful for inductively verifying various assumptions.

\begin{corollary}
    \label{cor:sub:bounded}
    Suppose \cref{ass:fb:all,ass:sub:energy,ass:sub:rebalancing} hold, and let $\opt\mu \in \Masses$.
    Pick $N \in \N$, and for initial $\mu^0 \in \Masses$ and $\tritrans^1 \in \ThreePlansCompatAlg(\gamma^1)$, generate $\{(\nexxt\mu, \nexxt\gamma)\}_{k=0}^{N-1}$ through the satisfaction of \cref{eq:fb:oc0} and either \eqref{eq:sub:remainder} or \eqref{eq:sub:remainder-tight} with \cref{ass:fb:subdiff}\,\ref{item:fb:subdiff:tau}.
    If $\inf F \ge i_F > -\infty$ and $G \ge \phi(\norm{\mu})$ for a coercive $\phi: [0, \infty) \to [0, \infty)$, then $\sup_{k=0,\ldots,N-1} \norm{\mu^k} \le m$ for a constant $m$ independent of $N$.
\end{corollary}

\begin{proof}We show the case \eqref{eq:sub:remainder}.
    The case \eqref{eq:sub:remainder-tight} follows
    similarly from
    \eqref{eq:fb:convergence:function:0}.

    Using $[F+G](\this\mu) \ge [F+G](\opt\mu)$ in \eqref{eq:fb:convergence:onestep} in the proof of \cref{thm:sub:convergence:function-ergodic} establishes
    \[
        \tau[F+G](\nexxt\mu)
        +
        \fbmetric{k+1}(\nexxt\mu, \opt\mu; \pi_\#^{0,2}\tritrans^{k+2})
        \le
        \tau[F+G](\this\mu)
        + \fbmetric{k}(\this\mu, \opt\mu; \pi_\#^{0,2}\nexxt\tritrans)
        + \nexxt{\remainder}(\opt\mu, \nexxt\tritrans).
    \]
    Summing this over $k=0,\ldots,N-1$, using  \cref{ass:fb:all}\,\cref{item:fb:all:tolerances}, \cref{lemma:sub:v-lower-bound}\,\cref{item:sub:v-lower-bound:opt}, and the bounds on $F$ and $G$, we obtain
    \[
        \tau i_F - \constCurvature
        + \tau \phi(\norm{\mu^N}_\Meas)
        \le
        \tau[F+G](\mu^0)
        + \fbmetric{0}(\mu^0, \opt\mu; \pi_\#^{0,2}\tritrans^1)
        + \sum_{k=0}^{N-1} C\nexxt\epsilon.
    \]
    Minding \cref{ass:fb:all}\,\cref{item:fb:all:tolerances}, we thus establish $\norm{\mu^N} \le m$ for a constant $m$ independent of $N$.
\end{proof}

\section{Controls on the remainder}
\label{sec:controls}

To complete \cref{alg:fb:rough}, we now develop explicit ways to enforce the remainder conditions
of \cref{thm:fb:subdiff-conv,thm:sub:convergence:function-ergodic,cor:sub:convergence:function} through simple bounds and aposteriori checks.
The approach we present here is just one of many possibilities.

\subsection{General scheme}
\label{sec:controls:general}

Recall from \eqref{eq:fb:r-def} that
\begin{equation}
    \label{eq:sub:controls:r-def}
    \begin{split}
    \nexxt{\remainder}(\mu, \tritrans)
    &
    \defeq
    - \dualprod{\nexxt{\tilde\epsilon}}{\mu- \nexxt\mu}
    - \tau\ellCurvature\pi_\#^{0,1}\abs{\tritrans}(c_2)
    \\
    \MoveEqLeft[-1]
    +\int_{\Omega^3}\tau\iprod{\grad \this v(x)}{z-y}
    +\nexxt\omega(z)-\nexxt\omega(y)
    \d\tritrans(x, y, z).
    \end{split}
\end{equation}
From \eqref{rem:fb:r-check-eq} we also recall that
\[
    \nexxt{\check\remainder}=\nexxt{\remainder}(\this \mu, \nexxt\tritrans_{\reverse}).
\]
We need to bound $\nexxt{\remainder}(\nexxt\mu, \nexxt\tritrans)$ by a factor of $\nexxt\epsilon$ for \cref{eq:sub:remainder,eq:sub:remainder-tight} (ergodic function value convergence, and function value covergence, respectively) to hold.
Likewise, we need to bound $\nexxt{\check\remainder}$ by a factor of $\nexxt\epsilon$ for \eqref{eq:fb:remainder} to hold (subdifferential convergence), and by a factor of $\nexxt\epsilon/(k+1)$ for \eqref{eq:sub:remainder-tight} (also required for function value convergence).
In the next lemma and the examples that follow it, we represent \emph{one practical possibility} for this bounding.
To use the lemma to prove \eqref{eq:fb:remainder}, we take $\mu=\this\mu$ and $\tritrans=\nexxt\tritrans_{\reverse}$, in which case $\nexxt{\remainder}(\mu, \tritrans)=\nexxt{\check\remainder}$.
To prove \eqref{eq:sub:remainder}, we take $\mu=\opt\mu$ and $\tritrans=\nexxt\tritrans$.
To prove \eqref{eq:sub:remainder-tight}, we perform both applications.

\begin{lemma}\label{lemma:fb:residual-bound:radon}
    Let $\mu \in \Masses$ and $\tritrans \in \ThreePlansCompatAlg(\nexxt\gamma)$.
    Assume:
    \begin{enumerate}[label=(\roman*)]
        \item\label{item:fb:residual-bound:marginal-strength}
        We have $
            \dualprod{\nexxt{\tilde\epsilon}}{\nexxt\mu-\mu}
            \le
            2 m \nexxt\epsilon
        $
        for some $m \ge 0$.
        \item\label{item:fb:residual-bound:control}
        For some $\Ctrans>0$, we impose
        \[
            \int_{\Omega^3}
                \sign\nexxt\gamma(x,y)[
                \tau \iprod{\grad \this v(x)}{z-y}
                + \nexxt\omega(z) - \nexxt\omega(y)
                ]
                - \frac{\tau\ellCurvature}{2}\norm{x-y}^2
            \d\abs{\tritrans}(x, y, z)
            \le \Ctrans\nexxt\epsilon.
        \]
    \end{enumerate}
    Then
    $
        \nexxt{\remainder}(\mu, \tritrans)
        \le
        (2m + \Ctrans)\nexxt\epsilon.
    $
\end{lemma}

\begin{proof}Using \cref{item:fb:residual-bound:marginal-strength} to estimate the first term of \eqref{eq:sub:controls:r-def}, and then \cref{lemma:sub:integrals} to transform the integration measure, we establish
    \[
        \begin{split}
        &\nexxt{\remainder}(\mu, \tritrans)
        \le
        2m\nexxt\epsilon
        +
        \int_{\Omega^3}
            \sign\nexxt\gamma(x,y)\left[
            \tau\iprod{\grad \this v(x)}{z-y}
            +\nexxt\omega(z) - \nexxt\omega(y)
            \right]
        \d\abs{\tritrans}(x, y, z)
        \\
        \MoveEqLeft[-1]
        - \tau\int_{\Omega^2} \frac{\ellCurvature}{2}\norm{x-y}^2 \d\pi_\#^{0,1}\abs{\tritrans}(x, y)
        \\
        & =
        2m\nexxt\epsilon
        +
        \int_{\Omega^3}
            \sign \nexxt\gamma(x, y)\left[
            \tau \iprod{\grad \this v(x)}{z-y}
            + \nexxt\omega(z)
            - \nexxt\omega(y)
            \right]
            -
            \frac{\ellCurvature}{2}\norm{x-y}^2
        \d\abs{\tritrans}(x, y, z).
        \end{split}
    \]
    The claim now follows by applying \cref{item:fb:residual-bound:control}.
\end{proof}

The conditions of the lemma are satisfied as follows:
\begin{enumerate}\item
    The condition \cref{item:fb:residual-bound:marginal-strength} strengthens the marginal condition \eqref{eq:fb:oc0:marginal}, i.e.,
    $
        -\nexxt\epsilon
        \le
        \nexxt{\tilde\epsilon}%
        \le
        \nexxt\epsilon.
    $
    We can ensure it when we solve for $\nexxt\mu$ in \cref{alg:fb:rough}.
    In fact, we can \emph{inductively} bound $\norm{\this\mu}$ using \cref{cor:fb:bounded} or \ref{cor:sub:bounded}, depending on the mode of convergence sought.
    When $\mu=\opt\mu$ is the unknown solution for \cref{eq:sub:remainder,eq:sub:remainder-tight}, it is by definition bounded.
    Therefore, using
    $
        -\nexxt\epsilon
        \le
        \nexxt{\tilde\epsilon}
        \le
        \nexxt\epsilon,
    $
    we see that \cref{item:fb:residual-bound:marginal-strength} is automatically satisfied for some bound $m > 0$ on $\norm{\mu}$ and $\norm{\nexxt\mu}$.

    \item The condition \cref{item:fb:residual-bound:control} can be satisfied by \emph{a posteriori}, after solving for $\nexxt\mu$, further reducing the mass of $\nexxt\gamma$, if necessary.
    This is done on \cref{step:fb:rough:reduce} of \cref{alg:fb:rough}.
    For the reduction, there are various alternative strategies: uniform reduction, soft-thresholding, removing the smallest or largest weights entirely, in order, etc.

    Practically, to avoid an explosion in the number spikes in $\nexxt\mu$, we do not want to have $\beta_i \ne 0$ and $\beta_i \ne \alpha_i$ for many indices $i$.
    We recall from \cref{alg:fb:rough} that the weights $\beta_i,\alpha_i$ are such that
    \begin{equation*}
        \this\mu = \sum_{i=1}^n \alpha_i\delta_{x_i}
        \quad\text{and}\quad
        \nexxt\gamma = \sum_{i=1}^n \beta_i\delta_{(x_i, y_i)}.
    \end{equation*}
    Thus, our implementation \cite{tuomov-pointsource-codes} orders the spikes of $\this\mu$ by age, and, in order, reduces $\beta_i$ corresponding to the newest spikes, if the contribution to the integral in \cref{item:fb:residual-bound:control} is positive, until the condition is satisfied.
    The next remarks and examples indicate how the integral can be evaluated and estimated in specific cases.
\end{enumerate}

\subsection{Specific cases}
\label{sec:controls:specific}

\begin{remark}[Subdifferential convergence]
    \label{rem:control:subdiff}
    For $\tritrans=\nexxt\tritrans_{\reverse}$, i.e., for proving \eqref{eq:fb:remainder} for \cref{thm:fb:subdiff-conv} (and part of \eqref{eq:sub:remainder-tight} for \cref{cor:sub:convergence:function}), the condition in \cref{item:fb:residual-bound:control} simplifies to
    \begin{rawalg}\abovedisplayskip=0pt
    \[
        \int_{\Omega^2}
            \sign\nexxt\gamma(x,y)[
            \tau \iprod{\grad \this v(x)}{x-y}
            - \nexxt\omega(y) + \nexxt\omega(x)
            ]
            - \frac{\tau\ellCurvature}{2}\norm{x-y}^2
        \d\abs{\nexxt\gamma}(x, y)
        \le \Ctrans\nexxt\epsilon.
    \]
    \end{rawalg}
\end{remark}

\begin{remark}[Function value convergence]
    \label{rem:control:value}
    To prove \cref{eq:sub:remainder} for \cref{thm:sub:convergence:function-ergodic} (and part of \eqref{eq:sub:remainder-tight} for \cref{cor:sub:convergence:function}), we take the supremum over $z \in \Omega$ inside the integral in \cref{item:fb:residual-bound:control}.
    As the integrand then becomes independent of $z$, we may continue with \cref{lemma:sub:integrals} to transform the unknown integration measure $\abs{\tritrans}$ into the known $\abs{\nexxt\gamma}$.
    We, thus, see \cref{item:fb:residual-bound:control} to hold if
    \begin{rawalg}\abovedisplayskip=0pt
    \[
        \int_{\Omega^2}
            \sup_{z \in \Omega}
            \sign\nexxt\gamma(x,y)
            [
            \tau \iprod{\grad \this v(x)}{z-y}
            - \nexxt\omega(y) + \nexxt\omega(z)
            ]
            - \frac{\tau\ellCurvature}{2}\norm{x-y}^2
        \d\abs{\nexxt\gamma}(x, y)
        \le \Ctrans\nexxt\epsilon.
    \]
    \end{rawalg}
\end{remark}

The possibly unconstructed function $\nexxt\omega$, we treat as follows:

\begin{example}[Particle-to-wave marginal term]
    \label{ex:fb:residual-bound:wave}
    For $E=E_\Wave$, the function $\nexxt\omega=\Wave(\nexxt\mu-\this{\breve\mu})$ is explicitly defined, so the integrals in \cref{rem:control:subdiff,rem:control:value} can be evaluated.
\end{example}

\begin{example}[Radon-squared marginal term]
    \label{ex:fb:residual-bound:radon}
    For $E=E_\Meas$, we cannot explicitly form $\nexxt\omega \in \norm{\nexxt\mu-\this{\breve\mu}}\subdiff\norm{\freevar}_\Meas(\nexxt\mu-\this{\breve\mu})$, but can bound
    $\nexxt\omega(z) \in \norm{\nexxt\gamma-\this{\breve\mu}}[-1,1]$ for all $z \in \Omega$.
    When $z=x$ or $z=y$, we can improve this estimate, as either the upper or lower bound is reached if $\nexxt\gamma$ differs from $\this{\breve\mu}$ at $z$.
    Recall that both measures are finite sums of Dirac measures.
    Indeed, by the construction of $\this{\breve\mu}$, and the initialisation and updates of $\nexxt\gamma$ in \cref{alg:fb:rough}, both involve the support of $\this\mu$, and their transported locations determined by \eqref{eq:fb:oc0:transport}.
\end{example}

\section{Extensions}
\label{sec:extensions}

We now sketch the extension of the work in \cref{sec:fb,sec:sub} to primal-dual methods for the problem
\begin{multline}
    \label{eq:extensions:problem0}
    \min_{\mu \in \Masses,\, z \in Z}~ F(\mu, z) + G(\mu) + R(z) + H(K(\mu, z))
    \\
    =
    \min_{\mu \in \Masses,\, z \in Z}\sup_{y \in Y}~ F(\mu, z) + G(\mu) + R(z) + \dualprod{K(\mu,z)}{y} - H^*(y),
\end{multline}
where $Z$ and $Y$ are Hilbert spaces, $F: \Masses \times Z \to \R$ is Fréchet differentiable; $G: \Masses \to \extR$, $R: Z \to \extR$, $H: Y \to \extR$ are convex, proper, and lower semicontinuous; and $K=(K_\mu, K_z) \in \linear(\Masses \times Z; Y)$.
By taking $K=0$ and $H=0$, this model can also treat the extension of forward-backward splitting to the product space $\Masses \times Z$.

We first formulate the algorithm in \cref{sec:extensions:algorithm}.
Then, in \cref{sec:extensions:convergence.fb}, we sketch the convergence of subdifferentials for the special case of forward-backward splitting.
Afterwards, in \cref{sec:extensions:convergence}, through an extension of \cref{lemma:fb:fejer}, we sketch the convergence of primal-dual gaps in the general case.
In \cref{sec:extensions:step-lengths}, for a specific case, we provide a more explicit expression of the step length conditions in the latter result.

\subsection{Algorithm}
\label{sec:extensions:algorithm}

\begin{algorithm}
    \caption{Sliding primal-dual proximal splitting in a product space}
    \label{alg:extensions:spdps}
    \begin{algorithmic}[1]
        \Require Setting of \cref{sec:extensions} with $G=\alpha\norm{\freevar}_\Meas + \delta_{\ge 0}$ and either $E=E_\Meas$ or $E=E_\Wave$ for a self-adjoint and positive semi-definite $\Wave \in \linear(\Masses; \DiffPredual)$.
        \Ensure \cref{ass:extensions:fb:subdiff,ass:extensions:fb:subdiff} or, if $E=E_\Wave$, alternatively, \cref{ass:extensions:value}.
        \State Follow \cref{alg:fb:rough} with the following changes:
        \State
        On \cref{step:fb:rough:v,step:fb:rough:vbreve}, set\label{step:extensions:spdps:v}
        $\this v= F^{(\mu)}(\this\mu, \this z) + K_\mu^*\this y$
        and
        $\this{\breve v}= F^{(\mu)}(\this{\breve\mu}, \brevez^{k+1}) + K_\mu^*\this y$,
        respectively.
        \State
        Before \cref{step:fb:rough:prune}, update\label{step:extensions:spdps:aux}
        \[
            \left\{
                \begin{array}{l}
                    \nexxt z \defeq \prox_{\sigma_p R}(\this z - \sigma_p \grad_z F(\this{\breve\mu}, \this z) - \sigma_p K_z^*\this y),
                    \\
                    \nexxt y \defeq \prox_{\sigma_d H^*}(\this y + \sigma_d K(2\nexxt \mu-\this \mu, 2\nexxt z-\this z))
                \end{array}
            \right.
        \]
    \end{algorithmic}

\end{algorithm}

We sketch the primal-dual method in \cref{alg:extensions:spdps}.
To analyse the method, we follow the approach of \cite[Section 4]{tuomov2024tracking} to represent it as a generalised forward-backward method.
We first write the problem \eqref{eq:extensions:problem0} in terms of the optimality conditions
\begin{equation}
    \label{eq:extensions:problem}
    0 \in \extF(u) + \extG(u) + \extXi u,
\end{equation}
where, for
\[
    u=(\mu, z, y) = (\mu, q) \in U \defeq \Masses \times Q
   \quad\text{with}\quad
   Q \defeq Z \times Y,
\]
we set
\begin{equation}
    \label{eq:extensions:extended-functions}
    \extG(u) = G(\mu) + R(z) + H^*(u), \quad
    \extF(u) = F(\mu, z),
    \quad\text{and}\quad
    \extXi = \begin{pmatrix}
        0 &  K^*  \\
        -K & 0 \\
    \end{pmatrix}.
\end{equation}
We also write for brevity
\[
    \GQ(z, y) \defeq R(z) + H^*(y).
\]

For step lengths $\sigma_p,\sigma_d>0$ that satisfy $\sigma_p\sigma_d \norm{K}^2 \le 1$, we set
\begin{equation}
    \label{eq:extensions:uwave}
    \UWave
    =
    \tau\begin{pmatrix}
        0 & 0 & -K_\mu^* \\
        0 & \inv\sigma_p \Id & -K_z^* \\
        -K_\mu & -K_z & \inv\sigma_d \Id
    \end{pmatrix}
    \quad\text{so that}\quad
    \tau\Xi + \UWave
    =
    \tau\begin{pmatrix}
        0 & 0 & 0 \\
        0 & \inv\sigma_p \Id & 0 \\
        -2K_\mu & -2K_z & \inv\sigma_d \Id
    \end{pmatrix}.
\end{equation}
Then $\UWave \in \linear(U; U_*)$ is positive semi-definite and self-adjoint, while $\Xi \in \linear(U; U_*)$ is skew-symmetric.
Let $U_* = \Predual \times Q$, recalling that $Q$ is a Hilbert space.
Write $\PPredual \defeq \begin{psmallmatrix} \Id &  0 & 0 \end{psmallmatrix}\in \linear(U_*; \Predual)$ and $P_{Q} \defeq \begin{psmallmatrix} 0 &  \Id & 0 \\ 0 & 0 & \Id \end{psmallmatrix} \in \linear(U_*; Q)$ for the projection operators from $U_*$ to $\Predual$ and $Q$.

\begin{rawalg}
    \begin{subequations}
    \label{eq:extensions:alg}
    We can now write \cref{alg:extensions:spdps} in an implicit expanded form.
    Let
    \begin{equation}
        \label{eq:extensions:u}
        \thisu \defeq (\this\mu, \this q) = (\this\mu, \this z, \this y)
        \quad\text{and}\quad
        \this{\breve u} \defeq (\this{\breve\mu}, \this q) = (\this{\breve\mu}, \this z, \this y).
    \end{equation}
    To update $\nexxt\mu$ and $\nexxt\gamma$, we solve \eqref{eq:fb:oc0} and one of the remainder condition \eqref{eq:fb:remainder}, \eqref{eq:sub:remainder}, or \eqref{eq:sub:remainder-tight}, with the redefinitions
    \begin{equation}
        \label{eq:extensions:v}
        \this v  \defeq  \extF^{(\mu)}(\thisu) + P_{\Meas_*}\extXi\thisy
        \quad\text{and}\quad
        \this{\breve v} \defeq  \extF^{(\mu)}(\this{\breve u}) + P_{\Meas_*}\extXi\thisy.
    \end{equation}
    This is \cref{step:extensions:spdps:v} of \cref{alg:extensions:spdps}.
    To update $\nexxt q$, we solve
    \begin{equation}
        \label{eq:extensions:non-measure-oc}
        0 \in \tau[\extF^{(q)}(\this{\breve u}) + \subdiff \GQ(\nexxt q) + P_{Q}\extXi\nextu] +  P_{Q}\UWave(\nextu-\thisu).
    \end{equation}
    This is \cref{step:extensions:spdps:aux} of \cref{alg:extensions:spdps}.
    \end{subequations}
\end{rawalg}

\subsection{Subdifferential convergence of a forward-backward method}
\label{sec:extensions:convergence.fb}

Suppose $H=0$ and $K=0$, so that \cref{alg:extensions:spdps} becomes a forward-backward method in the product space $\Masses \times Z$.
The iterates $\thisy \equiv 0$ are then redundant.
W can extend \cref{thm:fb:subdiff-conv}, which allows $E=E_\Meas$, to this setting.
We start by adapting \cref{ass:fb:all,ass:fb:subdiff}:

\begin{assumption}\label{ass:extensions:fb:subdiff}
    Besides the basic assumptions regarding $F$ and $G$ after  \eqref{eq:extensions:problem0}, and taking $H=0$ and $K=0$, we have:
    \begin{enumerate}[label=(\roman*)]
        \item\label{item:extensions:fb:all:e}
        \textbf{Convex energy:}
        $E: \Masses \times \Masses \to [0, \infty]$ is convex in the second parameter and $E(\nu,\nu)=0$ for all $\nu \in \Masses$.
        \item\label{item:extensions:fb:all:f}
        \textbf{Smooth data term:}
        $F: \Masses \times Z \to \R$ is Fréchet differentiable with $F^{(\mu)}(\mu, z) \in \DiffPredual$ for all $(\mu, z)$, and satifies for some $\ell,L,L_z \ge 0$ for all all $\mu,\nu \in \Masses$, $\gamma \in \TwoPlansSpace$, and $z, \tilde z \in Z$ the smoothness property
        \[
            B_F(\mu + (\pi_\#^1-\pi_\#^0)\gamma, \nu)
            \le
            V_{\ell c_2, L E}(\mu, \nu; \gamma)
            +
            \frac{L_z}{2}\norm{z-\tilde z}_{Z}^2.
        \]

        \item\label{item:extensions:fb:all:tolerances}
        \textbf{Vanishing tolerances:}
        The tolerances $\{\nexxt\epsilon\}_{k \in \N} \subset [0, \infty)$ satisfy
        $
            \lim_{N \to \infty} \frac{1}{N}\sum_{k=0}^{N-1} \nexxt\epsilon = 0.
        $

        \item\label{item:extensions:fb:subdiff:curvature}
        \textbf{Curvature lower bound:}
        For all $k \in \N$, we have $\ellF \abs{\nexxt\gamma}(c_2) \ge -B_{\bar F}(\this{\breve u},\thisu)$ for some $\ellF \ge 0$
        (e.g., $F$ is convex).

        \item\label{item:extensions:fb:subdiff:tau}
        \textbf{Step length bounds:}
        The step lengths $\tau,\theta>0$ and the parameters $\ell, \ellF, L$, and $L_z$ satisfy $\tau L < 1$, $\sigma_p L_z < 2$, and $\theta \tau[\ell+\ellCurvature+\ellF] < 2$.

        \item\label{item:extensions:fb:subdiff:e-f}
        \textbf{Marginal continuity:}
        $F'$ is continuous with respect to $E$ in the sense that,
        for any sequence $\{(\this\mu, \nexxt\gamma)\}_{k \in \N} \subset \Masses \times \TwoPlansSpace$,
        the convergences $E(\this\mu - \pi_\#^0\nexxt\gamma, \nexxt\mu - \pi_\#^1\nexxt\gamma) \to 0$ and $\nextz - \thisz \to 0$ imply
        $F'(\this\mu + (\pi_\#^1-\pi_\#^0)\nexxt\gamma, \thisz) - F'(\nexxt\mu, \nexxt z) \to 0$.

        \item\label{item:extensions:fb:subdiff:e-fenchel}
        \textbf{Inverse continuity of conjugate energy:}
        Let $\marginalEnergy{k}$ be as defined in \eqref{eq:fb:marginalenergy}.
        Then, for any sequence $\{\this\omega\}_{k \in \N} \subset \Predual$, the convergence $\marginalEnergy{k}^*(\this\omega) \to 0$ implies $\norm{\this\omega}_{\Predual} \to 0$.

        \item\label{item:extensions:fb:subdiff:f}
        \textbf{Objective lower bound:}
        $\inf [\extF+\extG] > -\infty$.
    \end{enumerate}
\end{assumption}

The next result establishes the convergence of subdiffereials for the product-space forward-backward method:

\begin{theorem}[Subdifferential convergence]
    With $H=0$ and $K=0$, suppose \cref{ass:extensions:fb:subdiff} holds, and that $\{(\this\mu, \this z, \nexxt\gamma)\}_{k \in \N}$ are generated by \cref{alg:extensions:spdps}.
    Then
    \[
        \inf_{\nexxt w \in \subdiff G(\nexxt\mu), \nexxt p \in \subdiff R(\nexxt z)} \norm{F'(\nexxt\mu, \nexxt z)+(\nexxt w, \nexxt p)} \to 0.
    \]
\end{theorem}

\begin{proof}The proof follows the outline of \cref{thm:fb:subdiff-conv}.
    Note from \eqref{eq:extensions:u} that now, as the $y$ variable is fixed to zero, we have
    $\thisu \defeq (\this\mu, \this z, 0)$ and $\this{\breve u} \defeq (\this{\breve\mu}, \this z, 0)$.
    Let $\nexxt w \in \subdiff G(\nexxt\mu)$ and $\nexxt p \in \subdiff R(\nexxt z)$.
    Then, due to $H=0$, we have $(\nexxt w, \nexxt p, 0)  \in \subdiff \extG(\nexxt u)$.
    By \eqref{eq:extensions:v}, since now $\extXi = 0$, we have $\this{\breve v} = F^{(\mu)}(\this{\breve u})$.
    By the subdifferentiability of $G$, and the definition of $\extF$ and the Bregman divergence, we, thus, have
    \begin{equation}
        \label{eq:extensions:fb:0}
        \begin{split}
        [\extF+\extG](\thisu) - [\extF+\extG](\nextu)
        &
        \ge
        \dualprod{\this{\breve v} + \nexxt w}{\this\mu-\nexxt\mu}
        + \iprod{\grad_z F(\this{\breve u}) + \nexxt p}{\this{z}-\nexxt{z}}
        \\
        \MoveEqLeft[-1]
        + B_{\extF}(\this{\breve u},\this u)
        - B_{\extF}(\this{\breve u},\nexxt u).
        \end{split}
    \end{equation}
    From \eqref{eq:extensions:non-measure-oc}, for some choice of $\nexxt p \in  \subdiff R(\nexxt z)$, we have
    \begin{equation}
        \label{eq:extensions:fb:z-cond}
        0 = \grad_z F(\this{\breve u}) + \nexxt p + \inv\sigma_p (\nextz - \thisz).
    \end{equation}
    (Observe that $\UWave$ in \eqref{eq:extensions:uwave} is multiplied by $\tau$, same as the first part of \eqref{eq:extensions:non-measure-oc}. Therefore, $\tau$ disappears here entirely.)
    By \cref{ass:extensions:fb:subdiff}\,\cref{item:extensions:fb:all:f}, we also have
    \begin{equation}
        \label{eq:extensions:fb:1}
        B_{\extF}(\this{\breve u},\nextu)
        \le
        V_{\ell c_2, LE}(\this\mu,\nexxt\mu; \nexxt\gamma)
        + \frac{L_z}{2}\norm{\nextz-\thisz}_{Z}^2.
    \end{equation}
    Using \cref{eq:extensions:fb:z-cond,eq:extensions:fb:1} in \cref{eq:extensions:fb:0}, we obtain
    \begin{equation}
        \label{eq:extensions:fb:2}
        \begin{split}
        [\extF+\extG](\thisu) - [\extF+\extG](\nextu)
        &
        \ge
        \dualprod{\this{\breve v} + \nexxt w}{\this\mu-\nexxt\mu}
        - V_{\ell c_2, LE}(\this\mu,\nexxt\mu; \nexxt\gamma)
        \\
        \MoveEqLeft[-1]
        +\left(\frac{1}{\sigma_p}-\frac{L_z}{2}\right)\norm{\nextz-\thisz}_Z^2
        + B_{\extF}(\this{\breve u},\this u).
        \end{split}
    \end{equation}
    We now observe that \eqref{eq:fb:monotonicity:0} in the proof of \cref{lemma:fb:monotonicity} continues to hold for $\nexxt{\check\remainder}$ as defined in \eqref{eq:fb:remainder} for the modified definitions of $\this{\breve v}$ and $\this v$.
    Multiplying \eqref{eq:extensions:fb:2} by $\tau$, and using \cref{eq:fb:monotonicity:0} and \cref{ass:extensions:fb:subdiff}\,\ref{item:extensions:fb:subdiff:curvature}, we, thus, obtain
    \[
        \tau [\extF+\extG](\nextu)
        + \marginalEnergy{k}^*(\nexxt\omega)
        + \fbmetricM{k}(\thisu, \nextu; \nexxt\gamma)
        \le
        \tau [\extF+\extG](\thisu)
        + \nexxt{\check\remainder},
    \]
    where, for $u=(\mu, z, 0)$ and $\tilde u=(\nu, \tilde z, 0)$, we define
    \[
        \fbmetricM{k}(u, \tilde u; \gamma)
        \defeq
        V_{(2\inv\theta -\tau[\ell+\ellCurvature+\ellF])c_2, (1-\tau L)E}(\mu, \nu; \gamma)
        + \tau\left(\frac{1}{\sigma_p}-\frac{L_z}{2}\right)\norm{\tilde z - z}_{Z}^2.
    \]
    By \cref{ass:extensions:fb:subdiff}\,\cref{item:extensions:fb:subdiff:tau}, we have $\fbmetricM{k} \ge 0$.
    Now we sum this over $k$ to obtain for some $C>0$, as the proof of \cref{thm:fb:subdiff-conv} that
    \[
        \sum_{k=0}^\infty\left(
            \marginalEnergy{k}^*(\nexxt\omega)
            + \fbmetricM{k}(\thisu, \nextu; \nexxt\gamma)
        \right)
        \le C.
    \]
    Due to \eqref{eq:extensions:fb:z-cond} and \cref{ass:extensions:fb:subdiff}\,\cref{item:extensions:fb:subdiff:tau}, this gives $\grad_z F(\this{\breve u}) + \nexxt p \to 0$.
    The convergence of the $\mu$-subdifferential follows as in the proof of  \cref{thm:fb:subdiff-conv}.
\end{proof}

\subsection{Gap convergence of the full primal-dual method}
\label{sec:extensions:convergence}

We now sketch extensions of \cref{lemma:fb:fejer,thm:extensions:convergence:function-ergodic} for the full primal-dual method.
For this we define the \term{Lagrangian gap functional}
\[
    \gap(u; \bar u) \defeq [\extF + \extG](x) - [\extF + \extG](\bar u) - \dualprod{\extXi u}{\bar u}_{U^*,U}
    \quad (u, \bar u \in U).
\]
For $\optu$ a solution to \eqref{eq:extensions:problem}, we have $\gap(u; \optu) \ge 0$ \cite[Lemma 11.1]{clasonvalkonen2020nonsmooth}\footnote{The lemma is written in Hilbert spaces, but the proof remains unchanged in general normed spaces.}.

The results here do not apply to the marginal term $E=E_{\Meas}$, as they require \eqref{eq:unbalanced:energy}.
For simplicity, we also require that $F$ be convex. Then $\ellF=0$. Nonconvex $F$ with sufficient local growth, could be treated using a priori--a posteriori locality arguments following \cite{tuomov-nlpdhgm-redo,tuomov-firstorder}.

In the following assumption, we adapt \cref{ass:fb:all,ass:sub:energy,ass:sub:rebalancing}.
Parts \cref{item:extensions:all:f,item:extensions:all:tau} ensure bounds analogous to \cref{lemma:sub:v-lower-bound} on $\fbmetricBar{k}^{i,2}$ and $\fbmetricL{k}$.

\begin{assumption}\label{ass:extensions:value}
    Besides the basic assumptions regarding $F, G, H$, and $K$ after  \eqref{eq:extensions:problem0}, we assume that
    \begin{enumerate}[label=(\roman*)]
        \item\label{item:extensions:all:e}
        \textbf{Sub-Pythagorean energy:}
        $E: \Masses \times \Masses \to [0, \infty]$ satisfies \eqref{eq:unbalanced:energy} with $E(\nu,\nu)=0$ for all $\nu \in \Masses$.
        \item\label{item:extensions:all:f}
        \textbf{Convex, smooth data term:}
        $F: \Masses \times Z \to \R$ is convex and pre-differentiable, $\PPredual F'(u) \in \DiffPredual$ for all $u \in U$. For some $\ell, L \ge 0$, for all $\mu,\nu \in \Masses$; $q,p \in Q$, for all $\gamma \in \TwoPlansSpace$ and the extended function $\extF$ from \eqref{eq:extensions:extended-functions}, we have
        \[
            V_{\ell c_2, L E}(\mu, \nu; \gamma)
            + \frac{1}{2\tau}\norm{(\mu,q)-(\nu,p)}_{\UWave}^2
            \ge
            B_{\extF}((\mu + (\pi_\#^1-\pi_\#^0)\gamma, \breveq), (\nu, p)).
        \]

        \item\label{item:extensions:all:tau}
        \textbf{Step length bounds:}
        The step length parameters $\tau,\theta>0$ and the factors $\ell, \ellCurvature$, and $L$ satisfy $\tau L \le 1$ and $\theta \tau [\ell+\ellCurvature] \le 1$.

        \item\label{item:extensions:all:tolerances}
        \textbf{Vanishing tolerances:}
        The tolerances $\{\nexxt\epsilon\}_{k \in \N} \subset [0, \infty)$ satisfy
        $
            \lim_{N \to \infty} \frac{1}{N}\sum_{k=0}^{N-1} \nexxt\epsilon = 0.
        $
        \item\label{item:extensions:all:rebalancing}
        \textbf{Rebalancing:}
        For all $k \in \N$, for given $\nexxt u \in \Masses$, $\nexxt\gamma \in \TwoPlansSpace$, and $\nexxt\tritrans \in \ThreePlansCompatAlg(\nexxt\gamma)$,
        the choice of
        $\gamma^{k+2} \in \TwoPlansSpace$ and $\tritrans^{k+2} \in \ThreePlansCompatAlg(\gamma^{k+2}) \isect \ThreePlansNext(\nexxt\tritrans)$
        satisfies
        $
            \fbmetricBar{k+1}^{1,2}(\nexxt u, \optu; \nexxt\tritrans)
            \ge
            \fbmetricBar{k+2}^{0,2}(\nexxt u, \optu; \tritrans^{k+2})
        $
        for all $\opt u \in U$.
    \end{enumerate}
\end{assumption}

We will return to step length parameter choices in \cref{lemma:extensions:spdps:step-lengths}.
The next result extends the first part of \cref{lemma:fb:fejer}.
There, for $u=(\mu, q)$ and $\tilde u=(\nu, \tilde q)$, writing $\breve u \defeq (\mu + (\pi_\#^1-\pi_\#^0)\gamma, q)$, we extend the distance definitions \cref{eq:sub:metric02-12,eq:sub:metric01} as
\begin{subequations}
\label{eq:extensions:metrics}
\begin{align}
    \fbmetricBar{k}^{i,2}(u, \tilde u; \tritrans)
    &
    \defeq
    \bar V_{\inv\theta c_2, E}^{i,2}(\mu, \nu; \tritrans)
    + \frac{1}{2}\norm{u - \tilde u}_{\UWave}^2
    \quad\text{for } i=0,1,
    \quad\text{ and}
    \\
    \fbmetricL{k}(u, \tilde u; \tritrans)
    &
    \defeq
    \bar V_{(\inv\theta -\tau\ell)c_2, (1-\tau L)E}^{0,1}(\mu, \nu; \tritrans)
        + \frac{1}{2}\norm{u - \tilde u}_{\UWave}^2.
\end{align}
\end{subequations}
It follows from \cref{ass:extensions:value} and the definition of $\UWave$ in \eqref{eq:extensions:uwave} that these functions are non-negative.

\begin{lemma}\label{lemma:extensions:descent}
    Suppose \cref{ass:extensions:value}\,\cref{item:extensions:all:e,item:extensions:all:f%
    } hold, and that $k \in \N$ and $(\this\mu, \this q, \nexxt \gamma) \in \Masses \times Q \times \TwoPlansSpace$ are given.
    If \cref{eq:fb:oc0,eq:extensions:alg} hold, then for any $u=(\mu, q) \in \Masses \times Q$, and any $\tritrans \in \ThreePlansCompatAlg(\nexxt\gamma)$,
    \begin{equation}
        \label{eq:extensions:descent}
        0 \ge
        \tau \gap(\nextu; u)
        +
        \fbmetricBar{k}^{1,2}(\nextu, u; \tritrans)
        - \fbmetricBar{k}^{0,2}(\thisu, u; \tritrans)
        + \fbmetricL{k}(\thisu, \nextu; \tritrans)
        - \nexxt{\remainder}(\mu, \tritrans).
    \end{equation}
\end{lemma}

\begin{proof}[Sketch of proof]
    Let $\nexxt p \in \subdiff \GQ(\nexxt q)$ be the element for which \eqref{eq:extensions:non-measure-oc} is satisfied.
    Then $(\nexxt w, \nexxt p) \in \subdiff \extG(\nextu)$.
    Since $\dualprod{\extXi\nextu}{\nextu}=0$, minding \cref{ass:extensions:value}\,\cref{item:extensions:all:f}, using the subdifferentiability of $G$ and the Bregman three-point identity of \cref{lemma:bregman-three-point} on $\extF$, we estimate
    \begin{equation}
        \label{eq:extensions:descent:0}
        \begin{split}
        0
        &
        \ge
        \gap(\nextu; u)
        + \dualprod{P_{Q} \extF'(\this{\breve u}) + \nexxt p + P_{Q}\extXi\nextu}{q-\nexxt q}_{Q,Q}
        \\
        \MoveEqLeft[-1]
        + \dualprod{\PPredual F'(\this{\breve u}) + \nexxt w +  \PPredual\extXi\nextu}{\mu-\nexxt\mu}_{\Predual,\Masses}
        + B_{\extF}(\this{\breve u},u)
        - B_{\extF}(\this{\breve u},\nextu).
        \end{split}
    \end{equation}
    From the definitions in and after \eqref{eq:extensions:uwave}, observe that
    $
        \PPredual[\tau\extXi+\UWave]=0.
    $
    and
    $
        \PPredual\UWave (\PPredual)^*=0.
    $
    Hence, also using \cref{eq:extensions:v}, we obtain
    \begin{gather}
        \label{eq:extensions:descent:2}
        \begin{split}
        \tau \PPredual[F'(\this{\breve u}) + \extXi\nextu]
        =
        \PPredual[
            \tau F'(\this{\breve u}) + \tau \extXi\this{\breve u}
            - \UWave(\nextu-\this{\breve u})
        ]
        &
        = \tau \this{\breve v} - \PPredual\UWave(\nextu-\this{\breve u})
        \\
        &
        = \tau \this{\breve v} - \PPredual\UWave(\nextu-\thisu).
        \end{split}
    \intertext{Likewise, by \eqref{eq:extensions:non-measure-oc},}
        \label{eq:extensions:descent:3}
        \tau[P_{Q} \extF'(\this{\breve u}) + \nexxt p + P_{Q}\extXi\nextu]
        =
        - P_{Q}\UWave(\nextu-\thisu).
    \end{gather}
    Multiplying \cref{eq:extensions:descent:0} by $\tau$ and using \cref{eq:extensions:descent:2,eq:extensions:descent:3}, as well as that $B_{\extF}(\this{\breve u},u) \ge 0$ by the convexity of $F$, we now get
    \begin{equation}
        \label{eq:extensions:descent:5}
        0
        \ge
        \tau \gap(\nextu; u)
        + \tau\dualprod{\this{\breve v} + \nexxt w}{\mu-\nexxt\mu}_{\Predual,\Masses}
        - \dualprod{\UWave(\nextu - \thisu)}{u-\nextu}_{X_*,X}
        - B_{\tau\extF}(\this{\breve u},\nextu).
    \end{equation}
    Using $\tritrans \in \ThreePlansCompatAlg(\nexxt\gamma)$, \eqref{eq:fb:oc0:transport} and \cref{ass:extensions:value}\,\cref{item:extensions:all:e,item:extensions:all:f} with \cref{thm:unbalanced:twocost-three-point}, we prove, exactly as \eqref{eq:fb:descent:1:alt} in the proof of \cref{lemma:fb:fejer}, that
    \[
        \begin{split}
        0
        &
        \ge
        \bar V_{\inv\theta c_2, E}^{1,2}(\nexxt\mu, \mu; \tritrans)
        - \bar V_{\inv\theta c_2, E}^{0,2}(\this\mu, \mu; \tritrans)
        +\bar V^{0,1}_{[\inv\theta -\tau\ellCurvature]c_2, E}(\this\mu, \nexxt\mu; \tritrans)
        \\
        \MoveEqLeft[-1]
        - \tau\dualprod{\this{\breve v}+\nexxt{w}}{\mu - \nexxt\mu}
        - \nexxt\remainder(\mu, \tritrans).
        \end{split}
    \]
    Using in \eqref{eq:extensions:descent:5} this inequality together with the Pythagoras' identity for the $\UWave$-semi-norm, and the definitions \eqref{eq:extensions:metrics}, we obtain the claim.
\end{proof}

Now, similarly to \cref{thm:sub:convergence:function-ergodic}, we obtain an ergodic convergence result.

\begin{theorem}[Ergodic gap convergence]
    \label{thm:extensions:convergence:function-ergodic}
    Suppose \cref{ass:extensions:value} holds.
    Generate $\{\nextu = (\nexxt\mu, \nexxt\gamma)\}_{k \in \N}$ through the satisfaction of \cref{eq:fb:oc0,eq:sub:remainder,eq:extensions:alg} for an initial $u^0 = (\mu^0, q^0) \in \Masses \times Q$.
    Pick $N \in \N$ and $\tritrans^1 \in \ThreePlansCompatAlg(\gamma^1)$.
    Then for $\tilde u^N \defeq \frac{1}{N}\sum_{k=0}^{N-1}u^{k+1}$ and any $\optu=(\opt\mu,\opt q)$, we have
    \begin{equation}
        \label{eq:extensions:convergence:function-ergodic}
        \gap(\tilde u^N; \optu)
        \le
        \frac{1}{\tau N} \fbmetricBar{0}^{0,2}(u^0, \optu; \tritrans^1)
        + \frac{\constCurvature}{N}
        + \frac{C_{\opt\mu}}{\tau N} \sum_{k=0}^{N-1} \nexxt\epsilon.
    \end{equation}
    In particular, if $\optx$ solves \eqref{eq:extensions:problem}, then $\gap(\tilde u^N; \optu) \to 0$ at the rate $O(1/N)$.
\end{theorem}

\begin{proof}[Sketch of proof]
    The proof is analogous to \cref{thm:sub:convergence:function-ergodic}.
    We apply \cref{lemma:extensions:descent}, bounding the distance terms similarly to \cref{lemma:sub:v-lower-bound} with \cref{ass:extensions:value}\,\cref{item:extensions:all:f,item:extensions:all:tau}, \cref{item:fb:subdiff:curvature}.
    Therefore, using \eqref{eq:sub:remainder}, the rebalancing \cref{ass:extensions:value}\,\cref{item:extensions:all:rebalancing}, and telescoping, we obtain \eqref{eq:extensions:convergence:function-ergodic}.
    For the final claim, we use \eqref{eq:extensions:convergence:function-ergodic} and the fact that $\gap(\freevar; \optu) \ge 0$ for $\optu$ solving \eqref{eq:extensions:problem}.
\end{proof}

\subsection{Primal-dual step lengths}
\label{sec:extensions:step-lengths}

We now verify \cref{ass:extensions:value} based on simpler step length conditions in the case $E=E_\Wave$ for a self-adjoint $\Wave \in \linear(\Masses; \DiffPredual)$, and $K_\mu=0$, as follows:
\begin{itemize}\item[\cref{item:extensions:all:e}] of \cref{ass:extensions:value} holds as $E_\Wave$ is a Bregman divergence; see the discussion after \eqref{eq:extensions:problem0}.
    \item[\cref{item:extensions:all:f}] we prove in \cref{lemma:extensions:spdps:step-lengths} below, subject to $\sigma_p,\sigma_d>0$ satisfying $\sigma_p L_z + \sigma_p\sigma_d\norm{K_z}^2 <1$.
    \item[\cref{item:extensions:all:tau}] is easy to satisfy through the choice of the primal step length $\tau$, the transport step length $\theta$, given a choice of $\ellCurvature>0$, as as well as $L,\ell>0$ from \cref{item:extensions:all:f}.
    \item[\cref{item:extensions:all:tolerances}] is satisfied through appropriate choice of the tolerances, e.g., $\epsilon^k=1/k^\eta$ for some $\eta>0$.
    \item[\cref{item:extensions:all:rebalancing}] holds similarly to \cref{lemma:fb:rebalancing}.
\end{itemize}
Aside from \cref{item:extensions:all:f}, these arguments also hold when $K_\mu \ne 0$.
The extension of the next lemma for \cref{item:extensions:all:f} to $K_\mu$ is also possible, but for simplicity of presentation, we consider only this case, as it is sufficient for our numerical experiments.
The condition \eqref{eq:extensions:spdps:step-lengths:cond} there is a simplified smoothness condition, which can be verified for $F(\mu, z)=\frac{1}{2}\norm{A\mu+z-b}^2$, similarly to \cref{ex:fb:assumptions-quadratic}.
As a result, we obtain a simple control on the step length parameters $\sigma_p$ and $\sigma_d$.

\begin{lemma}\label{lemma:extensions:spdps:step-lengths}
    Assume the setup of \cref{alg:extensions:spdps} with $\extF: U \to \R$ convex and pre-differentiable with $\PPredual F'(u) \in \DiffPredual$ for all $u \in U$.
    For simplicity, take $K_\mu=0$.
    Suppose for some $\ell, L, L_z \ge 0$ that
    \begin{equation}
        \label{eq:extensions:spdps:step-lengths:cond}
        B_{F}((\mu + (\pi_\#^1-\pi_\#^0)\gamma, z), (\nu, w))
        \le
        V_{\ell c_2, L E}(\mu, \nu; \gamma)
        +
        \frac{L_z}{2}\norm{z-w}_Z^2
    \end{equation}
    for all $\mu, \nu \in \Masses$; $\gamma \in \TwoPlansSpace$; and $(z, w) \in Z$.
    Then  \cref{ass:extensions:value}\,\cref{item:extensions:all:f} holds if the step length parameters $\sigma_p,\sigma_d>0$ satisfy $\sigma_p L_z + \sigma_p\sigma_d\norm{K_z}^2 <1$.
\end{lemma}

\begin{proof}Recall the definition of $\UWave$ in \eqref{eq:extensions:uwave}.
    Let $\beta \defeq \sigma_p\sigma_d\norm{K_z}^2/(1 - \sigma_p L_z)$.
    Then $\beta \in (0, 1)$.
    By several applications of Young's inequality (compare \cite[Lemma 9.12]{clasonvalkonen2020nonsmooth}), for any $x=(\mu,z,y)$,  we have
    \[
        \begin{split}
        \frac{1}{\tau}\norm{u}_{\UWave}^2
        &
        \ge
        \inv\sigma_p\norm{z}_Z^2
        - \frac{1-\beta}{\sigma_d} \norm{y}_Y^2
        + \frac{1}{\sigma_d}\norm{y}_Y^2
        - \frac{\sigma_d}{\beta} \norm{K_z z}_Y^2
        - \frac{\beta}{\sigma_d} \norm{y}_Y^2
        \\
        &
        =
        \inv\sigma_p \norm{z}^2
        - \frac{\sigma_d}{\beta}\norm{K_z z}_Y^2
        \ge L_z \norm{z}_Z^2.
        \end{split}
    \]
    Combining this with \eqref{eq:extensions:spdps:step-lengths:cond}, we obtain \cref{ass:extensions:value}\,\cref{item:extensions:all:f}.
\end{proof}

\section{Numerical experience}
\label{sec:numerical}

We now numerically treat the problem \eqref{eq:intro:problem}, following the simple convolution sensor setup of \cite{tuomov-pointsource}, as we describe in \cref{sec:numerical:experiments}, after briefly describing algorithm parametrisation in \cref{sec:numerical:details}.
We finish with a report and discussion on the performance in \cref{sec:numerical:results}.

Moreover, to demonstrate auxiliary parametrisation and the sliding PDPS of \cref{alg:extensions:spdps}, we consider the “biased” problem with TV-regularisation,
\begin{equation}
    \label{eq:numerical:problem:biased}
    \min_{\mu \in \Masses,\, z \in Z}~
        \frac{1}{2}\norm{A\mu+z-b}^2 + \alpha\norm{\mu} + \delta_{\ge 0}(\mu)
        + \lambda \norm{\grad_h z}_{2,1},
\end{equation}
where $\grad_h$ describe a discretised gradient operator from 1D-signals $Z=\R^n$ to $\R^n$ or from 2D images $Z=\R^{n_1n_2}$ to vector fields in $\linear(\R^{n_1n_2}; \R^2)$, and $\norm{\freevar}_{2,1}$ is a sum over 2-norms over each component.
The operator $A$ is as for \eqref{eq:intro:problem}.
The idea is that we observe on a sensor grid (a camera) the image $b$ that has the spikes $\mu$ superimposed over an unknown background image $z$.

In \cite{dangvalkonen2026leak}, we consider $A$ arising as the non-linear solution operator of a convection--diffusion model.

\subsection{Algorithm implementation and parametrisation}
\label{sec:numerical:details}

For the basic problem \eqref{eq:intro:problem} we implemented \cref{alg:fb:rough} (\sFBName), the algorithms from \cite{tuomov-pointsource} (\FBName~and \PDPSName) as well as the “fully corrective” conditional gradient method of \cite[Algorithm 2]{walter2019linear} (\FWfName).
Moreover, we include results for variants of the forward-backward methods with a Radon-norm-squared proximal term (\radonFBName~and \sFBradonName).
Our Rust implementations are available on Zenodo \cite{tuomov-pointsource-codes}.
The implementation also includes the inertial \FISTAName~and the “relaxed” conditional gradient method of \cite[Algorithm 5.1]{brediespikkarainen2013inverse}, but as these were not, correspondingly, better than \FBName~and \FWfName~in the experiments of \cite{tuomov-pointsource}, we do not include the results here.
For the “biased” problem \eqref{eq:numerical:problem:biased}, we implemented the sliding PDPS of \cref{alg:extensions:spdps} (\sPDPSName), as well as its non-sliding variant (\fPDPSName).
We also implemented their variants with a Radon-norm-squared proximal term (\sPDPSradonName~and \fPDPSradonName), although our theory does not show them to be convergent.

Several details of the implementation, such as the branch-and-bound method for the insertion \cref{alg:fb:insert-and-remove}, are documented in \cite{tuomov-pointsource}.
Conditional gradient methods crucially depend on merging heuristics to keep the number of spikes computationally manageable. These are also documented in \cite{tuomov-pointsource}.
Contrary to the experiments in \cite{tuomov-pointsource}, on some of the experiments, we now needed to also enable merging heuristics for \PDPSName~to avoid accumulating a high number of spikes once the method starts to be very near a solution.
The situation is the same with the new (non-sliding) \radonFBName~and \fPDPSradonName, with which we use the same objective function decrease based merging as with \FWfName.
This is possible with the weak convergence guarantees of \cref{thm:fb:subdiff-conv}, but not with \cref{thm:sub:convergence:function-ergodic,cor:sub:convergence:function}.
We do not use or need merging heuristics with the other variants of our methods, except to clean up after the final step.

Most methods use semismooth Newton (SSN) for the finite-dimensional subproblems of \cref{alg:fb:insert-and-remove}, occasionally falling back to forward-backward splitting if the Newton system is too ill-conditioned.
For the Radon norm proximal term the subproblem on \cref{step:fb:insert-and-remove-radon:sub} of \cref{alg:fb:insert-and-remove-radon} can be solved exactly in linear time.
Such an algorithm can be developed along similar lines as sorting algorithms for simplex projection (see, e.g., \cite{angerhausen2022stochastic}) or the constrained positive $\ell_1$ regularisation algorithm of \cite[Lemma D.2]{suonpera2024general}.
We provide details in the internal documentation of our implementation \cite{tuomov-pointsource-codes}.

We always take the initial iterate $\mu^0=0$. For \eqref{eq:numerical:problem:biased}, also initial $z^0=0$ and the dual iterate $y^0=0$.
Based on trial and error, we take the tolerance sequence $\epsilon_k = 0.5 \tau \alpha/(1+0.2k)^{1.4}$, where $k$ is the iteration number.
This choice balances between fast initial convergence and not slowing down later iterations too much via excessive accuracy requirements.
Moreover, as a bootstrap heuristic, on the first 10 iterations, we insert in \cref{alg:fb:insert-and-remove} at most one point irrespective of the tolerance $\epsilon_k$.
This does not affect convergence, as the convergence theory can be applied starting from any fixed iteration number.
We indicate in \cref{tab:numerical:algorithm-params,tab:numerical:algorithm-params:biased} the step length and other parameters that differ between each algorithm. For the step length parameters we indicate value relative to the maximal value, e.g.,
for \FBName~we take $\tau = \tau_0 / L$ where $\tau \in (0, 1)$ is indicated in the table, and $L$ satisfies $A_*A \le L \Wave$.
We take $\ellCurvature$ as three times the estimate of \cref{lemma:unbalanced:lip-diff-value-sensorA}.

\subsection{Experiments}
\label{sec:numerical:experiments}

\def\datapath{results_final_redo/}%

\def\getTolerance#1#2{$
    \JSONParseValue{#1}{#2.Power.initial}
    \tau\alpha/(
        1 + \JSONParseValue{#1}{#2.Power.factor}k
    )^{\JSONParseValue{#1}{#2.Power.exponent}}
$}

\def\getMerging#1#2{%
    \def\enabled{\JSONParseExpandableValue{#1}{#2.enabled}}%
    \def\interp{\JSONParseExpandableValue{#1}{#2.interp}}%
    \def\truevalue{\detokenize{true}}
    \ifthenelse{\equal{\enabled}{\truevalue}}{%
        \ifthenelse{\equal{\interp}{\truevalue}}{%
            i:$\JSONParseValue{#1}{#2.radius}$%
        }{%
            m:$\JSONParseValue{#1}{#2.radius}$%
        }%
    }{%
        no%
    }%
}

\def\OneDFastDirName{pointsource1d_fast}
\def\TwoDFastDirName{pointsource2d_fast}
\def\OneDFastBiasedDirName{pointsource1d_tv_fast}
\def\TwoDFastBiasedDirName{pointsource2d_tv_fast}
\def\OneDFastName{1D “fast”}
\def\TwoDFastName{2D “fast”}
\def\OneDFastBiasedName{1D “biased”}
\def\TwoDFastBiasedName{2D “biased”}

\def\loadsettings#1#2{%
    \edef\filename{\datapath\csname#1DirName\endcsname/\csname#2Filename\endcsname_config.json}
    \expandafter\JSONParseFromFile\csname CFG#1X#2\endcsname{\filename}%
}

\def\FBsettings#1#2{%
    \csname#2Name\endcsname
    & \expandafter\JSONParseValue\csname CFG#1X#2\endcsname{FB[0].τ0}
    &
    &
    & \expandafter\getMerging\csname CFG#1X#2\endcsname{FB[0].insertion.merging}
    & \expandafter\JSONParseValue\csname CFG#1X#2\endcsname{FB[0].insertion.inner.method}
    &
}
\def\sFBsettings#1#2{%
    \csname#2Name\endcsname
    & \expandafter\JSONParseValue\csname CFG#1X#2\endcsname{SlidingFB[0].τ0}
    & \expandafter\JSONParseValue\csname CFG#1X#2\endcsname{SlidingFB[0].transport.θ0}
    &
    & \expandafter\getMerging\csname CFG#1X#2\endcsname{SlidingFB[0].insertion.merging}
    & \expandafter\JSONParseValue\csname CFG#1X#2\endcsname{SlidingFB[0].insertion.inner.method}
    & \expandafter\JSONParseValue\csname CFG#1X#2\endcsname{SlidingFB[0].transport.tolerance_mult}
}
\def\PDPSsettings#1#2{%
    \csname#2Name\endcsname
    & \expandafter\JSONParseValue\csname CFG#1X#2\endcsname{PDPS[0].τ0}
    &
    & \expandafter\JSONParseValue\csname CFG#1X#2\endcsname{PDPS[0].σ0}
    & \expandafter\getMerging\csname CFG#1X#2\endcsname{PDPS[0].generic.merging}
    & \expandafter\JSONParseValue\csname CFG#1X#2\endcsname{PDPS[0].generic.inner.method}
    &
}
\def\FWsettings#1#2{%
    \csname#2Name\endcsname
    &
    &
    &
    & \expandafter\getMerging\csname CFG#1X#2\endcsname{FW.merging}
    & \expandafter\JSONParseValue\csname CFG#1X#2\endcsname{FW.inner.method}
    &
}
\def\sPDPSsettings#1#2{%
    \csname#2Name\endcsname
    & \expandafter\JSONParseValue\csname CFG#1X#2\endcsname{SlidingPDPS[0].τ0}
    & \expandafter\JSONParseValue\csname CFG#1X#2\endcsname{SlidingPDPS[0].transport.θ0}
    & \expandafter\JSONParseValue\csname CFG#1X#2\endcsname{SlidingPDPS[0].σp0}
    & \expandafter\JSONParseValue\csname CFG#1X#2\endcsname{SlidingPDPS[0].σd0}
    & \expandafter\getMerging\csname CFG#1X#2\endcsname{SlidingPDPS[0].insertion.merging}
    & \expandafter\JSONParseValue\csname CFG#1X#2\endcsname{SlidingPDPS[0].insertion.inner.method}
    & \expandafter\JSONParseValue\csname CFG#1X#2\endcsname{SlidingPDPS[0].transport.tolerance_mult}
}
\def\fPDPSsettings#1#2{%
    \csname#2Name\endcsname
    & \expandafter\JSONParseValue\csname CFG#1X#2\endcsname{ForwardPDPS[0].τ0}
    & \expandafter\JSONParseValue\csname CFG#1X#2\endcsname{ForwardPDPS[0].transport.θ0}
    & \expandafter\JSONParseValue\csname CFG#1X#2\endcsname{ForwardPDPS[0].σp0}
    & \expandafter\JSONParseValue\csname CFG#1X#2\endcsname{ForwardPDPS[0].σd0}
    & \expandafter\getMerging\csname CFG#1X#2\endcsname{ForwardPDPS[0].insertion.merging}
    & \expandafter\JSONParseValue\csname CFG#1X#2\endcsname{ForwardPDPS[0].insertion.inner.method}
    & \expandafter\JSONParseValue\csname CFG#1X#2\endcsname{ForwardPDPS[0].transport.tolerance_mult}
}

\begin{table}[t]
    \caption{Algorithm parametrisations for the basic experiments.
        Empty fields are not applicable to the algorithm in question.
        For merging, “i:$\rho$” indicates that weighted interpolation is used to form a new spike to replace the merged spikes, with $\rho$ the candidate search radius, while “m:$\rho$” indicates that merging moves mass between the merged spikes.
        The transport tolerance multiplier $\Ctrans$ is explained in \cref{lemma:fb:residual-bound:radon}.
    }
    \label{tab:numerical:algorithm-params}
    \tikzifexternalizing{}{%
        \centering
        \small
        \foreach \experiment in {OneDFast, TwoDFast} {
            \foreach \alg in {FB, PDPS, FWf, sFB, radonFB, sFBradon} {
                \loadsettings{\experiment}{\alg}
            }
            \subcaptionbox{\csname\experiment Name\endcsname~ experiment}{
                \begin{tabular}{lrrrrrrr}
                    \toprule
                    Method
                    & $\tau_0$
                    & $\theta_0$
                    & $\sigma_0$
                    & merging
                    & inner method
                    & $\Ctrans$
                    \\
                    \midrule
                    \FBsettings{\experiment}{FB}
                    \\
                    \PDPSsettings{\experiment}{PDPS}
                    \\
                    \FWsettings{\experiment}{FWf}
                    \\
                    \sFBsettings{\experiment}{sFB}
                    \\
                    \FBsettings{\experiment}{radonFB}
                    \\
                    \sFBsettings{\experiment}{sFBradon}
                    \\
                    \bottomrule
                \end{tabular}%
            }\\
            \medskip
        }
    }
\end{table}

\begin{table}[t]
    \caption{Algorithm parametrisations for the “biased” experiments.
        For merging, “m:$\rho$” indicates that merging moves mass between the merged spikes,
        with $\rho$ the candidate search radius.
        The transport tolerance multiplier $\Ctrans$ operates similarly to \cref{lemma:fb:residual-bound:radon}.
    }
    \label{tab:numerical:algorithm-params:biased}
    \tikzifexternalizing{}{%
        \centering
        \small
        \foreach \experiment in {OneDFastBiased, TwoDFastBiased} {
            \foreach \alg in {sPDPS, fPDPS, sPDPSradon, fPDPSradon} {
                \loadsettings{\experiment}{\alg}
            }
            \medskip
            \subcaptionbox{\csname\experiment Name\endcsname~ experiment}{
                \begin{tabular}{lrrrrrrrr}
                    \toprule
                    Method
                    & $\tau_0$
                    & $\theta_0$
                    & $\sigma_{p,0}$
                    & $\sigma_{d,0}$
                    & merging
                    & inner method
                    & $\Ctrans$
                    \\
                    \midrule
                    \sPDPSsettings{\experiment}{sPDPS}
                    \\
                    \fPDPSsettings{\experiment}{fPDPS}
                    \\
                    \sPDPSsettings{\experiment}{sPDPSradon}
                    \\
                    \fPDPSsettings{\experiment}{fPDPSradon}
                    \\
                    \bottomrule
                \end{tabular}%
            }\\
            \medskip
        }
    }
\end{table}

We use the squared data term $F(\mu)=\frac{1}{2}\norm{A\mu-b}^2$ in both $\Omega=[0, 1]$ and $\Omega=[0, 1]^2$.
For the forward operator, following the construction in \cite[Theorem 3.3]{tuomov-pointsource}, we take $A \in \linear(\Masses; \R^n)$ defined by $[A\mu]_i = \mu(\theta_i * \psi)$, ($i=1,\ldots,n$), where each sensor $i$ on a uniform grid with centres $z_i$ has the field-of-view $\theta_i(x)=\chi_{[-r,r]^n}(x-z_i)$.
In $[0, 1]$ we use 100, and in $[0, 1]^2$ we use $16 \times 16$ equally spaced sensors with $r$ the sensor spacing times 0.4.\footnote{%
    The small number of sensors along each axis in 2D is for visualisation purposes: quadrupling sensor count to a grid of $32 \times 32$, only doubles CPU time, indicating good scaling behaviour of our implementation.
}
For the physical spread $\psi$, we consider
the “fast” (compactly supported, differentiable, third-order polynomial) spread of \cite[Example 3.4]{tuomov-pointsource}.
We do not consider the “cut Gaussian” spread, as this is not differentiable.\footnote{Our implementation \cite{tuomov-pointsource-codes} includes a dynamic estimation workaround to the lack of Lipschitz gradient estimates of the dual variables, but the performance of the sliding methods is unpredictable due to the unsatisfied assumption.}
We take the kernel $\rho$ for the particle-to-wave operator $\Wave=\rho*$ from the same example.
The standard deviations and other parameters of the spread are as in the numerical experiments of \cite{tuomov-pointsource}.
To generate the synthetic measurement data $b$, we apply $A$ to a ground-truth measure $\hat\mu$ with four spikes of distinct magnitudes.
For the “biased” problem \eqref{eq:numerical:problem:biased} we additionally add to this the ground-truth bias $\hat z$ formed by summing the weighted indicator functions of two intervals (1D) or Euclidean balls (2D). The details can be found in our implementation \cite{tuomov-pointsource-codes}.
Then we add to each sensor reading independent Gaussian noise.
The noise level (standard deviation and resulting SSNR) and corresponding regularisation parameters, found by trial-and-error, are listed in \cref{tab:numerical:experiment-params} for each experiment.

\def\loadexperiment#1{%
    \edef\filename{\datapath\csname#1DirName\endcsname/experiment.json}%
    \edef\stats{\datapath\csname#1DirName\endcsname/stats.json}%
    \expandafter\JSONParseFromFile\csname EXP#1\endcsname{\filename}%
    \expandafter\JSONParseFromFile\csname STATS#1\endcsname{\stats}%
}
\def\getssnr#1#2{%
    \def\round##1{\num[group-minimum-digits=2,round-mode=places]{##1}}%
    \round{\JSONParseExpandableValue{#1}{#2}}
}
\def\BasicExperiment#1{%
    \csname#1Name\endcsname
    & \expandafter\JSONParseValue\csname EXP#1\endcsname{noise_distr.std_dev}
    & \expandafter\getssnr\csname STATS#1\endcsname{ssnr}
    & \expandafter\JSONParseValue\csname EXP#1\endcsname{regularisation.NonnegRadon}
    &
}
\def\BiasedExperiment#1{%
    \csname#1Name\endcsname
    & \expandafter\JSONParseValue\csname EXP#1\endcsname{base.noise_distr.std_dev}
    & \expandafter\getssnr\csname STATS#1\endcsname{ssnr}
    & \expandafter\JSONParseValue\csname EXP#1\endcsname{base.regularisation.NonnegRadon}
    & \expandafter\JSONParseValue\csname EXP#1\endcsname{λ}
}

\begin{table}[t]
    \caption{Experiment noise levels and regularisation parameters.
        Empty fields are not applicable to the experiment in question.}
    \label{tab:numerical:experiment-params}
    \tikzifexternalizing{}{%
        \foreach \experiment in {OneDFast, TwoDFast, OneDFastBiased, TwoDFastBiased} {
            \loadexperiment{\experiment}
        }
        \centering
        \small
        \begin{tabular}{lrrrr}
            \toprule
            Experiment & std.dev.~ & SNR (dB) & $\alpha$ & $\lambda$
            \\
            \midrule
            \BasicExperiment{OneDFast}
            \\
            \BasicExperiment{TwoDFast}
            \\
            \BiasedExperiment{OneDFastBiased}
            \\
            \BiasedExperiment{TwoDFastBiased}
            \\
            \bottomrule
        \end{tabular}
    }
\end{table}

\def\alglist{{FWf, FB, PDPS, radonFB, sFBradon, sFB}}

\begin{figure}[t!]
    \plotsONEd{pointsource1d_fast}
    \caption{%
        Reconstructions and performance on 1D problem with “fast” spread.
        \textbf{Top:} reconstruction and original data.
        The measurement data magnitude scale is on the right, spike magnitude on the left.
        \textbf{Middle:} Function value in terms of iteration count (left) and CPU time (right).
        \textbf{Bottom:} Spike evolution, inner iteration count (left), and kernels (right).
        The thick lines indicate the spike count, and the thinner and dimmer lines the inner iteration count.
    }
    \label{fig:1dproblem_fast}
\end{figure}

\begin{figure}[t!]
    \plotsTWOd{pointsource2d_fast}
    \caption{%
        Reconstructions and performance on 2D problem with “fast” spread.
        \textbf{Top:} reconstruction and original data.
        The pixel intensities indicate the measurement data, while the area of the top surface of the boxes on top of each pixel is proportional the noise level. Their colour additionally indicates the sign of the noise.
        \textbf{Middle:} Function value in terms of iteration count and CPU time.
        \textbf{Bottom:} Spike and inner iteration count evolution, and kernels.
    }
    \label{fig:2dproblem_fast}
\end{figure}

\def\alglist{{fPDPSradon, sPDPSradon, fPDPS, sPDPS}}

\begin{figure}[t!]
    \plotsONEdTV{pointsource1d_tv_fast}
    \caption{%
        Reconstructions and performance on 1D problem with “fast” spread and TV-regularised bias.
        \textbf{Top:} reconstruction and original data.
        The measurement data magnitude scale is on the right, spike magnitude on the left.
        \textbf{Middle:} Function value in terms of iteration count (left) and CPU time (right).
        \textbf{Bottom:} Spike evolution, inner iteration count (left), and kernels (right).
        The thick lines indicate the spike count, and the thinner and dimmer lines the inner iteration count.
    }
    \label{fig:1dproblem_tv_fast}
\end{figure}

\begin{figure}[t!]
    \plotsTWOdTV{pointsource2d_tv_fast}
    \caption{%
        Reconstructions and performance on 2D problem with “fast” spread and TV-regularised bias.
        \textbf{Top:} reconstruction and original data. The upper layer displays the data, noise level, and spike reconstructions, while the lower layer displays the bias and its reconstruction by \sPDPSName.
        The pixel intensities indicate the measurement data, while the area of the top surface of the boxes on top of each pixel is proportional the noise level. Their colour additionally indicates the sign of the noise.
        \textbf{Middle:} Function value in terms of iteration count and CPU time.
        \textbf{Bottom:} Spike and inner iteration count evolution, and kernels.
    }
    \label{fig:2dproblem_tv_fast}
\end{figure}

\subsection{Results}
\label{sec:numerical:results}

We ran the experiments on a 2020 MacBook Air M1 with 16GB of memory.
We take advantage of the 4 high performance CPU cores of the 8-core machine by using 4 parallel computational threads to calculate $A^*z$ and for the branch-and-bound optimisation.
We report the performance of all the algorithms applicable to each experiment in the corresponding \cref{fig:1dproblem_fast,fig:2dproblem_fast,fig:1dproblem_tv_fast,fig:2dproblem_tv_fast}.
Each of the figures depicts the spread $\psi$, kernel $\rho$, and sensor $\theta$ involved in $A$ and $\Wave$. They also depict the noisy and noise-free data, the ground-truth measure $\hat\mu$, and the algorithmic reconstructions. For \eqref{eq:numerical:problem:biased} also the ground-truth bias $\hat z$ and its reconstruction by sPDPS is included.
The reconstructions or optimal solutions to the problems \cref{eq:fb:problem,eq:numerical:problem:biased} cannot be expected to equal $\hat\mu$ and $\hat z$ due to noise and ill-conditioning of the inverse problem $A\mu+z=b$.

Each of the figures plots against both the iteration count and the spent CPU time, the relative error of the function value,
\[
    e^k = \frac{v(x^k)-v(x_{\min})}{v(x^0)-v(x_{\min})},
\]
where $v$ is the objective of \eqref{eq:intro:problem} and $x^k=\mu^k$ (or \eqref{eq:numerical:problem:biased} and $x^k=(\mu^k, z^k)$), and $v(x_{\min})$ is the minimum objective value observed over all algorithms. (Since the relative error is zero for such an iterate, it will not be shown in the logarithmic plots.)
The plots are logarithmic on both axes, and we sample the reported values logarithmically only on iterations $1,2,\ldots,10,20,\ldots,100,200,\ldots$.
We limit the number of iterations to 4000. The CPU time is a sum over the time spent by each computational thread, so several times the clock time requirement.

The figures also indicate the spike count evolution, and the number of iterations needed to solve the finite-dimensional subproblems. The subproblem iteration counts are averages over the corresponding period.

\subsection{Comparison}

Studying \cref{fig:1dproblem_fast,fig:2dproblem_fast}, we notice the sliding $\Wave$-(semi)norm proximal term {\sFBName} as well as the Radon-norm proximal term {\sFBradonName} to exhibit significant performance advantages over the comparison methods. The $\Wave$-(semi)norm usually provides a slightly faster algorithm, as is to be expected from both being more tailored to the problem, as well as from the worse convergence guarantees that we are able to obtain.
Likewise, for the biased problem, \cref{fig:1dproblem_tv_fast,fig:2dproblem_tv_fast} indicate the sliding {\sPDPSName} to overperform the other methods variants, with the Radon-norm proximal term exhibiting noticeably slower convergence than the $\Wave$-norm.
After all, we have not been able to show the convergence of a primal-dual method based on this marginal energy.
We were only able to show (\cref{sec:extensions:convergence.fb}) that a forward-backwrd method works with auxiliary variables, when combined with the Radon norm marginal energy.
In \cite{dangvalkonen2026leak} we evalute the performance of such a method on a convection-diffusion problem.

\subsection{Conclusion}

Overall, the numerical results confirm our intuition that proximal terms based on metrisation of the weak-$*$ topology---in particular, conventional optimal transport distances, but also particle-to-wave or MMD norms--improve the performance of optimisation methods in measure spaces.
Lipschitz properties of the data term $F$, can, however, be more easily verified for proximal terms based on the Radon norm, as well as for conditional gradient methods. An advantage of our approach, compared to the latter, is its flexibility, allowing easily to treat more complex problems through primal-dual methods, product spaces, and nonconvex objectives.

\providecommand{\eprint}[1]{\href{http://arxiv.org/abs/#1}{arXiv:#1}}\providecommand{\noopsort}[1]{}\providecommand{\eprint}[1]{\href{http://arxiv.org/abs/#1}{arXiv:#1}}

\appendix

\section{Auxiliary lemmas}

We recall a fundamental three-point identity of Bregman divergences.

\begin{lemma}\label{lemma:bregman-three-point}
    Let $J: \Meas(\Omega) \to \extR$ be convex, proper, and weakly-$*$ lower semicontinuous.
    Take $\nu_0,\nu_1,\nu_2 \in \Meas(\Omega)$, as well as $\omega_0 \in \subdiff J(\nu_0)$ and $\omega_1 \in \subdiff J(\nu_1)$. Then
    \[
        B_J^{\omega_0}(\nu_0,\nu_2)-B_J^{\omega_0}(\nu_0,\nu_1)
        =
        B_J^{\omega_1}(\nu_1,\nu_2)+\dualprod{\omega_1-\omega_0}{\nu_2-\nu_1}.
    \]
\end{lemma}

\begin{proof}We have
    \[
        \begin{split}
        B_J^{\omega_0}(\nu_0,\nu_2)-B_J^{\omega_0}(\nu_0,\nu_1)
        &
        =J(\nu_2) - J(\nu_1) - \dualprod{\omega_0}{\nu_2-\nu_1}
        \\
        &
        =J(\nu_2)-J(\nu_1) - \dualprod{\omega_1}{\nu_2-\nu_1}
        +\dualprod{\omega_1-\omega_0}{\nu_2-\nu_1}
        \\
        &
        =B_J^{\omega_1}(\nu_1,\nu_2)+\dualprod{\omega_1-\omega_0}{\nu_2-\nu_1}.
        \qedhere
        \end{split}
    \]
\end{proof}

\end{document}